\documentclass[12pt, leqno]{amsart}

\usepackage{graphicx}
\usepackage{amssymb,amsmath,amsthm,amscd}
\usepackage{mathrsfs,mathdots}
\usepackage{enumerate,stmaryrd}
\usepackage{upgreek}
\usepackage[usenames,dvipsnames]{color}
\usepackage[colorlinks=true, pdfstartview=FitV,
 linkcolor=blue,citecolor=blue,urlcolor=blue]{hyperref}
\usepackage[all]{xy}

\usepackage{verbatim}
\usepackage{oldgerm}
\usepackage[normalem]{ulem}  
\usepackage{xspace}
\usepackage{xcolor,cancel}

\numberwithin{equation}{section}
\allowdisplaybreaks[4]

\newcommand{\nc}{\newcommand}

\nc{\op}{\operatorname}

\theoremstyle{plain}
\newtheorem{lemma}{Lemma}[section]
\newtheorem{proposition}[lemma]{Proposition}
\newtheorem{theorem}[lemma]{Theorem}

\theoremstyle{definition}
\newtheorem{remark}[lemma]{Remark}
\newtheorem{conjecture}[lemma]{Conjecture}
\newtheorem{example}[lemma]{Example}

\newtheorem{definition}[lemma]{Definition}
\newtheorem{corollary}[lemma]{Corollary}

\nc{\Prop}{\begin{proposition}}
\nc{\enprop}{\end{proposition}}
\nc{\Lemma}{\begin{lemma}}
\nc{\enlemma}{\end{lemma}}
\nc{\Cor}{\begin{corollary}}
\nc{\encor}{\end{corollary}}

\nc{\shc}{\mathcal{C}}
\newcommand{\jj}{ \mathfrak{j}}
\newcommand{\Q}{\mathbb{Q}}
\newcommand{\C}{\mathbb{C}}
\newcommand{\Ainf}{A_{\infty}}
\newcommand{\Seq}{\Sigma}
\newcommand{\T}{\mathcal{T}}
\newcommand{\dT}{\mathrm{T}}
\newcommand{\new}{\mathrm{new}}
\newcommand{\Ser}{\mathcal{S}}
\newcommand{\A}{\mathcal{A}}
\newcommand{\FQ}{\mathcal{Q}}
\newcommand{\FO}{\Upomega}
\newcommand{\Ca}{\mathscr{C}}
\nc{\F}{\mathcal{F}}
\newcommand{\D}{\mathrm{D}}

\newcommand{\M}{\mathrm{M}}
\newcommand{\W}{\mathrm{W}}
\newcommand{\B}{\mathbf{B}}

\newcommand{\Aa}{\mathbf{A}} 
\newcommand{\Uamg}{U_\Aa^-(\g)}
\newcommand{\Uapg}{U_\Aa^+(\g)}
\newcommand{\Aan}{A_\Aa(\n)}
\newcommand{\Aanw}{A_{\Aa}(\n(w))}

\newcommand{\Ah}{{\Z[q^{\pm 1/2}]}}
\newcommand{\Z}{\mathbb{\ms{1mu}Z}}
\newcommand{\seteq}{\mathbin{:=}}
\newcommand{\hd}{{\operatorname{hd}}}
\newcommand{\soc}{{\operatorname{soc}}}
\nc{\ov}[1]{\overline{#1}}
\nc{\Wlmj}[3]{\W_{#2,#3}^{(#1)}}
\nc{\Mkl}[2]{\M_\ttww(#1,#2)}
\nc{\rmat}[1]{{\mathbf r}_{\mspace{-2mu}\raisebox{-.5ex}{${\scriptstyle{#1}}$}}}
\newcommand{\on}{\operatorname}
\nc{\de}{\on{\textfrak{d}}}
\nc{\tL}{\widetilde{\Lambda}}
\nc{\tl}{\widetilde{\lambda}}
\nc{\Cquiver}{\upsigma}

\newcommand{\g}{{\mathfrak{g}}}
\newcommand{\h}{{\mathfrak{h}}}
\newcommand{\n}{{\mathfrak{n}}}
\newcommand{\isoto}[1][]{\mathop{\xrightarrow%
[{\raisebox{.3ex}[0ex][.3ex]{$\scriptstyle{#1}$}}]%
{{\raisebox{-.6ex}[0ex][-.6ex]{$\mspace{2mu}\sim\mspace{2mu}$}}}}}
\newcommand{\ri}{{\mspace{1mu}\rm r}}
\newcommand{\li}{{\mspace{1mu}\rm l}}
\newcommand{\id}{\on{id}}

\newcommand{\soplus}{\mathop{\mbox{\normalsize$\bigoplus$}}\limits}
\newcommand{\sodot}{\mathop{\mbox{\normalsize$\bigodot$}}\limits}
\newcommand{\starop}{\mathop{\mbox{\normalsize$\star$}}\limits}

\newcommand{\snconv}{\mbox{\scriptsize$\odot$}}

\newcommand{\ww}{ \textbf{\textit{w}}}
\newcommand{\ttww}{{\widetilde{\ww}} }

\nc{\nconv}{\mathop{\mbox{\large $\odot$}}}
\nc{\nnconv}{\mathop{\mbox{\large $\star$}}}

\nc{\lb}{\llbracket}
\nc{\rb}{\rrbracket}

\newcommand{\ko}{{{\mathbf{k}}}}
\nc{\la}{\lambda}
\nc{\La}{\Lambda}
\nc{\ve}{\varepsilon}
\nc{\ep}{\epsilon}
\nc{\vp}{\varphi}
\nc{\lan}{\langle}
\nc{\ran}{\rangle}
\nc{\Uqg}{U_q(\g)}
\nc{\Aqg}{A_q(\g)}
\nc{\Aqn}{A_q(\n)}
\nc{\al}{\alpha}
\nc{\be}{\beta}
\nc{\ga}{\gamma}
\nc{\wt}{\operatorname{wt}}
\nc{\wtl}{\wt_\li}
\nc{\wtr}{\wt_\ri}
\nc{\ch}{\operatorname{ch}}

\nc{\norm}{{\mathrm{norm}}}
\nc{\aff}{{\rm{aff}}}
\nc{\Maf}{M_\aff}
\nc{\ev}{{\mathrm{even}}}
\nc{\od}{{\mathrm{odd}}}
\nc{\Sev}{\Seq^{\ev}}
\nc{\Sod}{\Seq^{\od}}
\nc{\Spl}{\Seq^{+}}
\nc{\Smi}{\Seq^{-}}
\nc{\low}{{\mathrm{low}}}
\nc{\upper}{{\mathrm{up}}}
\nc{\one}{{\bf{1}}}
\nc{\To}[1][{\hspace{2ex}}]{\xrightarrow{\,#1\,}}
\nc{\te}{\tilde{e}}
\nc{\tw}{{\widetilde{w}}}
\nc{\tww}{\ww}
\nc{\tuu}{{\mathsf{u}}}
\nc{\tel}{\tilde{e}^\low}
\nc{\teu}{\tilde{e}^\upper}
\nc{\tf}{\tilde{f}}
\nc{\tfl}{\tilde{f}^\low}
\nc{\tfu}{\tilde{f}^\upper}
\nc{\tE}{\widetilde{E}}
\nc{\tF}{\widetilde{F}}
\nc{\tFF}{\widetilde{\F}}
\nc{\tB}{\widetilde{B}}
\nc{\tz}{\tilde{z}}
\nc{\tQ}{\hspace{-.2ex}\textbf{\textit{Q}}}
\nc{\tb}{\tilde{b}}
\nc{\Ft}{\F^\dT}
\nc{\Seed}{\mathscr{S}}

\nc{\cor}{\mathbf{k}}
\nc{\tens}{\mathop\otimes}
\nc{\gmod}{\mbox{-$\mathrm{gmod}$}}
\nc{\gMod}{\mbox{-$\mathrm{gMod}$}}

\nc{\proj}{\mbox{-$\mathrm{proj}$}}
\nc{\gproj}{\mbox{-$\mathrm{gproj}$}}
\nc{\smod}{\mbox{-$\mathrm{mod}$}}
\nc{\nmod}{\mbox{-$\mathrm{nilmod}$}}

\newcommand{\cmA}{\mathsf{A}}
\nc{\Rnorm}{R^{\rm{norm}}}
\nc{\Runiv}{R^{\rm{univ}}}
\nc{\Rren}{R^{\rm{ren}}} 
\nc{\col}{\colon}
\nc{\epiTo}[1][]{\longtwoheadrightarrow[{#1}]}
\nc{\epito}{\twoheadrightarrow}
\nc{\monoTo}[1][]{\xymatrix{\ar@{>->}[r]^-{{#1}}&}}
\nc{\sym}{\mathfrak{S}}
\nc{\rl}{\mathsf{Q}}
\nc{\Qq}{{\Q(q)}}
\nc{\wl}{\mathsf{P}}   
\nc{\Oint}{\mathcal{O}_{{\rm int}}}
\newcommand{\scbul}{{\,\raise1pt\hbox{$\scriptscriptstyle\bullet$}\,}}

\nc{\conv}{\mathop{\mathbin{\mbox{\large $\circ$}}}}
\newcommand{\hconv}{\mathbin{\scalebox{.9}{$\nabla$}}}
\newcommand{\sconv}{\mathbin{\scalebox{.9}{$\Delta$}}}

\newcommand{\Hom}{\operatorname{Hom}}

\renewcommand{\Im}{\op{Im}}

\nc{\K}{{K}}
\nc{\Kex}{{\K}^{\mathrm{ex}}}
\nc{\Uex}{\Uppsi_{\mathrm{ex}}}
\nc{\Kfr}{{\K}^{\mathrm{f\mspace{.01mu}r}}}
\nc{\cl}{{\rm{cl}}}
\nc{\ben}{\begin{enumerate}}
\nc{\ee}{\end{enumerate}}
\nc{\bnum}{\begin{enumerate}[{\rm(i)}]}
\nc{\bna}{\begin{enumerate}[{\rm(a)}]}
\nc{\bc}{\begin{cases}}
\nc{\ec}{\end{cases}}

\newenvironment{myequation}
{\relax\setlength{\arraycolsep}{1pt}\begin{eqnarray}}
{\end{eqnarray}}
\newenvironment{myequationn}
{\relax\setlength{\arraycolsep}{1pt}\begin{eqnarray*}}
{\end{eqnarray*}}

\nc{\eq}{\begin{myequation}}
\nc{\eneq}{\end{myequation}}
\nc{\eqn}{\begin{myequationn}}
\nc{\eneqn}{\end{myequationn}}
\nc{\hs}{\hspace*}
\nc{\ba}{\begin{array}}
\nc{\ea}{\end{array}}
\nc{\noi}{\noindent}
\nc{\ang}[1]{\langle{#1}\rangle}
\nc{\fr}{{\mathrm{fr}}}
\nc{\qt}[1]{\quad\text{#1}}
\nc{\ol}{\overline}
\nc{\true}{\delta}
\nc{\ms}{\mspace}
\renewcommand{\mod}{\ms{3mu}\mathbin{\mathrm{mod}}\ms{1mu}}
\nc{\vs}{\vspace*}
\nc{\bl}{\bigl(}
\nc{\br}{\bigr)}
\nc{\bep}{\ol{\ep}}
\nc{\bal}{\,\ol{\al}}
\nc{\qtq}[1][{and}]{\quad\text{#1}\quad}
\nc{\set}[2]{\left\{{#1}\mid{#2}\right\}}
\nc{\ro}{{\rm(}}
\nc{\rf}{{\rm)}\xspace}
\nc{\Proof}{\begin{proof}}
\nc{\QED}{\end{proof}}
\nc{\monoto}[1][]{\xymatrix{\ar@{>->}[r]^-{{#1}}&}}
\nc{\etens}{\boxtimes}
\nc{\height}[1]{\vert{#1}\vert}
\nc{\dsum}{\mathop\sum\limits}
\nc{\Lrev}{L^{{\mathrm rev}}}
\nc{\rev}{\mathrm{rev}}
\nc{\ake}[1][2ex]{\rule[-1ex]{0ex}{#1}}

\newenvironment{magem}{\relax\color{magenta}}{\relax}

\newcommand{\bema}{\begin{magem}}
\newcommand{\ema}{\end{magem}}

\newlength{\mylength}
\title[Cluster algebra structures over quantum affine algebras]{Cluster algebra structures on module categories over quantum affine algebras}

\author[M. Kashiwara]{Masaki Kashiwara}
\thanks{The research of M.\ Kashiwara\
was supported by Grant-in-Aid for Scientific Research (B)
15H03608, Japan Society for the Promotion of Science.}
\address[M. Kashiwara]{
Kyoto University Institute for Advanced Study, Kyoto 606-8501, Japan,
Research Institute for Mathematical Sciences, Kyoto University,
Kyoto 606-8502, Japan \& Korea Institute for Advanced Study, Seoul 02455, Korea }
\email[M. Kashiwara]{masaki@kurims.kyoto-u.ac.jp}

\author[M. Kim]{Myungho Kim}
\address[M. Kim]{Department of Mathematics, Kyung Hee University, Seoul 02447, Korea}
\email[M. Kim]{mkim@khu.ac.kr}
\thanks{The research of M.\ Kim was supported by the National Research Foundation of
Korea(NRF) Grant funded by the Korea government(MSIP) (NRF-2017R1C1B2007824).}

\author[S.-j. Oh]{Se-jin Oh}
\thanks{ The research of S.-j.\ Oh was supported by the National Research Foundation of
Korea(NRF) Grant funded by the Korea government(MSIP) (NRF-2016R1C1B2013135).}
\address[S.-j. Oh]{Department of Mathematics, Ewha Womans University, Seoul 120-750, Korea}
\email[S.-j. Oh]{sejin092@gmail.com}

\author[E. Park]{Euiyong Park}
\thanks{The research of E.\ P.\ was supported by the National Research Foundation of Korea(NRF) Grant funded by the Korea Government(MSIP)(NRF-2017R1A1A1A05001058).}

\address[E. Park]{Department of Mathematics, University of Seoul, Seoul 02504, Korea}
\email[E. Park]{epark@uos.ac.kr}

\keywords{Quantum affine algebra, Quiver Hecke algebra, Quantum group, Quantum cluster algebra}

\subjclass[2010]{81R50, 16F60, 16G, 16T,17B37}

\date{April 2, 2019}

\begin{document}

\begin{abstract}
We study monoidal categorifications of certain monoidal subcategories $\Ca_J$ of finite-dimensional modules over quantum affine algebras, whose
cluster algebra structures coincide and arise from the category of finite-dimensional modules over quiver Hecke algebra of type $A_\infty$. In particular, when the quantum affine algebra is of type
$A$ or $B$, the subcategory coincides with the monoidal category $\Ca_\g^0$ introduced by Hernandez-Leclerc. As a consequence, the modules corresponding to
cluster monomials are real simple modules over quantum affine algebras.
\end{abstract}

\maketitle
\tableofcontents

\section*{Introduction}

The \emph{quiver Hecke algebras} (or \emph{Khovanov-Lauda-Rouquier algebras}),
introduced  independently
by Khovanov-Lauda (\cite{KL09, KL11}) and Rouquier (\cite{R08}), provide the categorification of a half of quantum group $U_q(\g)$.
Since then, quiver Hecke algebras have been studied actively and various new features were discovered in  the viewpoint of categorification.
Studying quantum groups via the quiver Hecke algebras has become one of the main research themes on quantum groups in aspect of categorification.
 In particular, the studies on representations of a quantum affine algebra $U_q'(\g)$ via \emph{the generalized quantum Schur-Weyl duality functors} (\cite{KKK18A}) and the quantum cluster algebra structures of quantum unipotent coordinate algebras $A_q(\n(w))$ via \emph{the monoidal categorifications} (\cite{KKKO18})  drew big attention of researchers on various areas.

The generalized quantum Schur-Weyl duality functor was developed in \cite{KKK18A}, which is a vast generalization of quantum affine Schur-Weyl duality.
Let $U_q'(\g)$ be a quantum affine algebra of \emph{arbitrary type} over a base field $\cor$, and let $\{ V_j \}_{j\in J}$ be a family of \emph{quasi-good} $U_q'(\g)$-modules.
The generalized quantum Schur-Weyl duality provides
a procedure to make a \emph{symmetric} quiver Hecke algebra $R^J$
 from the R-matrices among $\{ V_j \}_{j\in J}$ and
to construct a monoidal functor $\F\col  R^J \gmod \to\Ca_\g$ enjoying good properties.
Here  $R^J \gmod$ (resp.\ $ \Ca_\g$) denotes the monoidal category of graded $R^J$-modules (resp.\ integrable $U_q'(\g)$-modules) which are finite-dimensional over $\cor$.

Let us recall briefly the results of \cite{KKK18A}.
Let $N \in \Z_{>1}$ and $J=\Z$,  and consider the family $\{ V(\varpi_1)_{q^{2k}} \}_{k\in \Z}$ of the quantum affine algebra $U_q'(\g)$ of type $A_{N-1}^{(1)}$.
Then the corresponding
quiver Hecke algebra $R^J$ is of type $A_\infty$ and
the generalized quantum Schur-Weyl duality
associated with
$\{V(\varpi_1)_{q^{2k}} \}_{k\in \Z}$
gives a monoidal functor $\F\col  R^J \gmod \to \Ca_J \subset \Ca^0_\g$.
Here $\Ca_{\g}^0$  is the smallest Serre subcategory of $\Ca_{\g}$ which is stable by taking tensor products and contains a sufficiently large family of fundamental representations, and $\Ca_J$ is a certain subcategory of $\Ca^0_\g$ determined by the functor $\F$ (see Section \ref{sec:HL} and \ref{subsec A_inf}).
Let $\Ser_N$ be the smallest Serre subcategory of
$\A\seteq R^J\gmod$ such that
\begin{enumerate}[{\rm (i)}]
\item $\Ser_N$ contains $L[a,a+N]$ for any $a \in J$,
\item $X \conv Y, \ Y \conv X \in \Ser_N$ for all $X \in \A$ and $Y \in \Ser_N$,
\end{enumerate}
where $L[a, a+N]$ is the 1-dimensional $R^J(\epsilon_{a} - \epsilon_{a+N+1})$-module (see Section \ref{subsec: Ainfty} for the definition of $L[a,b]$).
Then they constructed the monoidal category $\T_N$
with the modified convolution product $\star$
 by localizing the quotient category $\A/\Ser_N$ at the commuting family coming from the objects  $L_a = L[a,a+N-1]$'s.
Moreover they showed that $\T_N$ is rigid,
the functor $\F \col \A \to \Ca_J$
factors through the canonical functor $ \FO_N\col \A \to \T_N$, and
the resulting functor $\T_N\to \Ca_J$
induces an isomorphism between the Grothendieck rings of $\Ca_J$
and $\T_N |_{q=1}$. The category $\T_N$ also plays the same role in
the studies on generalized quantum Schur-Weyl duality for
$\Ca_{A_{N-1}^{(2)}}^0$ (\cite{KKKO15C}) and $\Ca_{B_{N/2}^{(1)}}^0$ (for $N \in 2\Z$) (\cite{KKO18}).

The monoidal categorification of a quantum unipotent coordinate algebra $A_q(\n(w))$ using a certain monoidal subcategory  $\shc_w$ (see Definition \ref{def:cw}) of $R\gmod$ for \emph{symmetric}
quiver Hecke algebras was provided in \cite{KKKO18}, which gives a monoidal  categorical explanation on the quantum cluster algebra structure of $A_q(\n(w))$ given in Gei\ss--Leclerc--Schr\"oer \cite{GLS13}.
For each reduced expression $\tw$ of a Weyl group element $w$, the initial quantum monoidal seed in $\shc_w$ is given
by the \emph{determinantial modules} $ \M_\tw(s,0) $ which correspond to certain unipotent quantum minors of $A_q(\n(w))$. It was shown in
\cite{KKKO18} that  every cluster monomials is a member of the upper global basis of $A_q(\n(w))$ by using the monomial categorification of $A_q(\n(w))$  arising from $\shc_w$.

The cluster algebra structures also appear in certain monoidal
subcategories of the category $\Ca_\g$ of finite-dimensional
integrable representations of a quantum affine algebra $U_q'(\g)$.
The Grothendieck ring of the subcategory $\Ca_\ell$ ($\ell \in
\Z_{>0} $) of $\Ca_\g$ introduced in Hernandez--Leclerc \cite[Section 3.8]{HL10} was studied by using cluster algebra
structures (\cite{HL10}), and an algorithm for calculating
$q$-character of Kirillov-Reshetikhin modules for any untwisted
quantum affine algebras was described in \cite{HL16} by studying the
cluster algebra $\mathscr{A}$ which is isomorphic to the
Grothendieck ring of the subcategory $\Ca^-_\g$ of $\Ca_\g$
(see Section \ref{subsec: HL cluster -}). It was conjectured in
\cite{HL16} that all cluster monomials of $\mathscr{A}$ correspond
to the classes of certain simple objects of $\Ca^-_\g$ (see
Conjecture \ref{conj: HL}).

In a viewpoint of the categorification using quiver Hecke algebras, it is natural to study a cluster algebra structure on monoidal subcategories of quantum affine algebras using the monoidal categorification of quantum unipotent coordinate algebras via the generalized quantum Schur-Weyl duality.
In fact, it is what we perform in this paper.
We study a quantum cluster algebra structure of the category
$\A\seteq R^J\gmod$ of type $A_\infty$
and show that this quantum cluster algebra structure  is compatible with
the functor $\A\to\T_N$
for any $N\in \Z_{>1}$ (see Theorem \ref{th:monoidalTN}).

We further provide the condition $\eqref{eq: general assumption}$  on quasi-good modules of a quantum affine algebra $U_q'(\g)$ to make the generalized Schur-Weyl duality functor $\F\col \A  \to \Ca_J$
factor through the canonical functor $ \FO_N\col\A \to \T_N$
so that we have the exact monoidal functor $\tFF \col \T_N\to \Ca_J$ by  using the framework given in \cite{KKO18} (see Theorem \ref{thm: Factoring}).
In the cases of types $A_{N-1}^{(t)}$ ($t=1,2$), $B_{N/2}^{(1)}$ (for $N\in 2\Z$), $C_{N-1}^{(1)}$ (for $N>3$), $D_{N}^{(t)}$ (for ($N\ge 4$ and $t=1,2$) or ($ N=4$ and $t=3$)),
 we give explicit quasi-good modules satisfying the condition, which provide the cluster algebra structure on the Grothendieck ring $K(\Ca_J)$ induced from  $\T_N$.
 Moreover, all cluster monomials of $K(\Ca_J)$ correspond to the classes of real simple modules in $\Ca_J$ (Theorem \ref{Thm: main}).
 Note that the families for the types $A$ and $B$ appeared in \cite{KKK18A, KKKO15C}, and \cite{KKO18}, but the ones for types $C$ and $D$ are new.
 Since $\Ca_J = \Ca_\g^0$ in types $A$ and $B$, there exists a cluster algebra structure on $K(\Ca_\g^0)$ coming from
the quantum cluster algebra structure of $\T_N$.
In particular, we prove that the conjecture given in \cite{HL16} (see Theorem \ref{thm: HL}) is true when $\g$ is of type $A_{N-1}^{(1)}$:
all cluster monomials of $\mathscr{A}$ correspond to the classes of certain real simple objects of $\Ca_{A_{N-1}^{(1)}}^-$. Remark that for the case of the categories $\Ca_\ell$, the corresponding  property was  proved in \cite{Qin17}.

\medskip
Let us explain our results more precisely. Let $J=\Z$ and let $A^J$ be the Cartan matrix of type $A_\infty$. Let $R^J$ be the
\emph{symmetric} quiver Hecke algebra of type $A_\infty$. We first choose a special infinite sequence $\ttww$ of simple reflections
$s_j$ $(j\in J)$ of the Weyl group $W_J$ of type $A_\infty$ (see $\eqref{eq: def tw}$). It is shown in Proposition \ref{prop:
reduced} that every $p$-prefix $\ttww_{\le p}$ of $\ttww$ is a reduced expression in $W_J$ for each $p\in \Z_{\ge1}$ and, if $p =
t(2t+1) $ for $t\in \frac{1}{2}\Z_{\ge1}$, then $\ttww_{\le p}$ can be viewed as a reduced expression of the longest element of the
parabolic subgroup of $W_J$ generated by $s_j$  for $j\in [-\lfloor t-\frac{1}{2} \rfloor, \lfloor t \rfloor]$ (see Remark~\ref{rem:
longest}). In this sense, the category $R^J\gmod$ can be understood as a limit of the subcategories $\shc_{\ttww_{\le p}}$. We then
consider the determinantial modules $ \M_\ttww(p,0)$ and the cuspidal modules $ \M_\ttww(p^+,p)$ associated with $\ttww$. Let
$c\col \Z_{\ge1} \rightarrow \Z_{\ge1}\times \Z_{\ge1}$ be the bijection given in Definition \ref{def: coordinate} and let $L[a,b]$
be the 1-dimensional $R$-module for $a,b\in \Z$ defined in Section \ref{subsec: Ainfty}. For $p \in \Z_{\ge 1}$ with $c(p) = (\ell,m)$,
we prove that $\M_\ttww(p,0)$ is isomorphic to the head $\W^{(\ell)}_{m,\jj_p} $ of a certain convolution product of $m$-many
$R$-modules $L[a,b]$ with the length $\ell=|b-a|+1$  (Proposition \ref{thm :KR in KLR}), and describe also the cuspidal
module $ \M_\ttww(p^+,p)$  in terms of $L[a,b]$ (Corollary \ref{cor: cuspidal}). We next describe the quiver $\tQ$ associated with the initial seed given by
$\ttww$, which can be viewed as the square product of the bipartite Dynkin quiver $Q_{A_\infty}$ (see
$\eqref{eq: Q with W}$). Theorem \ref{Thm: monoidal cat} tells that the category $ R^J\gmod$ is a monoidal categorification of the
quantum cluster algebras  $\mathscr{A}_{q^{1/2}}(\infty)$ with the quantum seed $ [\Seed_\infty]$ arising from the determinantial
modules $ \M_\ttww(p,0)$ and the quiver $\tQ$.

We next consider the monoidal categorification structure of $R^J\gmod$ which we constructed.
In Proposition \ref{prop: first mutation}, we describe the mutation $ \M_\ttww(p,0)'$ of $\M_\ttww(p,0)$ as follows:
$$
\M_\ttww(p,0)' \simeq \begin{cases}
\W^{(\ell)}_{m,\jj_p+1 } & \text{ if } \ell+m \equiv 0 \mod 2, \\
\W^{(\ell)}_{m,\jj_p-1} & \text{ if } \ell+m \equiv 1 \mod 2.
\end{cases}
$$
We then find sequences of mutations $\mu_{\Spl}$ and $\mu_{\Smi}$ such that
$\mu_{\Spl}(\tQ)$ and $\mu_{\Smi}(\tQ)$ are the same as $ \tQ$ as quivers and
$$
\text{$\mu_{\Spl}(\M_\ttww(p,0)) \simeq \W^{(\ell)}_{m,\jj_p+1}$ and $\mu_{\Smi}(\M_\ttww(p,0)) \simeq \W^{(\ell)}_{m,\jj_p-1}$}
$$
for $p \in \K$  with $c(p)=(\ell,m)$ (Proposition \ref{pro: mu S+ S-}).
In this viewpoint, since $ \M_\ttww(p,0)$ is isomorphic to $ \W^{(\ell)}_{m,\jj_p } $,
applying $\mu_{\Spl}$ (resp.\  $\mu_{\Smi}$) to $\tQ$ is understood as shifting the indices $\jj_p$ of $ \W^{(\ell)}_{m,\jj_p } $ at all vertices of $\tQ$ by $1$ (resp.\ $-1$).
For each $N \in \Z_{\ge1}$, we define the subquiver $\tQ_N$ of $\tQ$ as in $\eqref{eq: the procedure}$ and investigate compatibility with  $\mu_{\Spl}$, $\mu_{\Smi}$ and the procedure from $\A$ to $\T_N$.
Then we can conclude that the category $\T_N$ is a monoidal categorification of the quantum cluster algebras  $\mathscr{A}_{q^{1/2}}(N)$ with the quantum seed
$ [\Seed_N]$ arising from $ \{ \FO_N(\M_{\ttww}(p,0)) \}_{c(p) \in \Kex_N }$  and the quiver $\tQ_N$ (Theorem~\ref{th:monoidalTN}).

For the construction of generalized quantum Schur-Weyl duality
functor $\T_N\to \Ca_\g$, we provide the condition $\eqref{eq: general assumption}$  on quasi-good modules of the quantum affine algebra $U_q'(\g)$.
Let $\{ V_a \}_{a\in J}$ be a family of quasi-good modules in $\Ca_\g$ satisfying the condition $\eqref{eq: general assumption}$.
Note that, to define Schur-Weyl duality functor, we have to choose duality coefficients $ P_{i,j}(u,v)$ ($i,j\in J$) which are elements in $\cor[[u,v]]$ satisfying certain conditions determined by $\{ V_a \}_{a\in J}$.
In Lemma \ref{lem: N> to 0}, we prove that any Schur-Weyl duality functor arising from $\{ V_a \}_{a\in J}$ sends $L[a,b]$ to $0$ for any $a,b\in \Z$ with $b-a \ge N$, which implies that  it factors through  $\A \rightarrow \A/ \Ser_N$.
We prove that there exists a suitable duality coefficients $\{ P_{i,j}(u,v)\}_{i,j\in J}$ such that the corresponding Schur-Weyl duality functor $\F$ factors through the canonical functor $ \FO_N\col \A \to \T_N$ by following the framework given in \cite{KKO18}.
Theorem \ref{thm: Factoring} gives an exact monoidal functor $\tFF \col \T_N\to \Ca_J$ such that the following diagram quasi-commutes$\colon$
$$
 \xymatrix{
\A\ar[rr]^-{\FO_N} \ar[drr]_-{\F}&   &\T_N\ar[d]^-{\tFF}\\
&&\Ca_J\,.}
$$
Then, the functor $\tFF$ induces an isomorphism $ K(\T_N)\vert_{q=1}
\isoto K(\Ca_J)$ and the category $\Ca_J$ gives a monoidal
categorification of the cluster algebra $\mathscr{A}(N) \seteq
\mathscr{A}_{q^{1/2}}(N)|_{q=1}$ via the functor $\tFF$. Therefore,
every cluster monomial in $\mathscr{A}(N)$ corresponds to the
isomorphism class of a real simple object in $\Ca_J$ (Theorem
\ref{Thm: main}). For types $A_{N-1}^{(t)}$ ($t=1,2$),
$B_{N/2}^{(1)}$ (for $N\in 2\Z$), $C_{N-1}^{(1)}$ (for $N>3$),
$D_{N}^{(t)}$ (for ($N\ge 4$ and $t=1,2$) or ($ N=4$ and
$t=3$)),
 we give explicit quasi-good modules satisfying the condition $\eqref{eq: general assumption}$ (see Section \ref{Sec: Ca_J}).
For types $A$ and $B$,  $\Ca_J$ is equal to $\Ca_{\g}^0$, which
means that there exists a cluster algebra structure on $K(\Ca_\g^0)$
coming from the quantum cluster algebra structure of $\T_N$. For the
other types $C$ and $D$, $\Ca_J$ is a proper subcategory of
$\Ca_{\g}^0$. As for type $A_{N-1}^{(1)}$, we obtain an additional result using our monoidal categorification. In
\cite{HL16}, for an untwisted quantum affine algebras, Hernandez and
Leclerc studied the cluster algebra $\mathscr{A}$ with initial
quiver $G^-$ of infinite rank, which is isomorphic to the
Grothendieck ring of a \emph{half} $\Ca^-_\g$ of $\Ca^0_\g$. It was
conjectured in \cite{HL16} that all cluster monomials of
$\mathscr{A}$ correspond to the classes of certain simple objects of
$\Ca^-_\g$. When $A$ is of type $A_{N-1}^{(1)}$, we prove that  this
conjecture is true using our monoidal categorification $ \Ca_{\g}^0$
(see Theorem \ref{thm: HL}).

This paper is organized as follows. In Section \ref{Sec: quantum group}, we review quantum groups and quantum affine algebras.
In Section  \ref{Sec: cluster}, we recall quantum cluster algebras and monoidal categorification.
In Section \ref{Sec: quiver Hecke}, we review quiver Hecke algebras for monoidal categorification and generalized quantum Schur-Weyl duality.
In Section \ref{Sec: Module},  we study the monoidal categorification of $R^J\gmod$ associated with the infinite sequence $\ttww$.
In Section \ref{Sec: T_N}, we investigate the cluster algebra structure of $R^J\gmod$ and prove that $\T_N$ has a monoidal categorification structure induced from $R^J\gmod$  via $ \FO_N\col R^J\gmod \to \T_N $.
In Section \ref{Sec: Main}, we provide the condition  on quasi-good $U_q'(\g)$-modules to make
the generalized Schur-Weyl duality functor $\F\col \A  \to \Ca_J$ factor through the functor $ \FO_N\col \A \to \T_N$, and give an explicit quasi-good modules satisfying the condition for various types.

\medskip
{\bf Acknowledgments}

The second, third and fourth authors gratefully acknowledge for the hospitality of RIMS (Kyoto University) during their visits in 2018 and 2019.

\section{Quantum groups, quantum coordinate rings and quantum affine algebras} \label{Sec: quantum group}
In this section, we shortly recall the basic materials on quantum groups, quantum coordinate rings and quantum affine algebras. We refer to~\cite{AK,Kash91,Kas93,Kas02,KKKO18} for details.

For simplicity, we use the following convention:
\[
\text{For a statement $P$, $\true(P)$ is $1$ if $P$ is true and $0$ if $P$ is false.}
\]

\subsection{Quantum groups}  Let $I$ be an index set. A Cartan datum is a quintuple $(\cmA,\wl,\Pi,\wl^\vee,\Pi^\vee)$ consisting of
(i) a symmetrizable generalized Cartan matrix $\cmA =(a_{i,j})_{i,j \in I}$,
(ii) a free abelian group $\wl$, called the {\it weight lattice}, (iii) $\Pi= \{\alpha_i  \in \wl \ | \ i \in I \}$, called the set of {\it simple roots},
(iv)  $\wl^\vee \seteq \Hom_\Z(\wl,\Z)$, called the {\it co-weight lattice}, (v) $\Pi^\vee =\{ h_i \ | \ i \in I \} \subset \wl^\vee$, called the set of {\it simple coroots}, satisfying %
(1) $\lan h_i,\al_j \ran = a_{i,j}$ for all $i,j\in I$,
(2) $\Pi $ is linearly independent,
(3) for each $i \in I$ there exists a $\Lambda_i \in \wl$ such that $\lan h_j,\Lambda_i \ran = \delta_{i,j}$ for all $j \in I$.
We call $\Lambda_i$ the {\it fundamental weights}. Note that there exists a diagonal matrix $\mathsf{D}={\rm diag}(\mathsf{d}_i\ | \ i \in I)$ such that
$\mathsf{d}_i\in\Z_{>0}$ and $\mathsf{D}\cmA$ is symmetric.

We denote by $\rl \seteq \bigoplus_{i \in I} \Z\alpha_i$ the {\it root lattice}, $\rl^+ \seteq \bigoplus_{i \in I} \Z_{\ge 0}\alpha_i$ the {\it positive root lattice}
and $\rl^- \seteq \bigoplus_{i \in I} \Z_{\le 0}\alpha_i$ the {\it negative root lattice}. For $\beta = \sum_{i \in I} m_i \al_i \in \rl^+$, we set $|\be|=\sum_{i \in I} m_i$.

Set $\h = \mathbb{Q} \tens_\Z \wl^\vee$. Then there exists a symmetric bilinear form $( \ , \ )$  on $\h^*$ such that
$$(\al_i,\al_j)=\mathsf{d}_ia_{i,j} \quad (i,j \in I) \  \text{ and } \ \lan h_i,\la \ran = \dfrac{2(\al_i,\la)}{(\al_i,\al_i)} \ \ \text{for any $\la \in \h^*$ and $i \in I$}. $$

Let $\g$ be the Kac-Moody algebra associated with a Cartan datum $(\cmA,\wl,\Pi,\wl^\vee,\Pi^\vee)$.
The {\it Weyl group} $W$ of $\g$ is the subgroup of ${\rm GL}(\h^*)$ generated by the set of reflections $S = \{ s_i \ | \  i \in I \}$,
where $s_i(\lambda) \seteq \lambda - \lan h_i,\lambda \ran \al_i$ for $\lambda \in \h^*$.

Let $q$ be an indeterminate and set $q_i \seteq q^{\mathsf{d}_i}$ for each $i \in I$. We denote by $U_q(\g)$ the {\it quantum group} associated to $\g$,
which is a $\mathbb{Q}(q)$-algebra generated by $e_i$, $f_i$ $(i \in I)$ and $q^{h}$ $(h \in\wl^\vee)$. We set $U_q^+(\g)$ (resp.\ $U_q^-(\g)$) the subalgebra of $U_q(\g)$ generated by $e_i$'s (resp.\,$f_i$'s).

Recall that $\Uqg$ admits the {\it weight space decomposition} $ \Uqg = \soplus_{\be \in \rl} \Uqg_\be$,
where $\Uqg_\be  \seteq \{ x \in \Uqg \ | \ q^hxq^{-h} =q^{\lan h,\be \ran} \text{ for any } h \in \wl^\vee \}$. For $x \in \Uqg_\be$, we set $\wt(x)=\be$.

Set $\Aa=\Z[q^{\pm 1}]$. Let us  denote by $\Uamg$ the $\Aa$-subalgebra of $U_q(\g)$ generated by $f_i^{(n)} \seteq f_i^n/[n]_i!$, and
by $\Uapg$ the $\Aa$-subalgebra of $U_q(\g)$ generated by $e_i^{(n)} \seteq e_i^n/[n]_i!$ ($i\in I$, $n \in \Z_{\ge 0}$), where $[n]_i =\dfrac{ q^n_{i} - q^{-n}_{i} }{ q_{i} - q^{-1}_{i} }$ and $[n]_i! =\displaystyle \prod^{n}_{k=1} [k]_i$.

There is a $\Q(q)$-algebra anti-automorphism $\varphi$ of $U_q(\g)$ given as follows$\colon$
\begin{align*}
&\varphi(e_i)=f_i, \quad \varphi(f_i)=e_i, \quad \varphi(q^h)=q^h.
\end{align*}

We say that a $U_q(\g)$-module $M$ is called {\it integrable} if the actions of $e_i$ and $f_i$  on $M$ are locally nilpotent for all $i \in I$.
We denote by $\Oint(\g)$ the category of integrable left $U_q(\g)$-module $M$ satisfying
\begin{enumerate} [{\rm (i)}]
\item $M = \bigoplus_{\eta \in \wl} M_\eta$ where $M_\eta =\{ m \in M \ | \  q^hm =q^{\lan h,\eta \ran}m \}$,  $\dim M_\eta < \infty$,
\item there exist finitely many weights $\lambda_1,\ldots,\lambda_m$ such that $\wt(M) \subset \cup_{j}(\lambda_j+\rl^-)$, where $\wt(M)=\{ \eta \in \wl \ | \ \dim M_\eta \ne 0 \}$.
\end{enumerate}
It is well-known that $\Oint(\g)$ is a semisimple category
such that every  simple object is isomorphic to an irreducible highest weight module $V(\La)$ with the highest weight vector $u_\La$ of highest weight $\La$.
Here $\La$ is an element of the set $\wl^+$ of {\it dominant integral weights}$\colon$
$$\wl^+ \seteq \{\mu \in \wl \ | \ \lan h_i, \mu \ran \ge 0 \ \text{for all} \ i \in I \}.$$

\subsection{Unipotent quantum coordinate rings and unipotent quantum minors}
Note that $U_q^+(\g) \tens U_q^+(\g)$ has an algebra structure defined by
$$(x_1 \tens x_2) \cdot (y_1 \tens y_2) = q^{-(\wt(x_2),\wt(y_1))}(x_1y_1 \tens x_2y_2),$$
for homogeneous elements $x_1$, $x_2$, $y_1$ and $y_2$.
Then we have the algebra homomorphism $\Delta_\n \colon U_q^+(\g) \to U_q^+(\g) \tens U_q^+(\g)$ given by
$$ \Delta_\n(e_i) = e_i \tens 1 + 1 \tens e_i.$$

\begin{definition} We define the {\it unipotent quantum coordinate ring} $A_q(\n)$ as follows$\colon$
$$A_q(\n) = \soplus_{\be \in \rl^-} \Aqn_\be  \qquad \text{ where $\Aqn_\be \seteq \Hom_{\Q(q)}(U_q^+(\g)_{-\be},\Q(q))$}.$$
\end{definition}

Note that $\Aqn$ has a ring structure given as follows $\colon$
$$(\psi \cdot \theta)(x) \seteq \theta(x_{(1)})\psi(x_{(2)})  \qquad \text{ where } \Delta_\n(x)=x_{(1)} \otimes x_{(2)}.$$
The $\Aa$-form of $\Aqn$ is defined by $\Aan \seteq \{ \psi \in \Aqn \ | \ \lan \psi, \Uapg \ran \subset \Aa \}$.

Recall that the algebra $\Aqn$ has the {\it upper global basis} (\cite{Kas93})
$$ \B^\upper(\Aqn) \seteq \{ G^\upper(b) \ | \  b \in B(\Aqn) \},$$
where $B(\Aqn)$ denotes the {\it crystal} of $\Aqn$.

\medskip

Let $( \ , \ )_\La$ be the non-degenerate symmetric bilinear form on $V(\La)$ such that $(u_\La,u_\La)_\La=1$ and $(xu,v)_\La=(u,\varphi(x)v)_\La$ for $u,v \in V(\La)$ and $x\in U_q(\g)$.

For $\mu, \zeta \in W \La$, the \emph{unipotent quantum minor} $D(\mu, \zeta)$ is an element in $A_q(\n)$ given by
$$D(\mu,\zeta)(x)=(x u_\mu,u_\zeta)_\La$$
for $x\in U_q^+(\g)$, where $u_\mu$ and $u_{\zeta}$ are the {\it extremal weight vectors} in $V(\La)$ of weight $\mu$ and $\zeta$, respectively.

\begin{lemma}[{\cite[Lemma 9.1.1]{KKKO18}}] $\D(\eta,\zeta)$ is either contained in $\B^\upper(\Aqn)$ or zero.
\end{lemma}

For $\eta,\zeta\in W\La$, we write  $\eta \preceq \zeta$
if there exists a sequence $\{ \be_k \}_{1 \le k \le l}$ of positive real
roots such that
we have $(\be_k,\la_{k-1}) \ge 0$, where $\la_0 \seteq\zeta$ and $\la_k \seteq s_{\be_k}\la_{k-1}$ for $1 \le k \le l$. Note that $\eta \preceq \zeta$ implies
$\eta -\zeta \in \rl^-$.

By \cite[Lemma 9.1.4]{KKKO18}, $\D(\eta,\zeta)\ne 0$ if and only if $\eta \preceq \zeta$, for $\La \in \wl^+$ and $\eta,\zeta \in W\La$.
The behaviors of multiplications among unipotent quantum minors were investigated intensively (see \cite{BZ93,GLS13,KKKO18}):

\begin{proposition} \label{prop: multi D} Let $\la,\mu \in \wl^+$.
\begin{enumerate}
\item[{\rm (i)}] For $u,v \in W$ such that $u \ge v$, we have
$$  \D(u\la,v\la)\D(u\mu,v\mu) = q^{-(v\la,v\mu-u\mu)}\D(u(\la+\mu),v(\la+\mu)).$$
\item[{\rm (ii)}] For $s,t,s',t' \in W$ satisfying
\begin{itemize}
\item $\ell(s's)=\ell(s')+\ell(s)$ and $\ell(t't)=\ell(t')+\ell(t)$,
\item $s's\la \preceq t'\la$ and $s'\mu \preceq t't\mu$,
\end{itemize}
we have
\begin{align*}
& \D(s's\lambda,t'\lambda)\D(s'\mu,t't\mu) = q^{(s's\lambda+t'\lambda,\;s'\mu-t't\mu)} \D(s'\mu,t't\mu)\D(s's\lambda,t'\lambda).
\end{align*}
\end{enumerate}
\end{proposition}

\subsection{The subalgebra $A_q(\n(w))$ of $A_q(\n)$}
\label{sec:Anw}
In this subsection, we assume that the generalized Cartan matrix $\cmA$ is symmetric.

Let $\ww$ be a sequence of the set $S\seteq\{ s_i \ | \ i \in I\}$ 
of reflections of $W$:
\begin{align} \label{eq: word}
\ww=s_{i_1}s_{i_2} \ldots s_{i_\ell}.
\end{align}
We set
\begin{eqnarray} &&
\parbox{75ex}{
\begin{itemize}
\item $w_{\le k}\seteq s_{i_1}\cdots s_{i_k} \in W$ for $0\le k \le \ell$,

\item $\la_k\seteq w_{\le k}\Lambda_{i_k} \in \wl$ for $1\le k \le \ell$.
\end{itemize}
}\label{eq: word conv}
\end{eqnarray}

For a given sequence $\ww$ of $S$, $p \in \{ 1,\ldots, \ell\}$ and $j \in I$, we set
\begin{equation*} 
\begin{aligned}
p_+&\seteq\min( \{k \ | \ p<k\le \ell,\; i_k=i_p \} \cup \{ \ell+1 \}),\\
p_-&\seteq\max(\{k \ | \  1\le k<p, \; i_k=i_p \}\cup\{0\}),\\
p^+(j)&\seteq\min(\{k \ | \ p< k \le \ell,\; i_k=j \}\cup\{\ell+1\}), \\
p^-(j)&\seteq\max(\{k \ | \ 1\le k<p,\; i_k=j \}\cup\{0\}).
\end{aligned}
\end{equation*}

For a reduced expression $\tw=s_{i_1}s_{i_2} \cdots s_{i_\ell}$ of $w \in W$ and $0\le t \le s \le \ell$, we set
\begin{align}\label{eq: D_tw}
\D_\tw(s,t) \seteq \begin{cases}
\ \D(\la_s,\la_t) & \text{ if } t >0 \text{ and } i_s=i_t,\\
\D(\la_s,\Lambda_{i_s}) & \text{ if } 0=t < s \le \ell, \\
\qquad \one  & \text{ otherwise}.
\end{cases}
\end{align}

The $\Q(q)$-subalgebra of $\Aqn$ generated by $\D_\tw(i,i_-)$ $(1 \le i \le \ell)$, is independent of the choice of $\tw$. We denote it by $A_q(\n(w))$.
Then every $\D_\tw(s,t)$ is contained in $A_q(\n(w))$ \cite[Corollary 12.4]{GLS13}. The set $\B^{\upper}(A_{q}(\n(w))) \seteq \B^{\upper}(\Aqn) \cap A_q(\n(w))$ forms a
$\Q(q)$-basis of $A_q(\n(w))$~\cite[Theorem 4.25]{Kimu12}. We call $\B^{\upper}(A_{q}(\n(w)))$ the upper global basis of $A_q(\n(w))$.

The $\Aa$-module  $\Aanw$  generated by $\B^{\upper}(A_{q}(\n(w)))$ is an $\Aa$-subalgebra of $\Aqn$ (\cite[Theorem 4.27]{Kimu12}).

\subsection{Quantum affine algebras}\label{subsec:quntum affine}
In this subsection, we briefly review the representation theory of finite-dimensional integrable modules over
quantum affine algebras by following~\cite{AK,Kas02}.

When concerned with quantum affine algebras, we always take
the algebraic closure of $\C(q)$ in
$\bigcup_{m >0}\C((q^{1/m}))$ as the  base field $\cor$.

Let $I$ be an index set and  let $\cmA=(a_{ij})_{i,j\in I}$ be a generalized Cartan matrix of affine type.
We choose $0\in I$ as the leftmost vertices in the tables in~\cite[{pages 54, 55}]{Kac} except $A^{(2)}_{2n}$-case
 where we take the longest simple root as $\al_0$. Set $I_0 =I\setminus\{0\}$.

We normalize the $\Q$-valued symmetric bilinear form  $(\scbul,\scbul)$ on $\wl$ by
\begin{align*}
(\updelta, \la)=\lan \mathsf{c},\la\ran \quad \text{for any } \ \la\in \wl,
\end{align*}
where $\updelta$ denotes the {\it null root} and $\mathsf{c}=\sum_{i\in I}\mathsf{c}_ih_i$ denotes the {\it center}.
We denote by $\gamma$ the smallest positive integer such that $\gamma \dfrac{(\al_i,\al_i)}{2} \in \Z$ for all $i \in I$.

Let us denote by $U_q(\g)$ the quantum group over $\Q(q^{1/\gamma})$ associated with the affine Cartan datum $(\cmA,\wl, \Pi,\wl^{\vee},\Pi^{\vee})$.
 We denote by $U_q'(\g)$ the subalgebra of $U_q(\mathfrak{g})$ generated by
$e_i,f_i,t_i^{\pm 1} \seteq q^{\pm
\frac{(\al_i,\al_i)}{2}  h_i}$ for $i \in I$ and call it the \emph{quantum affine algebra}.

We use the comultiplication $\Delta$ of $U_q'(\g)$ given by
$$ \Delta(e_i)    = e_i \otimes t_i^{-1} + 1 \otimes e_i,  \  \Delta(f_i)    = f_i \otimes 1  + t_i \otimes f_i, \  \Delta(q^h)=q^h \otimes q^h.   $$

Let us denote by $ \ \bar{ } \ $ the involution of $U_q'(\g)$ defined as follows:
\[e_i \mapsto e_i,
\qquad\qquad
f_i \mapsto f_i,
\qquad\qquad
t_i \mapsto t_i^{-1} ,
\qquad\qquad
q^{1/\gamma} \to q^{-1/\gamma}.
\]

We denote by $\Ca_\g$ the category of finite-dimensional integrable $U_q'(\g)$-modules. For $M \in \Ca_g$, we denote by $M_\aff$ the {\it affinization} of $M$ which is  $\cor[z,z^{-1}] \tens M$ as a vector space
endowed with the $U_q'(\g)$-module structure given by
$$e_i(u_z)=z^{\delta_{i,0}}(e_iu)_z, \quad  f_i(u_z)=z^{-\delta_{i,0}}(f_iu)_z, \quad t_i(u_z)=(t_iu)_z.$$
Here $u_z$ denote the element $ \one \tens u \in M_\aff$ for $u \in M$. We also write $M_z$ instead of $M_\aff$.

A simple module $M$ in $\Ca_\g$ contains a non-zero vector $u$ of weight $\lambda\in \wl_\cl \seteq \wl/\Z \updelta$ such that
(1) $\langle h_i,\lambda \rangle \ge 0$ for all $i \in I_0$,
(2) all the weight of $M$ are contained in $\lambda - \sum_{i \in I_0} \Z_{\ge 0} \cl(\alpha_i)$,
where $\cl\colon \wl\to \wl_\cl$ denotes the canonical projection.
Such a $\la$ is unique and $u$ is unique up to a constant multiple.
We call $\lambda$ the {\it dominant extremal weight} of $M$ and $u$ the {\it dominant extremal weight vector} of $M$.

For $x \in \cor^\times$ and $M \in \Ca_\g$, we define $M_x \seteq M_\aff / (z_M-x)M_\aff$,
where $z_M$ denotes the $U_q'(\g)$-module automorphism of $M_\aff$ of weight $\updelta$. We call $x$ the {\it spectral parameter}.

For each $i \in I_0$, we set $$\varpi_i \seteq {\rm
gcd}(\mathsf{c}_0,\mathsf{c}_i)^{-1}\cl(\mathsf{c}_0\Lambda_i-\mathsf{c}_i \Lambda_0) \in \wl_\cl.$$
Then there
exists a unique simple $U_q'(\g)$-module $V(\varpi_i)$ in $\Ca_\g$,
called the {\em fundamental module of $($level $0$$)$ weight $\varpi_i$}, satisfying
the certain conditions (see \cite[\S 5.2]{Kas02}).

For a $U_q'(\g)$-module $M$, we denote by the $U_q'(\g)$-module $\overline{M}=\{ \bar{u} \mid u \in M \}$ whose module structure is given as $x \bar{u} \seteq \overline{\ms{2mu}\overline{x} u\ms{2mu}}$ for $x \in U_q'(\g)$.
Then we have
\begin{align} \label{eq: bar structure}
\overline{M_a} \simeq (\overline{M})_{\,\overline{a}}, \qquad\qquad \overline{M \otimes N} \simeq \overline{N} \otimes \overline{M}.
\end{align}
In particular, $\overline{V(\varpi_i)}\simeq V(\varpi_i)$ (see \cite[Appendix A]{AK}).

For a module $M$  in $\Ca_\g$, let us denote the right and the left dual of $M$ by ${}^*M$ and $M^*$, respectively.
That is, we have isomorphisms
\begin{align*}
&\Hom_{U_q'(\g)}(M\hspace{-.4ex} \tens \hspace{-.4ex} X,Y) \hspace{-.2ex} \simeq \hspace{-.2ex} \Hom_{U_q'(\g)}(X, \hspace{-.4ex} {}^*M \hspace{-.4ex} \tens \hspace{-.4ex} Y),\
\Hom_{U_q'(\g)}(X \hspace{-.4ex} \tens \hspace{-.4ex} {}^*M,Y)\hspace{-.2ex} \simeq \hspace{-.2ex} \Hom_{U_q'(\g)}(X, Y \hspace{-.4ex} \tens \hspace{-.4ex} M),\\
&\Hom_{U_q'(\g)}(M^* \hspace{-.4ex} \tens \hspace{-.4ex}  X,Y)\hspace{-.2ex} \simeq \hspace{-.2ex} \Hom_{U_q'(\g)}(X, M \hspace{-.4ex}  \tens \hspace{-.4ex}  Y),\
\Hom_{U_q'(\g)}(X \hspace{-.4ex} \tens \hspace{-.4ex}  M,Y)\hspace{-.2ex} \simeq \hspace{-.2ex} \Hom_{U_q'(\g)}(X, Y \hspace{-.4ex}  \tens \hspace{-.4ex} M^*),
\end{align*}
which are functorial in $U_q'(\g)$-modules $X$ and $Y$.
In particular, the module $V(\varpi_i)_x$ $(x\in \cor^\times)$ has the left dual and right dual
as follows:
\begin{equation*}
\bigl( V(\varpi_i)_x\bigr)^*  \simeq   V(\varpi_{i^*})_{x(p^*)^{-1}},
\qquad {}^*\bigl(V(\varpi_i)_x \bigr) \simeq   V(\varpi_{i^*})_{xp^*}
\end{equation*}
where $p^*$ is an element in $\cor^\times$ depending only on $U_q'(\g)$
(see \cite[Appendix A]{AK}), and $i \mapsto i^*$ denotes the involution on $I_0$ given by $\al_i=-w_0\,\al_{i^*}$.
Here $w_0$ is the longest element of $W_0 = \lan s_i \ | \ i \in I_0 \ran \subset W$.

We say that a $U_q'(\g)$-module $M$ is {\em good}
if it has a {\em bar involution}, a crystal basis with {\em simple
crystal graph}, and a {\em global basis} (see~\cite{Kas02} for
the precise definition). For instance, every fundamental module
$V(\varpi_i)$ for $i \in I_0$ is a good module. Note that
every good module is a simple $U_q'(\g)$-module. Moreover the tensor product of good modules is again good.
Hence any good module $M$ is {\it real} simple, i.e., $M\tens M$ is simple.

\begin{definition}
We call a $U_q'(\g)$ module $M$ {\it quasi-good} if
$$ M \simeq V_c $$
for some good module $V$ and $c \in \cor^\times$.
\end{definition}

\subsection{R-matrices}
In this subsection, we briefly review the notion of $R$-matrices for quantum affine algebras following~\cite[\S 8]{Kas02}.

For $M, N \in \Ca_\g$, there is a morphism of
$\cor[[z_N/z_M]] \tens _{\cor[z_N/z_M]} \cor[z_M^{\pm 1},z_N^{\pm 1}] \tens  U_q'(\g)$-modules, denoted by  $\Runiv_{M_{z_M},N_{z_N}}$ and  called the
{\it universal $R$-matrix}:
\begin{align*}
\Runiv_{M_{z_M},N_{z_N}} \colon \cor[[z_N/z_M]] \tens_{\cor[z_N/z_M]} (M_{z_M} \tens  N_{z_N}) \to \cor[[z_N/z_M]] \tens _{\cor[z_N/z_M]} (N_{z_N}\tens  M_{z_M}).
\end{align*}

We say that $\Runiv_{M_{z_M},N_{z_N}}$ is {\it rationally renormalizable} if there exist $a \in \cor((z_N/z_M))$ and a $\cor[z_M^{\pm 1},z_N^{\pm 1}] \tens  U_q'(\g)$-module homomorphism
\[
\Rren_{M_{z_M},N_{z_N}} \colon M_{z_M} \tens  N_{z_N} \to N_{z_N} \tens  M_{z_M}
\]
such that $\Rren_{M_{z_M},N_{z_N}} = a\Runiv_{M_{z_M},N_{z_N}}$.
Then we can choose $\Rren_{M_{z_M},N_{z_N}}$ so that for any $a_M, a_N \in \cor^\times$, the specialization of $\Rren_{M_{z_M},N_{z_N}}$ at $z_M=a_M$, $z_N=a_N$,
\begin{align*}
\Rren_{M_{z_M},N_{z_N}}|_{z_M=a_M,z_N=a_N} \colon M_{a_M} \otimes N_{a_N} \to N_{a_N} \otimes M_{a_M}
\end{align*}
does not vanish provided that $M$ and $N$ are non-zero $U_q'(\g)$-modules in $\Ca_\g$. It is called a {\it renormalized $R$-matrix}.

We denote by
\[
\rmat{M,N} \seteq \Rren_{M_{z_M},N_{z_N}}|_{z_M=1,z_N=1} \colon M \tens  N \to N \tens  M
\]
and call it the {\it $R$-matrix}. By the definition $ \rmat{M,N}$ never vanishes.

For simple $U_q'(\g)$-modules $M$ and $N$ in $\Ca_\g$, the universal $R$-matrix $\Runiv_{M_{z_M},N_{z_N}}$ is rationally renormalizable. Then, for
dominant extremal weight vectors $u_M$ and $u_N$ of $M$ and $N$, there exists $a_{M,N}(z_N/z_M) \in \cor[[z_N/z_M]]^\times$ such that
\begin{align}\label{eq: aMN}
\Runiv_{M_{z_M},N_{z_N}} \big( (u_M)_{z_M} \tens  (u_N)_{z_N}) \big) =  a_{M,N}(z_N/z_M) \big( (u_N)_{z_N} \tens  (u_M)_{z_M}) \big).
\end{align}
Then $\Rnorm_{M_{z_M},N_{z_N}} \seteq a_{M,N}(z_N/z_M)^{-1}\Runiv_{M_{z_M},N_{z_N}}$ is
a unique $\cor(z_M,z_N) \tens _{\cor[z_N/z_M]} \cor[z_M^{\pm 1},z_N^{\pm 1}] \tens U_q'(\g)$-module homomorphism sending
$\big( (u_M)_{z_M} \tens  (u_N)_{z_N} \big)$ to $\big( (u_N)_{z_N} \tens  (u_M)_{z_M} \big)$.

It is known that $\cor(z_M,z_N) \tens _{\cor[z_M^{\pm1},z_N^{\pm1}]} ( M_{z_M} \tens  N_{z_N})$ is a simple
$\cor(z_M,z_N) \tens _{\cor[z_M^{\pm1},z_N^{\pm1}]} U_q'(\g)$-module (\cite[Proposition 9.5]{Kas02}). We call
$\Rnorm_{M_{z_M},N_{z_N}}$ the {\it normalized $R$-matrix}.

Let us denote by $d_{M,N}(u) \in \cor[u]$ a monic polynomial of the smallest degree such that the image of
$d_{M,N}(z_N/z_M)\Rnorm_{M_{z_M},N_{z_N}}$ is contained in $N_{z_N} \tens  M_{z_M}$. We call $d_{M,N}$
the {\it denominator} of $\Rnorm_{M_{z_M},N_{z_N}}$. Then,
\[
d_{M,N}(z_N/z_M)\Rnorm_{M_{z_M},N_{z_N}} \colon M_{z_M} \tens  N_{z_N} \to N_{z_N} \tens  M_{z_M}
\]
is a renormalized $R$-matrix, and the $R$-matrix $\rmat{M,N} \colon M \tens  N \to N \tens  M$
is equal to $d_{M,N}(z_N/z_M)\Rnorm_{M_{z_M},N_{z_N}}\big|_{z_M=1,z_N=1}$ up to a constant multiple.

\subsection{Denominators of normalized $R$-matrices}
The denominators of the normalized $R$-matrices $d_{k,l}(z)\seteq d_{V(\varpi_k),V(\varpi_{l})}(z)$ between $V(\varpi_k)$ and $V(\varpi_l)$ were calculated in~\cite{AK,DO94,KKK15B,Oh15} for classical affine types
and in \cite{OT18} for exceptional affine types (see also~\cite{FR15,KMN2,KMOY07,Ya98}). In this subsection, we recall $d_{k,l}(z)$ for quantum affine algebras of type $A$ and $B$. In Table~\ref{table:Dynkin}, we list the Dynkin diagrams with an enumeration of simple roots and the corresponding fundamental weights for types $A$ and $B$.

\begin{table}[ht]
\renewcommand{\arraystretch}{2.2}
\centering
\begin{tabular}[c]{c|c|p{5cm}}
Type& Dynkin diagram & ${}^{\quad}$  Fundamental weights\\
\hline \hline
 $A_{1}^{(1)} $ &
 $$\xymatrix@R=3ex{  *{\circ}<3pt> \ar@{<=>}[r]_<{\alpha_0 \ }  &*{\circ}<3pt> \ar@{}[l]^<{ \ \ \alpha_1}  }$$ & $\varpi_1 = \cl(\Lambda_1 -\Lambda_0)$ \\
\hline
 $A_{n}^{(1)}$
($n \ge 2$) &
$$
\xymatrix@R=3ex{ & &&*{\circ}<3pt> \ar@{-}[drrr]^<{\alpha_0} \ar@{-}[dlll]\\
*{\circ}<3pt> \ar@{-}[r]_<{\alpha_1}  &*{\circ}<3pt>
\ar@{-}[r]_<{\alpha_2} & *{\circ}<3pt> \ar@{-}[r]_<{\alpha_3} & {}
\ar@{.}[r]  & *{\circ}<3pt> \ar@{-}[r]_>{\,\,\,\ \alpha_{n-1}}
&*{\circ}<3pt>\ar@{-}[r]_>{\,\,\,\,\alpha_n} &*{\circ}<3pt> }
$$ & $\varpi_i = \cl(\Lambda_i -\Lambda_0)$ \newline $(1 \le i \le n)$\\
\hline $A_{2}^{(2)}$ & \hskip-3em $$ \put(0,0){\circle{4}}
\put(40,0){\circle{4}} \put(3.7,-3.6){\line(1,0){23}}
\put(5,-1.15){\line(1,0){26}} \put(5,1.15){\line(1,0){26}}
\put(3.7,3.6){\line(1,0){23}} \put(20.5,-5.2){{\LARGE\mbox{$>$}}}
\put(-3,-10){\tiny $\alpha_0$} \put(37,-10){\tiny $\alpha_1$}
$$
& $\varpi_1=\cl(2\Lambda_1- \Lambda_0)$
 \\[1ex]
\hline $A_{3}^{(2)}$ &
$$\xymatrix@R=3ex{*{\circ}<3pt>  \ar@{<=}[r]_<{\alpha_0} & *{\circ}<3pt> \ar@{=>}[r]_<{\alpha_2} \ar@{}[r]_>{\,\,\,\ \alpha_1}  & *{\circ}<3pt>}$$
 & $\varpi_1=\cl(\Lambda_1- \Lambda_0)$,  \newline   $\varpi_2=\cl(\Lambda_2-2\Lambda_0)$ \\
\hline $A_{2n-1}^{(2)}$ ($n \ge 3$) &
$$
\xymatrix@R=3ex{ & *{\circ}<3pt> \ar@{-}[d]^<{\alpha_0}\\
*{\circ}<3pt> \ar@{-}[r]_<{\alpha_1}  &*{\circ}<3pt>
\ar@{-}[r]_<{\alpha_2}  &
  *{\circ}<3pt> \ar@{-}[r]_<{\alpha_3} & {}
\ar@{.}[r] & *{\circ}<3pt> \ar@{-}[r]_>{\,\,\,\ \alpha_{n-1}}
&*{\circ}<3pt>\ar@{<=}[r]_>{\,\,\,\,\alpha_{n}} &*{\circ}<3pt> }
$$ & $\varpi_i=\cl(\Lambda_i-\Lambda_0)$ \newline $(i=1,n)$,  \newline $\varpi_i=\cl(\Lambda_i- 2\Lambda_0)$ \newline $(2\le i \le n-1)$  \\
\hline $A_{2n}^{(2)}$ ($n \ge 2$) & $$
\raisebox{-.7em}{\xymatrix@R=3ex{
*{\circ}<3pt> \ar@{=>}[r]_<{\alpha_0}  &*{\circ}<3pt>
\ar@{-}[r]_<{\alpha_1}   & *{\circ}<3pt> \ar@{-}[r]_<{\alpha_2} & {}
\ar@{.}[r] & *{\circ}<3pt> \ar@{-}[r]_>{\,\,\,\ \alpha_{n-1}}
&*{\circ}<3pt>\ar@{=>}[r]_>{\,\,\,\,\alpha_n} &*{\circ}<3pt> }}
$$ & $\varpi_i = \cl(\Lambda_i -\Lambda_0)$
\newline $(i=1,\ldots,n-1)$, \newline $\varpi_n =\cl(2\Lambda_n - \Lambda_0) \ $ \\ \hline
\raisebox{-2em}{$B_{n}^{(1)}$ ($n \ge 3$)}
 &
$
\xymatrix@R=3ex{ *{\circ}<3pt> \ar@{-}[dr]^<{0}   \\
&*{\circ}<3pt> \ar@{-}[r]_<{2}  & {} \ar@{.}[r]
& *{\circ}<3pt> \ar@{-}[r]_>{\,\,\,\ n-1} &*{\circ}<3pt>\ar@{=>}[r]_>{\,\,\,\,n} &*{\circ}<3pt>  \\
*{\circ}<3pt> \ar@{-}[ur]_<{1} }
$
&
$\varpi_1=\cl( \Lambda_1-\Lambda_0)$,  \newline
$\varpi_i =\cl(\Lambda_i-2\Lambda_0)$ \newline
$(2\le i \le N-1)$, \newline
$\varpi_n=\cl(\Lambda_n-\Lambda_0) $
\\
\hline
\raisebox{-1em}{$B_{2}^{(1)}$} &
$$\xymatrix@R=3ex{*{\circ}<3pt>  \ar@{=>}[r]_<{\alpha_0} & *{\circ}<3pt> \ar@{<=}[r]_<{\alpha_2} \ar@{}[r]_>{\,\,\,\ \alpha_1}  & *{\circ}<3pt>}$$
& $\varpi_1=\cl( \Lambda_1-\Lambda_0)$,  \newline
$\varpi_2 =\cl(\Lambda_2-\Lambda_0 )$  \\
\hline
\end{tabular}
\caption{Dynkin diagrams and fundamental weights$\ake[3.5ex]$}
\label{table:Dynkin}
\end{table}
\begin{remark} \hfill
\begin{enumerate} [{\rm (a)}]
\item Note that the convention for Dynkin diagram of type $A^{(2)}_{2n}$ is different from the one in~\cite[page 54,\;55]{Kac}. However,
for each $i \in I_0$ the corresponding fundamental modules $V(\varpi_i)$ are isomorphic to each other, since the corresponding fundamental weights are conjugate to each other
under the Weyl group action (see~\cite[\S 5.2]{Kas02}).
\item Note that the Dynkin diagrams of type $B^{(1)}_2$ and $A^{(2)}_3$
in Table~\ref{table:Dynkin} are denoted by $C^{(1)}_2$ and $D^{(2)}_3$
in~\cite[page 54,\;55]{Kac}, respectively.
\item Our conventions on quantum affine algebras are different from~\cite{HL16,HO18}. To compare, we refer to \cite[Remark 3.28]{HO18}.
\end{enumerate}
\end{remark}

\begin{theorem}[{\cite{DO94,Oh15}}]\label{thm: denominators}
We have the following denominator formulas.
\begin{enumerate}[{\rm (a)}]
\item For $\g=A^{(1)}_{n-1}$ $(n \ge 2)$, $1\leq k,l \leq n-1$, we have
\begin{align*} 
d_{k,l}(z)=\hspace{-4ex} \prod\limits_{s=1}^{\min(k, l, n-k, n-l)}\hspace{-2ex} \bigl( z-(-q)^{|k-l|+2s}\bigr).
\end{align*}
\item For $\g=A^{(2)}_{n-1}$ $(n \ge 3)$, $1\leq k,l \leq \lfloor n/2 \rfloor$, we have
\begin{align*} 
  d_{k,l}(z) = \prod_{s=1}^{\min(k,l)} (z-(-q)^{|k-l|+2s})(z+q^n(-q)^{-k-l+2s}).
\end{align*}
\item For $\g=B^{(1)}_{n}$ $(n \ge 2)$ $1\leq k,l \leq n-1$, we have
\begin{align}\label{eq:denom B^(1) le n}
d_{k,l}(z)  =\prod_{s=1}^{\min (k,l)} (z-(-q)^{|k-l|+2s})(z+(-q)^{2n-k-l-1+2s}).
\end{align}
For $1 \leq k \leq n$, we have
\begin{align}\label{eq:denom B^(1) n}
d_{k,n}(z) = \begin{cases}
\displaystyle\prod_{s=1}^{k}(z-(-1)^{n+k}q_s^{2n-2k-1+4s}) & \text{ if } 1 \leq k \leq n-1, \\*
\displaystyle\prod_{s=1}^{n} (z-(q_s)^{4s-2}). & \text{ if } k =n,
\end{cases}
\end{align}
where $q_s \seteq q_n=q^{1/2}$. 
\end{enumerate}
\end{theorem}

\subsection{Hernandez-Leclerc category}\label{sec:HL}
For each quantum affine algebra $U_q'(\g)$, we define a quiver $\Cquiver(\g)$ as follows:
\begin{enumerate}[{\rm (i)}]
\item Take the set of equivalence classes $\hat{I}_{\g} \seteq (I_0 \times \cor^\times) /  \sim$ as the set of vertices,
where the equivalence relation is given by $(i,x) \sim (j,y)$ if and only if $V(\varpi_i)_x  \simeq  V(\varpi_j)_y$.
\item  Put $d$-many arrows from $(i,x)$ to $(j,y)$, where $d$ denotes the order of zero of $d_{i,j}(z_j / z_i)$ at $z_j / z_i = {y / x}$.
\end{enumerate}
Note that  $(i,x)$ and $(j,y)$ are connected by at least one arrow in $\Cquiver(\g)$  if and only if $V(\varpi_i)_x \tens V(\varpi_j)_y$ is reducible (\cite[Corollary 2.4]{AK}).

\smallskip

Let $\Cquiver_0(\g)$ be a connected component of $\Cquiver(\g)$.
Note that a connected component of $\Cquiver(\g)$ is unique up to a spectral parameter shift
and hence $\Cquiver_0(\g)$ is uniquely determined up to a quiver isomorphism. For types $A$ and $B$, one can take
\eq&&\ba{rl}
\Cquiver_0(A^{(1)}_{n}) &\seteq \{(i,(-q)^{p}) \in I_0 \times \cor^{\times} \ | \  p \equiv i+1\ \mathrm{mod}\  2 \}, \\
\Cquiver_0(A^{(2)}_{2n-1}) &\seteq \{( i,\pm(-q)^{p}) \in I_0 \times \cor^{\times} \ | \ \text{$i  \in I_0$, $p \equiv i+1 \ {\rm mod}\  2$} \}, \\
\Cquiver_0(A^{(2)}_{2n}) &\seteq \{(i,(-q)^{p}) \in  I_0\times \cor^{\times} \ | \, p \in \Z \},  \\
\Cquiver_0({B_{n}^{(1)}}) &\seteq \{ (i,(-1)^{i-1}q_{\kappa} q^m),  (n,q^m)\ | \ 1\le i \le n-1, \ m \in \Z\},
\ea\label{eq:S0B^(1)_n}
\eneq
where $q_{\kappa}$ in~\eqref{eq:S0B^(1)_n} is defined as follows:
\begin{align*}
q_{\kappa} \seteq (-1)^{n+1}q_s^{2n+1}.
\end{align*}
We remark here that
\begin{enumerate}[{\rm (i)}]
\item $V(\varpi_n)_x\simeq V(\varpi_n)_{-x}$ in the $A^{(2)}_{2n-1}$-case,
\item $\Cquiver_0(A^{(1)}_{2n-1}) \simeq \Cquiver_0(A^{(2)}_{2n-1})$ as quivers (\cite[(2.7)]{KKKO15C}).
\end{enumerate}

\noi
Let us denote by $\Ca^0_\g$ the smallest abelian subcategory of  $\Ca_\g$ such that
\begin{enumerate}
\item[{\rm (a)}] $\Ca^0_\g$ contains $\{V (\varpi_i)_x \ | \  (i,x) \in \Cquiver_0(\g) \}$,
\item[{\rm (b)}] it is stable under taking submodules, quotients, extensions and tensor products.
\end{enumerate}
The category $\Ca^0_\g$ for symmetric affine type $U_q'(\g)$ was introduced in~\cite{HL10}.
Note that every simple module in $\Ca_\g$ is a tensor product of certain parameter shifts of some simple modules in $\Ca^0_\g$~\cite[\S 3.7]{HL10}.
The Grothendieck ring $K(\Ca^0_\g)$ of $\Ca^0_\g$ is the polynomial ring generated by the classes of modules in $\{V (\varpi_i)_x \ | \ (i,x) \in \Cquiver_0(\g)\}$~\cite{FR99}.

\section{Quantum cluster algebras and monoidal categorification} \label{Sec: cluster}
In this section, we recall the definition of quantum cluster algebras introduced in~\cite{BZ05,FZ02}.
Then we review the monoidal categorification of a quantum cluster algebra developed in~\cite{KKKO18} (see also~\cite{HL10}).

\subsection{Quantum cluster algebras} Fix a countable index set
$\K=\Kex \sqcup \Kfr$ which is decomposed into the subset $\Kex$ of exchangeable indices and the subset $\Kfr$ of frozen indices.
Let $L=(\la_{ij})_{i,j\in \K}$ be a skew-symmetric integer-valued $\K \times \K$-matrix.

Let $A$ be a $\Ah$-algebra.
We say that a family $\{ x_i \}_{i \in \K}$ of elements in $A$ is {\it $L$-commuting} if it satisfies
$$x_ix_j=q^{\la_{ij}}x_jx_i \quad \text{ for any } i,j \in \K.$$
We say that an $L$-commuting family $\{ x_i \}_{i \in \K}$
is {\em  algebraically independent} if the family
$$\set{x_{i_1}\cdots x_{i_\ell}}{\ell\in\Z_{\ge0},\; i_1,\ldots,i_\ell\in K,\; i_1\le\cdots\le i_\ell}$$
is linearly independent over  $\Z[q^{\pm1/2}]$.
Here $\le$ is a total order on $K$.

Let $\tB=(b_{ij})_{(i,j)\in \K \times \Kex}$ be an integer-valued matrix such that
\eq &&
\parbox{75ex}{
\begin{enumerate}[{\rm (a)}]
\item
 for each $j \in \Kex$, there exist finitely many $i \in \K$ such that $b_{ij} \ne 0$,
\item  the {\it principal part} $B \seteq (b_{ij})_{i,j \in \Kex}$ is skew-symmetric.
\end{enumerate}
}\label{eq: condition B}
\eneq
We extend the definition of $b_{ij}$ for $(i,j)\in K\times K$ by:
$$\text{
$b_{ij}=-b_{ji}$ if $i\in \Kex$ and $j\in K$ and
$b_{ij}=0$ for $i,j\in \Kex$. }$$

To the matrix $\tB$, we associate the quiver $Q_{\tB}$ such that
the set of vertices is $\K$ and
the number of arrows from $i\in\K$ to $ j\in\K$ is $\max(0,b_{ij})$.
Then, $Q_{\tB}$ satisfies that
\begin{eqnarray} &&
\parbox{75ex}{
\begin{enumerate}[{\rm (a)}]
\item the set of vertices of $Q_{\tB}$ are labeled by $\K$,
\item $Q_{\tB}$ does not have  loops, $2$-cycle and arrows between frozen vertices,
\item each exchangeable vertex $v$ of $Q_{\tB}$ has {\it finite degree}; that is, the number of arrows incident with $v$ is finite.
\end{enumerate}
}\label{eq: Quiver condition}
\end{eqnarray}

Conversely, for a given quiver satisfying~\eqref{eq: Quiver condition}, we can associate a matrix $\tB$ by
\begin{align}\label{eq: bij}
b_{ij} \seteq \text{(the number of arrows from $i$  to $j$)} \hspace{-.2ex}   -  \hspace{-.2ex}  \text{(the number of arrows from $j$  to $i$)}.
\end{align}
Then $\tB$ satisfies~\eqref{eq: condition B}.

We say that the pair $(L,\tB)$ is {\it compatible with a positive integer $d$}, if
$$   \sum_{k \in \K} \lambda_{ik}b_{kj} =\delta_{i,j}d \qquad \text{ for each $i \in \K$ and $j \in \Kex$}.$$

\begin{definition}
For a $\Ah$-algebra $A$, a triple $\Seed=(\{x_i\}_{i \in \K},L,\tB)$ consisting of
\begin{enumerate}
\item[{\rm (i)}] a compatible pair $(L,\tB)$,
\item[{\rm (ii)}] an $L$-commuting algebraically independent family $\{x_i\}_{i \in \K}$
\end{enumerate}
is called a {\it quantum seed} in $A$. We also call
\begin{enumerate}
\item[{\rm (a)}] $\{x_i\}_{i \in \K}$ the {\it cluster} of $\Seed$ and elements in $\{x_i\}_{i \in \K}$ the {\it cluster variables},
\item[{\rm (b)}] $x_i$ $(i \in \Kex)$ the {\it exchangeable variables}  and $x_i$ $(i \in \Kfr)$ the {\it frozen variables},
\item[{\rm (c)}] $x^{{\bf a}}$ $\bigl({\bf a} \in \Z_{\ge 0}^{\oplus \K}\bigr)$ the {\it quantum cluster monomials},
\end{enumerate}
where
$$ x^{{\bf a}} \seteq q^{1/2 \sum_{s>t}a_{i_s}a_{i_t} \la_{i_s,i_t}} x_{i_1}^{a_{i_1}}
 \cdots x_{i_r}^{a_{i_r}},$$
for ${\bf a}=(a_{i})_{i\in\K}$,
$\set{i\in\K}{a_i\not=0}\subset\{i_1,\ldots,i_r\}$. Note that $x^{{\bf a}}$ does not depend on the choice of $\{i_1,\ldots,i_r\}$.
\end{definition}

For $k \in \Kex$, the {\it mutation $\mu_k(L,\tB) \seteq (\mu_k(L),\mu_k(\tB))$ of a compatible pair $(L,\tB)$ in direction $k$} is a pair $(\mu_k(L),\mu_k(\tB))$ consisting of
a $\K \times \K$-matrix $\mu_k(L)$ and a $\K \times \Kex$-matrix $\mu_k(\tB)$ defined as follows$\colon$
\begin{eqnarray} &&
\parbox{81ex}{
\begin{enumerate}
\item[{\rm (a)}] $\mu_k(L)_{ij} =
\begin{cases}
  -\la_{ij}+\displaystyle\sum _{t\in\K} \max(0, -b_{tk}) \la_{tj} \quad \  & \text{if} \ i=k, \ j\neq k, \\
  -\la_{ij}+\displaystyle\sum _{t\in\K} \max(0, -b_{tk}) \la_{it} & \text{if} \ i \neq k, \ j= k, \\
   \la_{ij} & \text{otherwise.}
\end{cases}$

\vs{1ex}
\item[{\rm (b)}]  $\mu_k(\tB)_{ij} =
\begin{cases}
  -b_{ij} & \text{if}  \ i=k \ \text{or} \ j=k, \\
  b_{ij} + (-1)^{\true(b_{ik} < 0)} \max(b_{ik} b_{kj}, 0) & \text{otherwise.}
\end{cases}
$
\end{enumerate}
}\label{eq: mutation in a direction k}
\end{eqnarray}
Then one can check that $\mu_k(\tB)$ satisfies~\eqref{eq: condition B} and the pair $(\mu_k(L),\mu_k(\tB))$ is compatible with the same integer $d$ as
in~\cite{BZ05}.

We define
\begin{align*} 
a_i'= \begin{cases}
  -1 & \text{if} \ i=k, \\
 \max(0,b_{ik}) & \text{if} \ i\neq k,
\end{cases} \qquad
a_i''= \begin{cases}
  -1 & \text{if} \ i=k, \\
 \max(0,-b_{ik}) & \text{if} \ i\neq k.
\end{cases}
\end{align*}
and set ${\bf a}'\seteq(a_i')_{i\in\K}$ and ${\bf a}''\seteq(a_i'')_{i\in\K}$ which are contained in $\Z^{\oplus \K}$.

Let $A$ be a $\Ah$-algebra contained in a skew-field $\mathscr{F}$. Let $\Seed=(\{x_i\}_{i \in \K},L,\tB)$ be a quantum seed in $A$. For $k \in \Kex$, we define the elements $\mu_k(x)_i$ of $K$ as follows$\colon$
$$
\mu_k(x)_i  =\begin{cases}
x^{{\bf a}'}  +   x^{{\bf a}''}, & \text{if} \ i=k, \\
x_i & \text{if} \ i\neq k.
\end{cases}
$$
Then $\{\mu_k(x)_i\}_{i \in \K}$ is
a $\mu_k(L)$-commuting algebraically independent family. 
We call
\begin{align*}
\mu_k(\Seed)\seteq \bigl(\{\mu_k(x)_i\}_{i\in\K},\mu_k(L),\mu_k(\tB)\bigr)
\end{align*}
the {\it mutation of $\Seed$ in direction $k$}.

 \begin{definition}
Let $\Seed=(\{x_i\}_{i\in\K},L, \tB)$ be a quantum seed in $A$.
The {\it quantum cluster algebra $\mathscr{A}_{q^{1/2}}(\Seed)$} associated to the quantum seed $\Seed$ is  the $\Ah$-subalgebra of the skew field $\mathscr{F}$
generated by all the quantum cluster variables in the quantum seeds obtained from $\Seed$ by any sequence of mutations.
 \end{definition}
We call $\Seed$ the {\it initial quantum seed} of the quantum cluster algebra $\mathscr{A}_{q^{1/2}}(\Seed)$.

\subsection{Quantum cluster algebras $\mathscr{A}_{q^{1/2}}(\n(w))$}
In this subsection, we assume that the generalized Cartan matrix $\cmA$ is symmetric. 
Let $\tw=s_{i_1}\cdots s_{i_r}$ be a reduced expression of $w\in W$.
By Proposition~\ref{prop: multi D},
$\D_\tw(i,0)$ and $\D_\tw(j,0)$ {\it $q$-commute}; i.e., there exists $\la_{ij} \in \Z$ satisfying
$$  \D_\tw(i,0)\D_\tw(j,0)=q^{\la_{ij}}\D_\tw(j,0)\D_\tw(i,0).$$
Hence we have an integer-valued skew-symmetric matrix $L=(\lambda_{ij})_{1 \le i,j \le r}$.

Set
\begin{align} \label{eq: decomposition of indices}
\text{$\K=\{1,\ldots,r\}$, $\Kfr=\{ k \in \K \ | \ k_+=r+1\}$, and $\Kex \seteq K \setminus \Kfr$.}
\end{align}

\begin{definition}[\cite{GLS13}] \label{def: quiver Q assoc tw}
We define the quiver $Q$ with the set of vertices $Q_0$ and the set of arrows $Q_1$ which is associated to $\tw$ as follows:
\begin{enumerate}
\item[$(Q_0)$] $Q_0=\K=\{1,\ldots,r\}$.
\item[$(Q_1)$] There are two types of arrows:
\begin{itemize}
\item ordinary arrows $ \  \  \colon$ $s\To[\ |a_{i_s,i_t}|\ ] t$ \hs{3.8ex} if $1\le s<t<s_+<t_+\le r+1$,
\item horizontal arrows $\colon$ $s\To[{ \hspace{6.3ex} }] s_-$
\hs{3ex}if $1\le s_-<s\le r$.
\end{itemize}
\end{enumerate}
Let $\tB = (b_{ij})$ be the integer-valued $\K\times\Kex$-matrix associated to the quiver $Q$ by~\eqref{eq: bij}.
\end{definition}

\begin{proposition}[{\cite[Proposition 10.1]{GLS13}}] The pair $(L,\tB)$ is compatible with $d=2$.
\end{proposition}

\begin{theorem}[{\cite[Theorem 12.3]{GLS13}, \cite[Corollary 11.2.8]{KKKO18}}]  Let
$\mathscr{A}_{q^{1/2}}(\Seed)$ be the quantum cluster algebra associated to the initial quantum seed
$$ \Seed \seteq( \{ q^{-(d_s,d_s)/4}\D_\tw(s,0) \}_{1 \le s \le r}, L, \tB ),$$
where $d_s \seteq\wt\bl\D_\tw(s,0)\br$.
Then we have $\Z[q^{\pm 1/2}]$-algebra isomorphism
\begin{align*}
 \mathscr{A}_{q^{1/2}}(\Seed) \simeq
\mathscr{A}_{q^{1/2}}(\n(w))\seteq \Z[q^{\pm 1/2}] \otimes_{\Z[q^{\pm 1}]}\Aanw.
 \end{align*}
\end{theorem}

\subsection{Cluster algebras $K(\Ca_\g^-)$} \label{subsec: HL cluster -}
In \cite{HL16}, Hernandez and Leclerc introduced a proper
subcategory $\Ca_\g^-$ of $\Ca_\g$ for untwisted quantum affine
algebras which contains all simple objects of $\Ca_\g$ up to
parameter shifts. The definitions of $\Ca_\g^-$ for types
$A_n^{(1)}$ and $B_n^{(1)}$ can be taken as follows. Take
$\Cquiver_-(\g)$ the subset of $I \times \cor^\times$ as
\begin{align*}
\Cquiver_-(A^{(1)}_{n}) \seteq & \{(i,(-q)^p)  \ | \ p \le 0 \text{ and } p \equiv i \; \mathrm{mod}\  2 \}, \\
\Cquiver_-({B_{n}^{(1)}}) \seteq & \{ (i,(-1)^{i-1}q_{\kappa} q^m) \in \Cquiver_0(B^{(1)}_{n}) \ | \ n+\frac{1}{2}+m  \in \frac{1}{2}\Z_{\le 0}\} \\
& \hs{30ex}\cup  \{ (n,q^m)\in \Cquiver_0(B^{(1)}_{n}) \ | \ m \in \Z_{\le 0} \}.
\end{align*}
The category $\Ca^-_\g$ is the smallest abelian full subcategory of
$\Ca_\g$ such that
\begin{enumerate}
\item[{\rm (a)}] $\Ca^-_\g$ contains $\{V (\varpi_i)_x \ | \  (i,x) \in \Cquiver_-(\g) \}$,
\item[{\rm (b)}] it is stable under taking submodules, quotients, extensions and tensor products.
\end{enumerate}

\noi Also, they defined the quiver $G_\g^-$ of infinite rank, whose
vertices are labeled by a certain subset $\{(i,x)\}$ of $I_0 \times
\Z_{\le 0}$ (see \cite{HL16} for details). Let
$\mathbf{z}=\{z_{i,x}\}$ be the indeterminates labeled by the
$\{(i,x)\}$.

\begin{theorem}[{\cite[Theorem 5.1]{HL16}}] \label{thm: HL16} There exists an isomorphism between the cluster algebra $\mathscr{A}$ associated with the initial seed $(\mathbf{z},G_\g^-)$
and the Grothendieck ring $K(\Ca^-_\g)$ of $\Ca^-_\g$.
\end{theorem}

\begin{conjecture}[{\cite[Conjecture 13.2]{HL10}, \cite[Conjecture 5.2]{HL16}}] \label{conj: HL}
The cluster monomials of $\mathscr{A}$ can be identified with
the real simple modules of $\Ca^-_\g$ under the isomorphism in Theorem~\ref{thm: HL16}.
\end{conjecture}

We will give a proof of Conjecture~\ref{conj: HL} for $\Ca^-_{A^{(1)}_{N-1}}$ in Section~\ref{Sec: CJ for affine A}.

\subsection{Monoidal categorification of quantum cluster algebras}
\label{sec:monoiddal}

In this subsection, we fix a base field $\cor$
and a free abelian group $\rl$ equipped with a symmetric bilinear form
$( \ , \ ) \colon \rl \times \rl$ such that $(\be,\be) \in 2\Z \text{ for all } \be \in \rl$.

Let $\shc$ be a $\cor$-linear abelian monoidal category (see~\cite[Appendix A.1]{KKK18A}) in the sense that
it is abelian and the tensor functor $\tens$ is $\cor$-bilinear and exact.
A simple object $M$ in $\shc$ is called {\it real} if $M \otimes M$ is simple.

We assume that $\shc$ satisfies the following conditions$\colon$
\bnum
\item  Any object of $\shc$ is of a finite length. \label{item: cond1}
\item $\cor \isoto \Hom_\shc(M,M)$ for any simple object $M$ of $\shc$. \label{item: cond2}
\item $\shc$ admits a direct sum decomposition $\shc = \soplus_{\be \in \rl} \shc_\be$ such that the tensor functor $\tens$ sends $\shc_\be \times \shc_\gamma \to \shc_{\be+\gamma}$ for every
$\be,\gamma \in \rl$. \label{item: cond3}
\item There exists an object $Q \in \shc_0$ satisfying
\begin{itemize}
\item[{\rm (a)}] there is an isomorphism $R_Q(X) \colon Q \tens X \isoto X \tens Q$
functorial in $X \in \shc$ that
$$\xymatrix@C=7ex{
Q\tens X\tens Y\ar[r]_-{R_Q(X)}\ar@/^3ex/[rr]^{R_Q(X\tens Y)}
&X\tens Q\tens Y\ar[r]_-{R_Q(Y)}
&X\tens Y\tens Q}$$
 commutes for any $X,Y\in \shc$,
\item[{\rm (b)}] the functor $X \mapsto Q \tens X$ is an  equivalence of categories.
\end{itemize} \label{item: cond4}
\item For any $M$, $N\in\shc$, we have
$\Hom_\shc(M,Q^{\tens n}\tens N)=0$ except finitely many integers $n$ (see Remark~\ref{rem: Qn} below). \label{item: cond5}
\end{enumerate}

\begin{remark}\label{rem: Qn}
Note that, since the functor $X \mapsto Q \tens X$ is an  equivalence of categories, there exists an object $Q^{-1} \in \shc_0$ such that $ Q \tens Q^{-1} \simeq Q^{-1} \tens Q \simeq \one$,
and the functor $X \mapsto Q^{-1} \tens X$ give an equivalence of categories also (see \cite[Appendix A.1]{KKK18A}). Thus, for each $n \in \Z$, we define
$Q^{\tens n} \seteq  \underbrace{Q \tens Q \tens \cdots \tens Q}_{\text{$n$-times}}$ if $n \ge 0$, $\underbrace{Q^{-1} \tens Q^{-1} \tens \cdots \tens Q^{-1}}_{\text{$-n$-times}}$  if $n < 0$.
\end{remark}

We denote by $q$ the auto-equivalence $Q \tens \scbul$, and call it the {\it grading shift functor}. With the grading shift functor, the Grothendieck ring
$K(\shc)$ becomes a $\rl$-graded $\Z[q^{\pm 1}]$-algebra.

For $M \in \shc_\be$, we write $\be = \wt(M)$ and call it the {\it weight} of $M$. Similarly, for $x \in \Ah \otimes_{\Z[q^{\pm 1}]} K(\shc_\be)$, we write $\be = \wt(x)$ and call it the weight of $x$.

\begin{definition} \label{def:quantum monoidal pair}
A pair $(\{M_i\}_{i\in\K},\tB)$ consisting of
\begin{enumerate}
\item[{\rm (1)}] a family of real simple objects $\{M_i\}_{i\in\K}$ of $\shc$, where $M_i \in \shc_{d_i}$ for some $d_i \in  \rl $,
\item[{\rm (2)}] an integer valued $\K \times \Kex$-matrix $\tB$
\end{enumerate}
is called a \emph{quantum monoidal seed} if it satisfies the following properties:
\begin{itemize}
\item[{\rm (i)}] for all $ i, j \in\K$, there exists an integer $\lambda_{ij}$ satisfying $$M_i \tens M_j \simeq q^{\lambda_{ij}} M_j \otimes M_i,$$
\item[{\rm (ii)}] $M_{i_1} \tens \cdots \tens M_{i_t}$ is simple for any finite sequence  $(i_1,\ldots,i_t)$ in $\K$,
\item[{\rm (iii)}] the  integer valued $\K \times \Kex$-matrix  $\tB$ satisfies~\eqref{eq: condition B},
\item[{\rm (iv)}] $(L,\tB)$  is  compatible with $d=2$, where $L =(\la_{ij})_{i,j\in \K}$,
\item[{\rm (v)}] $\la_{ij} - (d_i,d_j) \in 2\Z$ for all $i,j \in\K$, where $D\seteq= \{ d_i \in \rl \ | \ M_i \in \shc_{d_i}  \}_{i \in \K} \subset \rl$.
\item[{\rm (vi)}] $\sum_{i\in\K}b_{ik}d_i =0$ for all $k\in\Kex$.
\end{itemize}
\end{definition}
Note that for a quantum monoidal seed $(\{M_i\}_{i\in\K},\tB)$, the matrix $L=(\la_{ij})_{i,j\in \K}$ and the family $D \seteq \{ d_i \}_{i \in \K}$ is determined by the family $\{M_i\}_{i\in\K}$.

Let $\Seed=(\{M_i\}_{i \in \K} ),\tB)$ be a quantum monoidal seed in $\shc$.
For $X \in \shc_\be$ and $Y \in \shc_\ga$ such that $X \otimes Y \simeq q^c\, Y \otimes X$ and $c+(\be,\ga) \in 2\Z$, we define
\begin{align*}
\widetilde{\la}(X,Y) \hspace{-.2ex}\seteq\hspace{-.2ex} \dfrac{1}{2}\bigl( -c+(\be,\ga) \bigr) \hspace{-.2ex} \in\hspace{-.2ex} \Z \   \text{ and }  \  X \sodot Y \hspace{-.2ex} \seteq \hspace{-.2ex} q^{\widetilde{\la}(X,Y)}X \otimes Y \hspace{-.2ex}\simeq \hspace{-.2ex} q^{\widetilde{\la}(Y,X)} Y \otimes X.
\end{align*}
Then we have $X \sodot Y \simeq Y \sodot X$.

For any finite sequence $(i_1,\ldots,,i_\ell)$ in $\K$, we define
$$\sodot_{k=1}^{\ell} M_{i_k} \seteq(\cdots(M_{i_1}\sodot M_{i_2})\sodot\cdots)\sodot M_{i_{\ell-1}})\sodot M_{i_\ell}.$$
When the $L$-commuting family $\{ [M_i] \}_{i \in \K}$ of elements in $\Ah \tens_{\Z[q^{\pm 1}]} K(\shc)$ is algebraically independent,
we define a quantum seed $[\Seed]$ in $\Ah \tens_{\Z[q^{\pm 1}]}K(\shc)$ by
$$  [\Seed] = \bigl( \{q^{-(d_i,d_i)/4} [M_i] \}_{i \in \K},L,\tB \bigr).$$

For a given $k\in\Kex$, we define the {\it mutation $\mu_k(D) \in \rl$ of $D$ in direction $k$ with respect to $\tB$} by
\begin{align*}
\mu_k(D)_i =d_i \ (i \neq k), \quad \mu_k(D)_k=-d_k+\sum_{b_{ik} >0}   b_{ik} d_i.
\end{align*}
Note that, for any $k\in\Kex$, we have
$\mu_k(\mu_k(D))=D$, and $(\mu_k(L),\mu_k(\tB),\mu_k(D))$ satisfies
(v) and (vi) in Definition~\ref{def:quantum monoidal pair}.

For $k \in \Kex$, set
\begin{align} \label{eq:mm'}
m_k \seteq \dfrac{1}{2}(d_k,\zeta) +\dfrac{1}{2} \displaystyle  \sum_{b_{ik} < 0}\la_{ki}  b_{ik} \quad  \text{ and } \quad
m'_k \seteq \dfrac{1}{2}(d_k,\zeta) +\dfrac{1}{2} \displaystyle \sum_{b_{ik} > 0} \la_{ki}b_{ik},
\end{align}
where $\zeta=-d_k+\sum_{b_{ik}>0}b_{ik}d_i$.
\begin{definition} \label{def:monoidal mutation}
We say that a quantum monoidal seed $\Seed$ in $\shc$ {\it admits a mutation in direction $k\in\Kex$}
if there exists a real simple object  $M_k' \in \shc_{\mu_k(D)_k}$ such that
\begin{enumerate}
\item[{\rm (i)}]
there exist exact sequences in $\shc$
\begin{align*}
& 0 \to q \sodot_{b_{ik} >0} M_i^{\snconv b_{ik}} \to q^{m_k} M_k \otimes M_k' \to
 \sodot_{b_{ik} <0} M_i^{\snconv (-b_{ik})} \to 0, \\
& 0 \to q \sodot_{b_{ik} <0} M_i^{\snconv(-b_{ik})} \to q^{m_k'} M_k' \otimes M_k \to
  \sodot_{b_{ik} >0} M_i^{\snconv b_{ik}} \to 0,
\end{align*}
where $m_k$ and $m'_k$ are given in~\eqref{eq:mm'}.
\item[{\rm (ii)}] $\mu_k(\Seed) \seteq\bigl(\{M_i\}_{i\neq k}\sqcup\{M_k'\}, \mu_k(\tB)\bigr)$ is a quantum monoidal seed  in $\shc$.
\end{enumerate}
We call $\mu_k(\Seed)$ the {\it mutation} of $\Seed$ in direction $k$.
\end{definition}

\begin{definition}
Assume that a $\cor$-linear abelian monoidal category $\shc$ satisfies the conditions {\rm \eqref{item: cond1}--\eqref{item: cond5}} in the beginning of
this subsection. The category $\shc$ is called a {\it monoidal categorification of a quantum cluster algebra $\mathscr{A}$ over $\Z[q^{\pm1/2}]$} if
\begin{enumerate}
\item[{\rm (i)}] $\Z[q^{\pm1/2}]\tens_{\Z[q^{\pm1}]} K(\shc)$ is isomorphic to $\mathscr{A}$,
\item[{\rm (ii)}] there exists a quantum monoidal seed $\Seed =(\{M_i\}_{i\in\K},\tB)$ in $\shc$ such that
$[\Seed]\seteq(\{q^{-(d_i,d_i)/4}[M_i]\}_{i\in\K}, L, \widetilde B)$ is a quantum seed of $\mathscr{A}$,
\item[{\rm (iii)}] $\Seed$ admits successive mutations in all the directions.
\end{enumerate}
\end{definition}

\section{Quiver Hecke algebras} \label{Sec: quiver Hecke}
\subsection{Quiver Hecke algebras}
Now we briefly recall the definition of quiver Hecke algebra associated to a symmetrizable Cartan datum $(\cmA,\wl,\Pi,\wl^\vee,\Pi^\vee)$ (see \cite{KL09,R08} for more detail).

Let $\cor$ be a base field. We take a family of polynomials $(Q_{i,j})_{i,j\in I}$ in $\cor[u,v]$ satisfying
\begin{align}\label{eq:Q}
Q_{i,j}(u,v) = \true(i \ne j) \times \sum\limits_{ \substack{ (p,q)\in \Z^2_{\ge0}
\\ (\al_i , \al_i)p+(\al_j , \al_j)q=-2(\al_i , \al_j)}}
t_{i,j;p,q} u^p v^q,
\end{align}
where $t_{i,j;p,q}=t_{j,i;q,p}$ and $t_{i,j:-a_{ij},0} \in \cor^{\times}$. Then one can check that $Q_{i,j}(u,v)=Q_{j,i}(v,u)$.

For $n \in \Z_{\ge 0}$ and $\beta \in \rl^+$ such that $|\be| = n$, we set
$$I^{\beta} = \{\nu = (\nu_1, \ldots, \nu_n) \in I^{n} \ | \ \alpha_{\nu_1} + \cdots + \alpha_{\nu_n} = \beta \}.$$

We denote by
$\sym_{n} = \langle s_1, \ldots, s_{n-1} \rangle$ the symmetric group
of degree $n$, where $s_i\seteq (i, i+1)$ is the transposition of $i$ and $i+1$.
Then $\sym_n$ acts on $I^\beta$ by place permutations.

\begin{definition}
For $\beta \in \rl^+$ with $|\beta|=n$, the {\em
quiver Hecke algebra}  $R(\beta)$  at $\beta$ associated
with a symmetrizable Cartan datum $(\cmA,\wl, \Pi,\wl^{\vee},\Pi^{\vee})$ and a matrix
$(Q_{i,j})_{i,j \in I}$ is the $\cor$-algebra generated by
the elements $\{ e(\nu) \}_{\nu \in  I^{\beta}}$, $ \{x_k \}_{1 \le
k \le n}$ and $\{ \tau_m \}_{1 \le m \le n-1}$ satisfying the following
defining relations$\colon$
\begin{align*}
& e(\nu) e(\nu') = \delta_{\nu, \nu'} e(\nu), \ \ \sum_{\nu \in  I^{\beta} } e(\nu) = 1,  \ \
 x_{k} x_{m} = x_{m} x_{k}, \ \ x_{k} e(\nu) = e(\nu) x_{k}, \allowdisplaybreaks\\
& \tau_{m} e(\nu) = e(s_{m}(\nu)) \tau_{m}, \ \ \tau_{k} \tau_{m} = \tau_{m} \tau_{k} \ \ \text{if} \ |k-m|>1, \ \ \tau_{k}^2 e(\nu) = Q_{\nu_{k}, \nu_{k+1}} (x_{k}, x_{k+1}) e(\nu), \allowdisplaybreaks\\
& (\tau_{k} x_{m} - x_{s_k(m)} \tau_{k}) e(\nu) = \begin{cases}
-e(\nu) \ \ & \text{if} \ m=k, \ \nu_{k} = \nu_{k+1}, \\
e(\nu) \ \ & \text{if} \ m=k+1, \  \nu_{k}=\nu_{k+1}, \\
0 \ \ & \text{otherwise},
\end{cases} \allowdisplaybreaks \\
& (\tau_{k+1} \tau_{k} \tau_{k+1}-\tau_{k} \tau_{k+1} \tau_{k}) e(\nu)
\\*
& \hspace{8ex}=\begin{cases} \dfrac{Q_{\nu_{k}, \nu_{k+1}}(x_{k},
x_{k+1}) - Q_{\nu_{k}, \nu_{k+1}}(x_{k+2}, x_{k+1})} {x_{k} -
x_{k+2}}e(\nu) \ \ & \text{if} \
\nu_{k} = \nu_{k+2}, \\
0 \ \ & \text{otherwise}.
\end{cases}
\end{align*}
\end{definition}

The above relations are homogeneous with
\begin{equation*} \label{eq:Z-grading}
\deg e(\nu) =0, \quad \deg\, x_{k} e(\nu) = (\alpha_{\nu_k}
, \alpha_{\nu_k}), \quad\deg\, \tau_{l} e(\nu) = -
(\alpha_{\nu_l} , \alpha_{\nu_{l+1}}),
\end{equation*}
and $R( \beta )$ is endowed with a $\Z$-graded algebra structure.

 For a graded $R(\beta)$-module $M=\bigoplus_{k \in \Z} M_k$, we define
$qM =\bigoplus_{k \in \Z} (qM)_k$, where
 \begin{align*}
 (qM)_k = M_{k-1} & \ (k \in \Z).
 \end{align*}
We call $q$ the \emph{grading shift functor} on the category of
graded $R(\beta)$-modules.

For graded $R(\beta)$-modules $M$ and $N $, $\Hom_{R(\beta)}(M,N)$ denotes the space of degree preserving module homomorphisms.
We set $ \deg(f) \seteq k$ for $f \in \Hom_{R(\beta)}(q^{k}M, N)$.

For an $R(\beta)$-module $M$, we set $\wt(M)\seteq -\beta \in \rl^-$
and call it the {\em weight} of $M$.

Let us denote
by $R(\beta)\gmod$ the category of graded $R(\beta)$-modules which are finite-dimensional over $\cor$. We set
$$R\gmod=\soplus_{\beta\in\rl^+}R(\beta)\gmod.$$

For $\beta, \gamma \in \rl^+$ with $|\beta|=m$, $|\gamma|= n$, we define an idempotent $e(\beta,\gamma)$ as follows:
$$e(\beta,\gamma)\seteq \displaystyle\sum_{\substack{\nu \in I^{\beta+\gamma}, \\ (\nu_1, \ldots ,\nu_m) \in I^{\beta}}} e(\nu) \in R(\beta+\gamma). $$
Then we have an injective ring homomorphism
$$R(\beta)\tens R(\gamma)
\monoto e(\beta,\gamma)R(\beta+\gamma)e(\beta,\gamma).$$

For an $R(\beta)$-module $M$ and an $R(\gamma)$-module $N$,
\begin{itemize}
\item the \emph{convolution product} $M\conv N$ is an $R(\be+\ga)$-module defined by
$$M\conv N=R(\beta + \gamma) e(\beta,\gamma) \tens_{R(\beta )\otimes R( \gamma)}(M\otimes N),$$
\item the dual space $M^* \seteq \Hom_{\cor}(M, \cor)$ admits an $R(\beta)$-module structure via
\begin{align*}
(r \cdot  f)(u) \seteq f(\psi(r) u) \quad (r \in R( \beta), \ u \in M),
\end{align*}
where $\psi$ denotes the $\cor$-algebra anti-involution on $R(\beta)$ fixing the generators.
\end{itemize}
We denote by $u\etens v$ the image of $u\tens v$ in $M\conv N$.

A simple module $M$ in $R \gmod$ is called \emph{self-dual} if $M^* \simeq M$.
Every simple module is isomorphic to a grading shift of a self-dual simple module (\cite[\S 3.2]{KL09}).

Then $R\gmod$ has a monoidal category structure with $\conv$ as a tensor product. 
Let us denote by $K(R\gmod)$ the  Grothendieck ring of $R\gmod$
which is an algebra over  $\Aa =\Z[q^{\pm 1}]$ with  the multiplication induced by the convolution product and the $\Aa$-action induced by the  grading shift functor $q$.

In~\cite{KL09,KL11,R08}, it is shown that a quiver Hecke algebra \emph{categorifies} the corresponding unipotent quantum coordinate ring. More precisely, we have the following theorem.

\begin{theorem}[{\cite{KL09,KL11,R08}}] \label{thm:categorification 1}
For a given symmetrizable Cartan datum $(\cmA,\wl, \Pi,\wl^{\vee},\Pi^{\vee})$, we take a parameter matrix $(Q_{ij})_{i,j \in I}$ satisfying the conditions in~\eqref{eq:Q}, and let $A_q(\n)$ and $R(\beta) $ be
the associated unipotent quantum coordinate ring and quiver Hecke algebra, respectively.
Then there exists an $\Aa$-algebra isomorphism
\begin{align} \label{eq:KLRU}
  \ch \colon K(R\gmod) \isoto  \Aan.
\end{align}
\end{theorem}

\begin{definition}
We say that a quiver Hecke algebra $R$ is {\em symmetric} if
$Q_{i,j}(u,v)$ is a polynomial in $u-v$ for all $i,j\in I$.
\end{definition}

In particular, the corresponding generalized Cartan matrix $\cmA$ is symmetric.
In symmetric case,  we assume $\mathsf{d}_i=1$ for all $i \in I$.

\begin{theorem}[\cite{R11,VV09}] \label{thm:categorification 2}
Assume that the quiver Hecke algebra $R$ is symmetric and the base field $\cor$ is of characteristic $0$. Then, under the isomorphism~\eqref{eq:KLRU}
in {\rm Theorem~\ref{thm:categorification 1}}, the upper global basis $\B^\upper(\Aqn)$  corresponds to the set of the isomorphism classes of self-dual simple $R$-modules.
\end{theorem}

\subsection{R-matrices}

For $\beta,\gamma\in \rl^+$ with $|\beta|=m$ and $|\gamma|=n$,  let
$M$ be an $R(\beta)$-module and $N$ an $R(\gamma)$-module. Then, by~\cite[Lemma 1.5]{KKK18A}, there exists an $R( \beta +\gamma)$-module homomorphism
(up to a grading shift)
\begin{align}\label{eq: R}
R_{M,N}\colon M\conv N\To  N\conv M,
\end{align}
which is defined by {\it intertwiners} $\varphi_k \in R(\be+\ga)$ $(1 \le k \le m+n-1)$ (see \cite[\S 1.3.1]{KKK18A}) and satisfies the Yang-Baxter equation (see~\cite[(1.9)]{KKK18A}).

{\it For the rest of this paper, we assume that quiver Hecke algebra is symmetric and $\mathsf{d}_i=1$ for all $i \in I$}.
{\it We also work always in the category of graded $R$-modules}.

Then each $R(\beta)$-module $M$ admits an {\it affinization} $$M_z \simeq \cor[z] \tens_\ko M$$
with the action of $R(\beta)$ twisted by the algebra homomorphism $\psi_z \colon R(\be) \to \cor[z] \tens R(\beta)$ (see \cite[\S 1.3.2]{KKK18A})
sending $x_k$ to $x_k+z$, where $z$ is an indeterminate of homogeneous degree $2$.

By~\cite[Proposition 1.10]{KKK18A}, the $R( \beta +\gamma)$-module homomorphism (up to a grading shift)
$$ R_{M_z,N}\colon M_z\conv N\To  N\conv M_z  $$
induces an $R( \beta +\gamma)$-module homomorphism
$$\rmat{M,N} \colon M \conv N \to q^{-\Lambda(M,N)}  N \conv M,$$
which also satisfies the Yang-Baxter equation and is non-zero provided that $M$ and $N$ are non-zero.
Here $\Lambda(M,N)$ denotes the degree of $\rmat{M,N} $ and we call $\rmat{M,N}$ the {\it $R$-matrix}.

\begin{definition}[{\cite{KKKO18}}]
For non-zero $R$-modules $M$ and $N$ in $R\gmod$, we define integers $\de(M,N)$ and $\tL(M,N)$ as follows$\colon$
\begin{enumerate}
\item[{\rm (a)}] $\de(M,N)= \dfrac{1}{2}\left( \Lambda(M,N)+\Lambda(N,M) \right) \in \Z_{\ge 0}$,
\item[{\rm (b)}] $\tL(M,N)= \dfrac{1}{2}\left( \Lambda(M,N)+\bigl(\wt(M),\wt(N)\bigr) \right) \in \Z_{\ge 0}$.
\end{enumerate}
\end{definition}

A simple module $M \in R\gmod$ is called \emph{real} if $M\conv M$ is simple.

\begin{lemma}[\cite{KKKO15}] \label{lem:commute_equiv}
 Let $M$ and $N$ be simple modules in $R\gmod$, and assume that one of them is real. Then
\begin{enumerate}
\item[{\rm (i)}] $M \conv N$ and $N \conv M$ have simple socles and simple heads.
\item[{\rm (ii)}] $\Im(\rmat{M,N})$ is equal to the head of $M \conv N$ and the socle of $N \conv M$.
\end{enumerate}
\end{lemma}

For $R$-modules $M$ and $N$, we denote by $M \hconv N$ the head of $M \conv N$ and by $M \sconv N$ the socle of $M \conv N$.

\begin{proposition}[{\cite[Corollary 3.7]{KKKO15}}] \label{prop: head inj}
For $\be,\gamma \in \rl^+$, let M be a real simple $R(\beta)$-module.
Then the map $N \mapsto M \hconv N$ is injective from the set of the isomorphism
classes of simple objects $N$
of $R(\gamma)\gmod$ to the set of the isomorphism classes of simple objects of
$R(\be+\gamma)\gmod$.
\end{proposition}

\begin{lemma}[{\cite[Lemma 3.1.4]{KKKO18}}] \label{lem: tLa} Let $M$ and $N$ be self-dual simple modules. If one of them is real, then
$$ q^{\tL(M,N)}M \hconv N \text{ is a self-dual simple module}.$$

\end{lemma}
Thus $\tL(M,N)$ indicates the degree shift that makes $M \hconv N$ self-dual.

\begin{definition}
For non-zero $R$-modules $M_1$ and $M_2$, we set
\begin{align*}
& M_1\sodot M_2\seteq q^{\tL(M_1,M_2)}M_1\conv M_2.
\end{align*}
\end{definition}

The non-negative integer $\de(M,N)$ measures the degree of complexity of $M\conv N$
as seen in the following lemma.

\begin{lemma}[{\cite{KKKO15,KKKO18}}]\label{lem:linked}
Let $M$ and $N$ be simple modules in $R \gmod$, and assume that one of them is real.

\begin{enumerate}
\item[{\rm (i)}] $M\conv N$ is simple if and only if $\de(M,N)=0$.
\item[{\rm (ii)}] If $\de(M,N)=1$, then $M\conv N$ has length $2$, and there exists an exact sequence
$$0\to M\sconv N\to M\conv N\to M\hconv N\to 0.$$
\end{enumerate}
\end{lemma}

\begin{definition}
For simple $R$-modules $M$ and $N$, we say that
\begin{enumerate}
\item[{\rm (i)}] $M$ and $N$ \emph{strongly commute} if $M\conv N$ is simple,
\item[{\rm (ii)}] $M$ and $N$ are \emph{simply-linked} if $\de(M,N)=1$.
\end{enumerate}
\end{definition}

\subsection{Determinantial modules}

The set of self-dual simple $R$-modules corresponds to
the upper global basis of $\Aqn$ by Theorem~\ref{thm:categorification 2}.

\smallskip

Recall the $\Aa$-algebra isomorphism $\ch$ in~\eqref{eq:KLRU}.

\begin{definition}[{\cite[\S\;9.1, \S\,10.2]{KKKO18}, 
\cite[Proposition~4.1]{KKOP18}}]
For $\La \in \wl^+$ and $\eta,\zeta \in W\La$ such that $\eta \preceq \zeta$, let $\M(\eta,\zeta)$ be the self-dual simple $R(\zeta -\eta)$-module such that
$\ch\bigl(\M(\eta,\zeta)\bigr)=\D(\eta,\zeta).$
\end{definition}

By Proposition~\ref{prop: multi D} {\rm (i)}, the module $\M(\eta,\zeta)$ is real. We call $\M(\eta,\zeta)$ the {\it determinantial module}.
We also write $\M_\tw(s,t)$ (see \eqref{eq: D_tw}) which is the determinantial module such that
\begin{align}\label{eq: M_tw(s,t)}
\D_\tw(s,t) = \ch(\M_\tw(s,t)).
\end{align}

\begin{theorem}[{\cite[Theorem 10.3.1]{KKKO18},
\cite[Proposition~4.6]{KKOP18}}] \label{thm: p+,p monoidal}
For $\La \in\wl^+$ and $\eta_1,\eta_2,\eta_3 \in W\La$ with $\eta_1 \preceq \eta_2\preceq \eta_3$, we have
$$ \M(\eta_1,\eta_2) \hconv \M(\eta_2,\eta_3) \simeq \M(\eta_1,\eta_3).$$
\end{theorem}

Now we recall the definition of several functors on $R$-modules.

\begin{definition} Let $\beta\in\rl^+$.
\begin{enumerate}
\item[{\rm (i)}] For $i \in I$ and $1 \le a \le |\beta|$, set
$$ e_a(i)=\sum_{\nu \in I^\beta,\nu_a=i}e(\nu) \ \  \text{ and } \ \ e^*_a(i)=\sum_{\nu \in I^\beta,\nu_{|\be|+1-a}=i}e(\nu) \in R(\beta).$$
\item[{\rm (ii)}] For $M \in R(\beta) \gmod$,
\begin{align*}
&E_iM \seteq  e_1(i)M  \quad \text{ and } \quad  E^*_iM=e^*_{1}(i)M,
\end{align*}
which are functors from $R(\beta) \gmod$ to $R(\beta-\al_i) \gmod$.
\item[{\rm (iii)}] Let $L(i)$ be the $1$-dimensional $R(\al_i)$-module $R(\al_i)/R(\al_i)x_1$. For a simple $R$-module $M$, we set
\begin{align*}
&\ve_i(M)=\max\{ n\in\Z_{\ge0}\ | \ E_i^nM\ne 0 \}, \allowdisplaybreaks\\
&\ve^*_i(M)=\max\{n\in\Z_{\ge0}\ | \ E_i^{*\;n}M\ne 0\},\allowdisplaybreaks\\
& \tF_iM = q^{\ve_i(M)} L(i) \hconv M,  \quad \quad \  \tF^*_iM = q^{\ve^*_i(M)} M \hconv L(i), \allowdisplaybreaks\\
& \tE_iM = q^{1-\ve_i(M)} \soc(E_iM), \quad   \tE^*_iM = q^{1-\ve^*_i(M)} \soc(E^*_iM).
\end{align*}
Here, for an $R$-module $N$, $\soc(N)$ denotes the socle of $N$ and $\hd(N)$ denotes the head of $N$.
\item[{\rm (iv)}] For $i \in I$ and $n \in \Z_{\ge 0}$, we set $L(i^n)=q^{n(n-1)/2} L(i)^{\conv n}$ which  is a self-dual real simple  $R(n\alpha_i)$-module.
\end{enumerate}
\end{definition}

\begin{proposition}[{\cite[Proposition 10.2.3]{KKKO18}}] \label{prop: weyl left right}
Let $\La \in \wl^+$, $\eta,\zeta \in W\La$
such that $\eta \preceq \zeta$ and $i \in I$.
\begin{enumerate}
\item[{\rm (i)}] If $n \seteq \lan h_i,\eta \ran \ge 0$, then
$$ \ve_i(\M(\eta,\zeta))=0 \  \text{ and } \  \M(s_i\eta,\zeta) \simeq \widetilde{F}_i^n\M(\eta,\zeta) \simeq L(i^n) \hconv \M(\eta,\zeta)
\ \text{in $R\gmod$.}$$
\item[{\rm (ii)}] If $\lan h_i,\eta \ran \le 0$ and $s_i\eta \preceq \zeta$, then we have $\ve_i(\M(\eta,\zeta))=-\lan h_i,\eta \ran$ and
$\M(s_i\eta,\zeta)\simeq \tE_i^{-\ang{h_i,\eta}}\M(\eta,\zeta)$.
\item[{\rm (iii)}] If $m \seteq -\lan h_i,\zeta \ran \ge 0$, then
$$ \ve^*_i(\M(\eta,\zeta))=0 \  \text{ and } \  \M(\eta,s_i\zeta) \simeq
\tF_i^{*\,m}\M(\eta,\zeta) \simeq \M(\eta,\zeta) \hconv L(i^m )
\ \text{ in $R\gmod$.}$$
\item[{\rm (iv)}] If $\lan h_i,\zeta \ran \ge 0$ and $\eta \preceq s_i\zeta$, then $\ve^*_i(\M(\eta,\zeta))=\lan h_i,\zeta \ran$
and $\M(\eta,s_i\zeta)\simeq \tE_i^{\ang{h_i,\zeta}}\M(\eta,\zeta)$.
\end{enumerate}
\end{proposition}

\subsection{Admissible pair and  $\mathscr{A}_{q^{1/2}}(\n(w))$} In this subsection, we review the results in~\cite{KKKO18} on the monoidal categorification by using graded modules over quiver Hecke algebras.  Throughout this subsection, we focus on a category $\shc$ which is a full subcategory of $R\gmod$ which is stable under taking convolution products, subquotients,
extensions, and grading shift. Then we have
$$ \shc = \soplus_{\be \in \rl^-} \shc_\be \qquad \text{ where } \shc_\be \seteq \shc \cap R(-\be)\gmod.$$

\begin{definition}[{cf. Definition~\ref{def:quantum monoidal pair}}] \label{def: admissible}
A pair $(\{M_i\}_{i\in\K},\tB)$ consisting of
\begin{enumerate}
\item[{\rm (i)}] a family of self-dual real simple objects $\{M_i\}_{i\in\K}$ of $\shc$ strongly commuting with each other,
\item[{\rm (ii)}] an integer valued $\K \times \Kex$-matrix $\tB$ satisfying~\eqref{eq: condition B},
\end{enumerate}
is called {\it admissible} if, for each $k\in\Kex$, there exists an object $M'_k$ of $\shc$  such that
\begin{enumerate}
\item[{\rm (a)}] there is an exact sequence in $\shc$
\begin{align} \label{eq:ses graded mutation KLR}
0 \to q \sodot_{b_{ik} >0} M_i^{\snconv b_{ik}} \to q^{\tL(M_k,M_k')} M_k \conv M_k' \to
 \sodot_{b_{ik} <0} M_i^{\snconv (-b_{ik})} \to 0,
\end{align}
\item[{\rm (b)}] $M_k'$ is self-dual simple and strongly commutes with $M_i$ for any  $i \neq k$.
\end{enumerate}
\end{definition}

For an admissible pair $(\{M_i\}_{i\in\K},\tB)$, we can take
\begin{itemize}
\item  $-\Lambda = (-\Lambda_{i,j})_{i,j \in \K}$  the skew-symmetric integer-valued matrix by $$\Lambda_{i,j}\seteq \Lambda(M_i,M_j),$$
\item $D=\{ d_i \}$ the family of elements in $\rl^-$ by $d_i=\wt(M_i)$,
\end{itemize}
as in Section~\ref{sec:monoiddal}.

\begin{proposition}[{\cite[Proposition 7.1.2]{KKKO18}}] \label{prop: KKKO Prop 7.1.2}
For an admissible pair $(\{M_i\}_{i\in\K},\widetilde B)$, we have the following properties$\colon$
\bnum
\item $\Seed\seteq (\{M_i\}_{i\in\K}, \tB)$ is a quantum monoidal seed in $\shc$. \label{item:1}
\item The self-dual simple object $M_k'$ is real for every $k\in\Kex$.
\item The quantum monoidal seed $\Seed$ admits a mutation in each direction $k\in\Kex$.\label{item:3}
\item $M_k$ and $M'_k$ is simply-linked for any $k\in\Kex$ $($i.e., $\de(M_k,M'_k)=1)$.\label{item:4}
\end{enumerate}
\end{proposition}

For an admissible pair $(\{M_i\}_{i\in\K},\tB)$, let us denote by
\begin{align*} 
[\Seed] \seteq(\{q^{-\frac{1}{4}(\wt(M_i), \wt(M_i))}[M_i]\}_{i\in\K}, -\Lambda,\tB).
\end{align*}

\begin{theorem}[{\cite[Theorem 7.1.3, Corollary 7.1.4]{KKKO18}}] \label{thm: monoidal}
For an admissible pair $(\{M_i\}_{i\in\K},\tB)$, we assume that
\begin{align*} 
\Ah \tens_{\Z[q^{\pm 1}]} K(\shc) \text{ is isomorphic to the quantum cluster algebra } \mathscr{A}_{q^{1/2}}([\Seed]).
\end{align*}
Then $\shc$ is a monoidal categorification of the quantum cluster algebra $\mathscr{A}_{q^{1/2}}([\Seed])$.

In particular, the following statements holds$\colon$
\begin{enumerate}
\item[{\rm (i)}] Each cluster monomial in $\mathscr{A}_{q^{1/2}}([\Seed])$ corresponds to the isomorphism class of a real simple object in $\shc$ up to a power of $q^{1/2}$.
\item[{\rm (ii)}] Each cluster monomial in $\mathscr{A}_{q^{1/2}}([\Seed])$ is a Laurent polynomial of the initial cluster variables with coefficient in $\Z_{\ge0}[q^{\pm1/2}]$.
\end{enumerate}
\end{theorem}

\begin{definition} \label{def:cw}
For $w \in W$, let $\shc_w$ be the smallest full subcategory of $R\gmod$ satisfying the following properties:
\begin{enumerate}
\item[{\rm (i)}] $\shc_w$ is stable by convolution, taking subquotients, extensions, and grading shifts,
\item[{\rm (ii)}] $\shc_w$ contains $\M_\tw(s,s_-)$ $(1 \le s \le \ell(w))$, where $\tw$ is a reduced expression of $w$.
\end{enumerate}
\end{definition}

By~\cite{GLS13}, we have an $\Aa$-algebra isomorphism
$$   K(\shc_w) \simeq \Aanw.$$

Let $\tB$ be the integer-valued $\K\times\Kex$-matrix associated to $\tw$ (see Definition~\ref{def: quiver Q assoc tw}).

\begin{theorem}[{\cite[Theorem 11.2.2, Theorem 11.2.3]{KKKO18}}] \label{thm: monoidal categorification}
The pair $(\{ \M_\tw(s,0) \}_{1 \le s \le r},\tB )$ is admissible.
Thus $\shc_w$ is a monoidal categorification of the quantum cluster algebra $\mathscr{A}_{q^{1/2}}(\n(w))$,
with the quantum monoidal seed $(\{ \M_\tw(s,0) \}_{1 \le s \le r},\tB )$.

Furthermore, the self-dual real simple $R$-module
$\M_\tw(k,0)'$ for $k \in \Kex$ in~\eqref{eq:ses graded mutation KLR} is given as follows$\col$ $($up to a grade shift$)$
\begin{align} \label{eq: prime}
\M_\tw(k,0)' \simeq \M_\tw(k_+,k) \hconv \Bigl(\ms{5mu}\sodot_{t < k <t_+ <k_+} \M_\tw(t,0)^{\snconv |a_{i_k,i_t}|}\Bigr).
\end{align}
\end{theorem}

\subsection{Generalized quantum affine Schur-Weyl duality functor} \label{Subsec: functors} In this subsection, we recall the
generalized quantum affine Schur-Weyl duality functor in~\cite{KKK18A}.

\smallskip

Let us assume that we are given an index set $J$ and a family $\{V_j\}_{j\in J}$ of quasi-good $U_q'(\g)$-modules.

We define a quiver $\Gamma^J$ associated with the pair $(J, \{V_j\}_{j\in J})$ as follows:
\begin{eqnarray} &&
\parbox{81ex}{
\begin{enumerate}
\item[{\rm (a)}] we take $J$ as the set of vertices,
\item[{\rm (b)}]  we put $d_{ij}$ many arrows from $i$ to $j$, where $d_{ij}$ denotes the order of zero of $d_{V_{i},V_{j}}(u)$ at $u=1$.
\end{enumerate}
}\label{eq: GammaJ}
\end{eqnarray}
Note that we have $d_{ij}d_{ji}=0$ for $i,j\in J$ (see \cite[Theorem 2.2]{KKK18A}).

We define a symmetric Cartan matrix $\cmA^J =(a^J_{ij})_{i,j\in J}$  by
$a^J_{ij}=2$ if $i=j$ and $a^J_{ij}=-d_{ij}-d_{ji}$ otherwise.
Then we choose a family of polynomial $\{Q_{i,j}(u,v)\}_{i,j\in J}$ satisfying~\eqref{eq:Q} with the form,
$$Q_{i,j}(u,v)=\pm(u-v)^{-a^J_{ij}}\quad \text{ for }i\ne j $$
for some choices of sign $\pm$.

Now we take a family $\{P_{i,j}(u,v)\}_{i,j \in J}$
of elements in $\cor[[u,v]]$  satisfying the following conditions$\colon$
\begin{eqnarray} &&
\parbox{81ex}{
\begin{enumerate}
\item[{\rm (a)}] $P_{i,i}(u,v)=1$ for $i \in J$.
\item[{\rm (b)}] The homomorphism $P_{ij}(u,v) \Rnorm_{V_i,V_j} (z_{V_i}, z_{V_j})$
has neither pole nor zero at $u=v=0$,
where   $\cor[z_{V_i}^{\pm 1}, z_{V_j}^{\pm 1}] \to\cor[[u,v]]$ is  given  by
 $z_{V_i} \mapsto 1+u $ and  $z_{V_j} \mapsto 1+v$.
\item[{\rm (c)}]  $ Q_{i,j}(u,v)=P_{i,j}(u,v)P_{j,i}(v,u)$ for $i \ne j$.
\end{enumerate}
}\label{eq:Pij}
\end{eqnarray}
We call such a family $\{P_{i,j}(u,v)\}_{i,j\in J}$ a {\it duality coefficient}.

Let $\{\al^J_i \ | \ i \in J\}$  be the set of simple roots associated to $\cmA^J$ and $\rl^+_J=\sum_{i \in J} \Z_{\ge 0} \al^J_i$ be the corresponding positive root lattice.
We define the symmetric bilinear form $(\ ,\ )$ satisfying $(\al^J_i,\al^J_j)=a^J_{ij}$.

Let us denote by $R^{J}( \beta )$  $(\beta \in \rl^+_J)$  the symmetric quiver Hecke algebra associated with $\cmA^J$ and $\{ Q_{i,j}(u,v)\}_{i,j\in J}$.

In \cite[\S 3]{KKK18A}, Kang-Kashiwara-Kim constructed a functor, called the {\it generalized quantum affine Schur-Weyl duality functor}
\begin{align}\label{eq: F}
\F\col\soplus_{\beta \in \rl^+_J} R^J( \beta) \gmod \to\Ca_\g,
\end{align}
which is a monoidal functor in the following sense: There exist canonical $U_q'(\g)$-isomorphisms
\begin{align*}
\F(R^J(0))\simeq \cor,  \quad  \F(M_1 \conv M_2) \simeq \F(M_1) \otimes \F(M_2)
\end{align*}
for any $M_1 ,M_2 \in R^J\gmod$ such that the diagrams in ~\cite[(A.2)]{KKK18A} are commutative.

\begin{proposition}[{\cite[Proposition 3.2.2]{KKK18A}}] \label{prop:image of tau}The monoidal functor $\F$ is a unique \ro up to an isomorphism\rf
functor which satisfies the following properties:
\bnum
\item For any $i \in J$, we have
\begin{align*}
\mathcal F(L(i)_z) \simeq \cor[[z]]  \tens_{\cor[z_{V_i}^{\pm 1}]} (V_i)_\aff,
\end{align*}
where $\cor[z_{V_i}^{\pm 1}] \to\cor[[z]]$ is given by $z_{V_i} \mapsto 1+z $.
In particular, we have
\begin{align*}
\F(L(i))\simeq V_i  \quad \text{for } i \in J.
\end{align*}

\item For $i,j\in J$, let
$R_{L(i)_z,L(j)_{z'}}\col L(i)_z \conv L(j)_{z'} \rightarrow L(j)_{z'} \conv L(i)_z$
be the $R^J(\al^J_i+\al^J_j)$-module homomorphism in~\eqref{eq: R}. Then we have
\begin{align} \label{eq:Fphi}
\F(R_{L(i)_z,L(j)_{z'}}) =P_{i,j}(z,z')\Rnorm_{V_i,V_j} (z_{V_i}, z_{V_j}) .
\end{align}
\end{enumerate}
\end{proposition}

Note that
the functor $\F$ depends on the choice of $\{P_{i,j}(u,v)\}_{i,j\in J}$ as seen in~\eqref{eq:Fphi}.

\begin{lemma}[{\cite[Lemma 1.7.8]{KKO18}}] \label{lem: non-simple tensor}
Let $M$, $N\in R^J \gmod$ be simple modules, and assume that one of them is real.
Assume that
\bna
\item the functor $\F$ in~\eqref{eq: F} is exact,
\item $\de(M,N)\le 1$,
\item $\F(M),\F(N) \neq 0$ and $\F(M)\otimes \F(N)$ is not simple.
\end{enumerate}
Then we have
\bnum
\item $\F(\rmat{M,N})= \rmat{\F(M),\F(N)}$ up to a non-zero constant multiple,
\item $\F(M\hconv N) \simeq \F(M)\hconv \F(N)$ which is simple.
\end{enumerate}
\end{lemma}

\begin{theorem}[{\cite[Theorem 3.8]{KKK18A}}] \label{thm:exact}
If the Cartan matrix $\cmA^J$ associated with  $\Gamma^J$ is of type
$A_n \ (n \ge 1), D_n \ (n \ge 4)$ or $E_n  \ (n = 6,7,8)$, then the functor $\F$ is exact.
\end{theorem}

\section{Cluster structure on $R^{A_\infty}\gmod$} \label{Sec: Module}
In this section,
we study the cluster structure on the monoidal category $R^{A_\infty}\gmod$.
Here Dynkin diagram of type $\Ainf$ is depicted as follows:
$$\xymatrix@R=0.5ex@C=4ex{
\ar@{.}[r] & *{\circ}<3pt> \ar@{-}[r]_>{-1} &*{\circ}<3pt>\ar@{-}[r]_>{ \ \ 0} &*{\circ}<3pt> \ar@{-}[r]_>{ \ \ 1} &*{\circ}<3pt> \ar@{-}[r] &  *{\circ}<3pt> \ar@{.}[r] &  }
$$
We first introduce an infinite sequence $\ttww$ of simple reflections whose first $p$-parts is reduced for every
$p \in \Z_{\ge 0}$. Then we explicitly compute determinantial modules of type $\Ainf$ which are associated to $\ttww$.
In the last part of this section, we will construct a certain quantum monoidal seed $\Seed_\infty$ and prove that the category $R^J\gmod$
of type $A_\infty$ gives a monoidal categorification
of the quantum cluster algebra $\Aan$ of type $\Ainf$
whose initial quantum seed is $[\Seed_\infty]$.

\subsection{Sequence of simple reflections of length $\infty$} Let $J=\Z$ be an index set.
We also take the weight lattice $\wl_J =\soplus_{i \in \Z}\Z \La_i$
with $(\La_i,\La_j)=-|i-j|/2$.
We set $\ep_a=\La_a-\La_{a-1}$ and $\al_j =\ep_j -\ep_{j+1}$.
We have $(\ep_a,\ep_b)=\delta_{a,b}$ and $(\Lambda_i,\alpha_j)=\delta_{i,j}$.
The matrix $A^J\seteq\bl (\al_i,\al_j)\br_{i,j\in J}$ is
a Cartan matrix of type $A_\infty$.
We write the root lattice $\rl_J = \soplus_{j \in J}\Z \al_j\subset \wl_J$.

Let $W_J$ be the Weyl group of type $\Ainf$ generated by the set of simple reflections $S_J\seteq\{ s_j \mid j \in J\}$.
Note that, for all $i,j \in J$, we have
\begin{align} \label{eq: computation}
s_i(\ep_j) = \ep_{s_i(j)} \quad  \text{ and } \quad  s_i(\Lambda_{j}) =
\begin{cases}
\Lambda_{j+1}-\ep_j =\La_{j-1}+\ep_{j+1} &\text{if } i=j, \\
\Lambda_{j}  &\text{otherwise}.
\end{cases}
\end{align}
For $k \in \Z_{> 0}$, we sometimes use the notation $\ov{k}$ to denote $-k \in \Z_{< 0} \subset J$.

A pair of integers $[a,b]$ with $a \le b$ is called a {\it segment}. The length of $[a,b]$ is defined to be the positive integer $b-a+1$.
For a segment $[a,b]$, we denote by $W_{[a,b]}$ the subgroup of
$W_J$ generated by $s_{k}$ for  $a \le k \le b$.

Recall the convention in~\eqref{eq: word} and~\eqref{eq: word conv} on sequences of simple reflections.
For each $t \in \Z_{\ge 0}$, let $\ww^{(t)}$ be the sequence
in $S_J$ of length $4t+3$ defined by$\colon$
\begin{align*}
\ww^{(t)} &= s_{\ov{t}}s_{\ov{t}+1}\cdots s_{\ov{1}}s_{0}s_1\cdots s_{t-1}s_{t}  \ s_{t+1} \ s_{t} s_{t-1}\cdots s_2s_1s_{0}s_{\ov{1}}\cdots s_{\ov{t}+1}s_{\ov{t}}.
\end{align*}
We set
$$ \ww^{(s,t)} \seteq
\begin{cases}
\ww^{(s)}*\ww^{(s+1)}* \cdots \ww^{(t-1)}*\ww^{(t)} & \text{ if } 0 \le s \le t, \\
\qquad\qquad\qquad \id  & \text{ otherwise,}
\end{cases}
$$
where $*$ denotes the concatenation of sequences.

Finally, we define an infinite sequence $\ttww$ of $S_J$ by$\colon$
\begin{equation} \label{eq: def tw}
\begin{aligned}
\ttww \seteq \lim_{t \to \infty} \ww^{(0,t)} & = \ww^{(0)} * \ww^{(1)} * \ww^{(2)} * \cdots \\
& = s_{0}s_{1}s_{0}\ms{6mu}s_{\ov{1}}s_{0}s_{1}s_{2}s_{1}s_{0}s_{\ov{1}}\ms{6mu}
s_{\ov{2}}s_{\ov{1}}s_{0}s_{1}s_{2}s_3s_2s_{1}s_{0}s_{\ov{1}}s_{\ov{2}} \; \cdots   \cdots.
\end{aligned}
\end{equation}

We define $\jj_p\in J$ ($p\in\Z_{\ge1}$) by
\eqn
\ttww=s_{\jj_1}s_{\jj_2}s_{\jj_3}\cdots\cdots.
\eneqn

For $t \in \frac{1}{2} \Z$, we set
\begin{align*} 
a(t) \seteq t(2t+1) \in \Z.
\end{align*}
Note that, for $t \in \Z_{\ge 0}$, $a(t)$ coincides with the length of
$\ww^{(0,t-1)}$.

We have for $t\in\Z_{\ge0}$
\eq
\hs{5ex}\jj_p=\left\{\ba{l}
p-a(t)-t-1=p-a(t+1/2)+t\\
\hs{35ex}\text{if $a(t)+1\le p\le a(t+1/2)+1$,}\\[1ex]
t+2+a(t+1/2)-p=-t+a(t+1)-p\\
\hs{35ex}\text{if $a(t+1/2)+1\le p\le a(t+1)+1$.}
\ea\right.\label{def:j}
\eneq

For each $p\in\Z_{\ge1}$, we will denote the element in the Weyl group $W_J$ defined by the sequence $\ttww_{\le p}$ by the same symbol.

\Prop \label{prop: reduced}
The sequence $\ttww$ has the following properties$\colon$
\bnum
\item For $t\in\Z_{\ge0}$, we have
$$\ttww_{\le p-1}(\al_{\jj_p})=
\left\{\ba{l}
\ep_{-t}-\ep_{t+2-p+a(t)}=\ep_{-t}-\ep_{1-t+a(t+1/2)-p}\\[1ex]
\hs{30ex}\text{if $a(t)+1\le p\le a(t+1/2)$,}\\[1.5ex]
\ep_{-t-1+p-a(t+1/2)}-\ep_{t+2}=\ep_{t+1+p-a(t+1)}-\ep_{t+2}\\[1ex]
\hs{30ex}\text{if $a(t+1/2)+1\le p\le a(t+1)$.}
\ea
\right.
$$
\item For any $p \in \Z_{\ge 1}$, $\ttww_{\le p}$ is reduced.
\ee
\enprop

\begin{proof}
Let us first show (i).
Note that, for each $t \in \Z_{\ge 0}$, we have
\begin{enumerate}
\item[{\rm (a)}] $\jj_{a(t)+1}=-t$  and $\jj_{p} > -t $ for all $p \le a(t)$,
\item[{\rm (b)}] $\jj_{a(t+\frac{1}{2})+1}= t+1$ and $\jj_{p} < t+1$ for all $p < a(t+\frac{1}{2})+1$.
\end{enumerate}
Assume that $a(t)+1 \le p \le a(t+\frac{1}{2})$. Then we have
\begin{align*}
\ttww_{\le p-1}(\al_{\jj_p})=\ttww_{\le a(t)}s_{-t}\cdots s_{\jj_p-1}(\ep_{\jj_p}-\ep_{\jj_{p+1}})
=\ttww_{\le a(t)}(\ep_{-t}-\ep_{\jj_p+1}).
\end{align*}
Note that $\ttww_{\le a(t)}$ is contained in $W_{[-t+1,t]}$ as a Weyl group element.
Thus, for all $a(t)+1 \le p \le a(t+\frac{1}{2})$, we have
$$
\ttww_{\le p-1}(\al_{\jj_p})=\ep_{-t}-\ep_{\ttww_{\le a(t)}(\jj_p+1)}=\ep_{-t}-\ep_{\ttww_{\le a(t)}(-t+p-a(t))}.
$$

Now we claim that
$$\ttww_{\le a(t)}(-t+p-a(t)) = t+2-(p-a(t)).$$
Note that \begin{itemize}
\item  $\ttww_{\le a(t)} = \ww^{(0,t-1)} = \ww^{(0,t-2)} * \ww^{(t-1)}$,
\item $\ww^{(t-1)}(-t+1)=t+1$ and $\ww^{(t-1)}(t+1)=-t+1$,
\item $\ww^{(t-1)}(k)=k$ for all $k \in \Z \setminus \{t+1,-t+1\}$.
\end{itemize}
Thus the claim follows from an induction on $t$. Hence we have
$\ttww_{\le p-1}(\al_{\jj_p})=\ep_{-t}-\ep_{t+2-p+a(t)}$.

The case when $a(t+\frac{1}{2})+1 \le p \le a(t+1)$ can be proved in a similar way.

\smallskip\noi
(ii) follows from (i) since $\ttww_{\le p-1}\al_{\jj_p}$ is a positive root for any $p\in\Z_{\ge1}$.
\end{proof}

\begin{remark} \label{rem: longest} \hfill
\begin{enumerate}[{\rm (a)}]
\item For each $k \in \frac{1}{2}\Z_{\ge 1}$, $a(k)$ coincides with the length of the longest element of Weyl group of type $A_{2k}$. Thus
$\ttww_{\le a(k)}$ is the longest element of
$$
\begin{cases}
W_{[-k+1,k]} & \text{ if } k \in \Z_{\ge 1}, \\
W_{[ -k +\frac{1}{2}  , k - \frac{1}{2}  ]} &\text{if $k\in\frac{1}{2}+\Z_{\ge0}$.}
\end{cases}
$$
\item For $t >0$, the reduced expressions $\tww^{(0,t)}$  are not {\it adapted} (\cite{B99}) in the sense that there exists no Dynkin quiver $\mathfrak{Q}$ of type $A_{2t+2}$ satisfying
$$    \text{$i_k$ is a sink of $s_{i_{k-1}}s_{i_{k-1}}\cdots s_{i_1}(\mathfrak{Q})$} \quad \text{ for all } k.$$
Here $s_i(\mathfrak{Q})$ denotes the quiver obtained by reversing all arrows incident with $i$.
\end{enumerate}
\end{remark}

\begin{definition} \label{def: coordinate}
For each $p \in \Z_{\ge 1}$, we assign a pair $c(p)=(\ell,m)$ of positive integers in the following way$\colon$
\begin{align} \label{eq: ell m}
(\ell,m) \seteq \begin{cases}
 \bigl(p-a(t), a(t+\frac{1}{2})-p+1\bigr) & \text{ if } a(t) <  p \le a(t+\frac{1}{2}), \\
\qquad\qquad \ (2t+2,1) & \text{ if } p =a(t+\frac{1}{2})+1, \\
\bigl( a(t+1)+1-p ,p-a(t+\frac{1}{2}) \bigr) & \text{ if } a(t+\frac{1}{2})+2 \le p \le a(t+1),
\end{cases}
\end{align}
where $t$ is a unique non-negative integer satisfying
\begin{align} \label{eq: t}
a(t) < p \le a(t+1).
\end{align}
\end{definition}

\begin{remark} \label{rem: coordinate}
The map $c\colon \Z_{\ge 1} \to \Z_{\ge 1}\times \Z_{\ge 1}$ in~\eqref{eq: ell m} satisfies the following properties$\colon$
\begin{itemize}
\item $\ell+m = \begin{cases} 2t+2 \equiv 0 \mod 2 & \text{ if } a(t) <  p \le a(t+\frac{1}{2}), \\
 2t+3 \equiv 1 \mod 2 & \text{ if } a(t+\frac{1}{2}) <  p \le a(t+1).
 \end{cases}$
\item $c$ is a bijective map whose inverse is given as follows$\colon$
\begin{align}\label{eq: c inverse}
c^{-1}(\ell,m) = \begin{cases}
a(\frac{\ell+m-2}{2})+\ell \ & \text{ if } \ell+m \equiv 0 \mod 2, \\
a(\frac{\ell+m-2}{2})+ m   \ & \text{ if } \ell+m \equiv 1 \mod 2.
\end{cases}
\end{align}
\item The integer $t$ in~\eqref{eq: t} is equal to
$\lfloor (\ell+m-2)/2 \rfloor=\lceil (\ell+m-3)/2\rceil$.
\end{itemize}
\end{remark}

Using the lattice point $\Z_{\ge 1} \times \Z_{\ge 1}$, the pairs $(p,\jj_p)$ corresponding to $c(p)=(\ell,m)$ can be exhibited as follows:
\begin{align*} 
\raisebox{6em}{\scalebox{0.78}{\xymatrix@C=2ex@R=1.8ex{
\vdots &  \vdots & \\
5& (11,-2) & \cdots\\
4& (10,-1) & (12,-1) & \cdots\\
3& (4,-1) & (9,0)& (13,0)& \cdots\\
2& (3,0) & (5,0) & (8,1) & (14,1) & \cdots\\
1& (1,0) & (2,1) & (6,1), & (7,2) & (15,2) & \cdots \\
m / \ell & 1 & 2 & 3 & 4 & 5 & \cdots}}}
\end{align*}

\begin{lemma} \label{lem: ell,m ip}
For $p \in \Z_{\ge 1}$ with $c(p)=(\ell,m)$, the index $\jj_p$ is given as follows$\colon$
$$
\jj_p =\left\lfloor(\ell-m+1)/2\right\rfloor =
\left\lceil(\ell-m)/2\right\rceil
=\begin{cases}
(\ell-m)/{2}& \text{if $\ell+m\equiv0\mod 2$,}\\[1ex]
({\ell-m+1})/{2}& \text{if $\ell+m\equiv1\mod 1$.}
\end{cases}
$$
\end{lemma}

\begin{proof}
It is enough to show that
$\ell-m-2\jj_p$ is $0$ or $-1$.
Assume that $a(t)<p\le a(t+1/2)$ for $t\in\Z_{\ge0}$.
Then by \eqref{def:j} and \eqref{eq: ell m}, we have
\eqn
\ell-m-2\jj_p&&=
\bl p-a(t)\br-\bl a(t+1/2)-p+1\br\\
&&\hs{10ex}-\bl p-a(t)-t-1\br-\bl p-a(t+1/2)+t\br=0.
\eneqn
If  $a(t+1/2)+2\ge p\le a(t+1)$, then we have
\eqn
\ell-m-2\jj_p&&=
\bl a(t+1)+1-p\br-\bl p-a(t+1/2)\br\\
&&\hs{10ex}-\bl t+2+a(t+1/2)-p\br-\bl -t+a(t+1)-p\br=-1.
\eneqn
If  $p=a(t+1/2)+1$, then we have
\eqn
\ell-m-2\jj_p&&=
\bl 2t+2\br-1
-2\bl t+2+a(t+1/2)-p\br
=-1. \qedhere
\eneqn
\end{proof}

\subsection{Determinantial modules of type $\Ainf$} \label{subsec: Ainfty}
We keep the notations in the previous subsection. For $i,j \in J$, let us set
$$
Q^J_{i,j}(u,v)= \begin{cases}
\pm(u-v) & \text{ if } j=i\pm 1, \\
\quad 0 & \text{ if } i=j, \\
\quad 1 & \text{ otherwise}.
\end{cases}
$$

We denote by $R^J$ the quiver Hecke algebra of type $\Ainf$ which is associated to $(Q^J_{i,j})_{i,j \in J}$.
{\it In the sequel, we sometimes drop the subscript $J$, if there is no danger of confusion.}

A {\it multisegment} is a finite sequence of segments. We assign a total order on the set of segments as follows$\colon$
$$ [a_1,b_1] > [a_2,b_2]  \quad \text{ if } a_1>a_2, \text{ or } a_1=a_2 \text{ and } b_1>b_2.$$

If a multisegment $\bigl( [a_1,b_1],\ldots,[a_r,b_r] \bigr)$ satisfies $[a_k,b_k] \ge [a_{k+1},b_{k+1}]$ for each $k$, we call it an {\it ordered multisegment}.

For $j \in J$ and a pair of positive integer $(\ell,m) \in \Z_{\ge 1} \times \Z_{\ge 1}$, we define an ordered multisegment $\lb \ell,m \rb^{(j)}$ as follows$\colon$
$$\lb \ell,m \rb^{(j)} \seteq   \bigl(\underbrace{[j-\ell+m,j+m-1], \ldots,[j-\ell+2,j+1],[j-\ell+1,j]}_{\text{$m$-times}}\bigr).$$

For each segment $[a,b]$ of length $\ell$, there exists a graded $1$-dimensional $R(\ep_a-\ep_{b+1})$-module $L[a,b] =\cor u[a,b]$ which is generated by a vector $u[a,b]$ of degree $0$. The action
of $R(\ep_a-\ep_{b+1})$ is given as follows (see~\cite[Lemma 1.16]{KKK18A})$\colon$
\begin{align*}
x_ku[a,b] & = 0, \ \ \tau_mu[a,b]=0, \ \ e(\nu) u[a,b] = \begin{cases}
u[a,b] & \text{ if } \nu=(a,a+1,\ldots,b),\\
\ \ \ 0 & \text{ otherwise}.
\end{cases}
\end{align*} Then one can check that
$$  L[a,b] \simeq \tF_{a}\tF_{a+1}\cdots \tF_{b} \cdot \one  \simeq  \tF^*_{b}\cdots \tF^*_{a+1}\tF^*_{a} \cdot \one,$$
where $\one$ is the trivial $R(0)$-module.

\begin{proposition}[{ \cite[Theorem 7.2 and Section 8.4]{KR11}}]  \label{prop:multisegments}\hfill
\begin{enumerate}[{\rm (i)}]
\item Let $M$ be a simple module in $R^J(\beta) \gmod$ with $\beta\in\rl_J^+$.
Then there exists a unique pair of an ordered multisegment $\big([a_1,b_1], \ldots , [a_t,b_t] \big)$ and $c \in \Z$ such that
$$ M \simeq q^c\hd \big(L[a_1,b_1]\conv \cdots \conv L[a_t,b_t] \big),$$  where $\hd$ denotes the head.
\item For an ordered multisegment   $\big ([a_1,b_1], \ldots , [a_t,b_t] \big)$,
$$\hd \big(L[a_1,b_1]\conv \cdots \conv L[a_t,b_t] \big)$$ is
a
simple $R^J(\beta)$-module, where $\be=\sum_{k=1}^t(\ep_{a_k}-\ep_{b_k+1}) $.
\end{enumerate}
\end{proposition}
We call the ordered multisegment $\big ([a_1,b_1], \ldots , [a_t,b_t] \big)$ in Proposition~\ref{prop:multisegments} {\rm (i)} {\it the multisegment associated with $M$}.
It is uniquely determined by $M$.

For $j \in J$ and a pair of positive integer $(\ell,m) \in \Z_{\ge 1} \times \Z_{\ge 1}$, we define a self-dual simple $R^J$-module $\W^{(\ell)}_{m,j}$ associated to $\lb \ell,m \rb^{(j)}$ as follows$\colon$ (up to a grading shift)
\begin{align} \label{eq: def Wlmj}
\W^{(\ell)}_{m,j} \simeq \hd( \underbrace{L[j\hspace{-.15ex}-\hspace{-.15ex}\ell\hspace{-.15ex}+\hspace{-.15ex}m,j\hspace{-.15ex}+\hspace{-.15ex}m\hspace{-.15ex}-\hspace{-.15ex}1]\conv \cdots \conv
L[j\hspace{-.15ex}-\hspace{-.15ex}\ell\hspace{-.15ex}+\hspace{-.15ex}2,j\hspace{-.15ex}+\hspace{-.15ex}1] \conv L[j\hspace{-.15ex}-\hspace{-.15ex}\ell\hspace{-.15ex}+\hspace{-.15ex}1,j]}_{\text{$m$-times}}).
\end{align}

\begin{remark}
The $R^J$-module $\W^{(\ell)}_{m,j}$ is known to be 
{\em homogeneous }in the sense that its grading is concentrated in a single degree (\cite{KR10}).
\end{remark}

Since $\ttww_{\le p}$ is reduced for any $p \in \Z_{\ge 1}$, we can define $\M_\ttww(p_1,p_2)$ for $0 \le p_2 \le p_1$ (see \eqref{eq: D_tw} and \eqref{eq: M_tw(s,t)}):
$$ \M_\ttww(p_1,p_2) \seteq \M_{\ttww_{\le p_1}}(p_1,p_2).$$

Note that, for any $w \in W_J$ and $\Lambda_j$ $(j \in J)$, we have
\begin{align*}
-1 \le \lan h_k, w \Lambda_j \ran \le 1  \quad \text{ for any $k \in J$}.
\end{align*}

Let us find the multisegment corresponding to $\M_\ttww(p,0)$.

For $j \in J$ and $(\ell,m) \in \Z_{\ge 1}\times  \Z_{\ge 1}$, we shall define a sequence $\tw_{\lb \ell,m \rb^{(j)}}$ in $S_J$ which is reduced
(as seen in Lemma~\ref{lem:PBW} below),
and arises from the ordered multisegment $\lb \ell,m \rb^{(j)}$ as follows$\colon$
\begin{itemize}
\item $\tw_{[a,b]} \seteq s_{a}s_{a+1} \cdots  s_{b-1}s_b$ for $a \le b$.
\item $\tw_{\lb \ell,m \rb^{(j)}} = \underbrace{\tw_{[j-\ell+m,j+m-1]} \cdots\cdots \tw_{[j-\ell+2,j+1]}\;\tw_{[j-\ell+1,j]}}_{\text{$m$-times}}$.
\end{itemize}

\begin{proposition} \label{prop: KR quiver redex}
For every $p \in \Z_{\ge 1}$,
$$  \ttww_{\le p} \Lambda_{\jj_p} = \tw_{\lb \ell,m \rb^{(\jj_p)}} \Lambda_{\jj_p}, \quad \text{ where } \  \ c(p)=(\ell,m).$$
\end{proposition}

\begin{proof}
We shall use the induction on $\ell+m$.
Let $t$ be a unique non-negative integer in~\eqref{eq: t}.

\smallskip\noi
(a) Assume that $a(t) < p \le a(t+\frac{1}{2})$. In this case, we have
\begin{itemize}
\item $\ell+m \equiv 0 \mod 2$, by Remark~\ref{rem: coordinate},
\item $\jj_p +t+1 =\ell$ since $\jj_{a(t)+1}=-t$ and $\jj_{a(t)+u}=-t+u-1$ for $1 \le u \le a(t+\frac{1}{2})-a(t)+1$.
\end{itemize}
Then we have
$$  \ttww_{\le p}= \tww^{(0,t-1)}  s_{\ov{t}}s_{\ov{t}+1} \cdots    s_{\jj_p-1}  s_{\jj_p}
=\tww^{(0,t-1)} *\tw_{[-t,\jj_p]}.$$
Using~\eqref{eq: computation}, we have
\begin{equation} \label{eq: step 1 for a}
\begin{aligned}
\tw_{[-t,\jj_p]}\Lambda_{\jj_p} & =s_{\ov{t}}s_{\ov{t}+1} \cdots    s_{\jj_p-1}  s_{\jj_p}\Lambda_{\jj_p} = \Lambda_{-t-1}+\sum_{k=-t+1}^{\jj_p+1}\ep_{k} =\Lambda_{\jj_p+1}-\ep_{-t}.
\end{aligned}
\end{equation}

Note that
$$ s_{k} (\Lambda_{\jj_p+1}-\ep_{-t}) = \Lambda_{\jj_p+1}-\ep_{-t} \quad \text{ for } \  -t+1 \le k \le \jj_p.$$
Thus we have
$$\tww^{(0,t-1)}(\Lambda_{\jj_p+1}-\ep_{-t}) =  \ttww_{\le p^-(\jj_p+1)}(\Lambda_{\jj_p+1}-\ep_{-t}),$$
where $c(p^-(\jj_p+1))=(\ell,m-1)$ by Lemma~\ref{lem: ell,m ip}.
(If $m=1$, i.e. $p=a(t+1/2)$, then $p^-(\jj_p+1)=0$ and $\ttww_{\le p^-(\jj_p+1)}=1$.)

Note that,  $\tww^{(0,t-1)} \in W_{[-t+1,t]}$ and
$W_{[-t+1,t]}$ fixes $\ep_{-t}$.
Hence we have
$$\ttww_{\le p}  \La_{\jj_p} =\ttww_{\le p^-(\jj_p+1)}\Lambda_{\jj_p+1}-\ep_{-t}.$$

By the  induction hypothesis, we have
$$\ttww_{\le p^-(\jj_p+1)}\Lambda_{\jj_p+1} =  \tw_{\lb \ell,m-1 \rb^{(\jj_{p}+1)}}\Lambda_{\jj_p+1},$$
and one can check that the following two reduced expressions coincide
$$\tw_{\lb  \ell,m-1  \rb^{(\jj_{p}+1)} } *\tw_{[-t,\jj_p]} = \tw_{\lb  \ell,m  \rb^{(\jj_{p})}},$$
where $*$ denotes the concatenation.
Note that $\ttww_{\le p^-(\jj_p+1)}\in W_{[-t+1,t]}$ also fixes $\ep_{-t}$.
Then our assertion follows from~\eqref{eq: step 1 for a}.

(b) Assume that $a(t+\frac{1}{2}) < p \le a(t+1)$. Then one can check that
\begin{equation} \label{eq: step 1 for b}
\begin{aligned}
 \ttww_{\le p} \Lambda_{\jj_p} & = \tww^{(0,t-1)}    s_{\ov{t}}s_{\ov{t}+1}  \cdots s_{\jj_p-1}s_{\jj_p}  \cdots  s_{t} \ s_{t+1}s_{t} \cdots    s_{\jj_p+1}  s_{\jj_p}\Lambda_{\jj_p} \\
&  = \tww^{(0,t-1)} s_{\ov{t}}s_{\ov{t}+1}  \cdots s_{\jj_p-1}s_{\jj_p}  \cdots  s_{t} \left (\Lambda_{\jj_p-1}+ \ep_{t+2} \right) \\
&  = \tww^{(0,t-1)} s_{\ov{t}}s_{\ov{t}+1}  \cdots s_{\jj_p-1} \left (\Lambda_{\jj_p-1}+ \ep_{t+2} \right) \\
& =  \ttww_{\le p^-(\jj_p-1)} \left (\Lambda_{\jj_p-1}+ \ep_{t+2} \right),
\end{aligned}
\end{equation}
as in the previous case. Note that
\begin{itemize}
\item $\ell+m \equiv 1 \mod 2$ by Remark~\ref{rem: coordinate},
\item $t+2-\jj_p=m$ and $c(p^-(\jj_p-1))=(\ell-1,m)$, by Lemma~\ref{lem: ell,m ip},
\item $\ttww_{\le p^-(\jj_p-1)} \in W_{[t,-t]}$,
\end{itemize}
By the induction hypothesis, we have
$$     \ttww_{\le p^-(\jj_p-1)}\Lambda_{\jj_p-1} =  \tw_{\lb \ell-1,m \rb^{(\jj_{p}-1)}}\Lambda_{\jj_p-1},$$
and one can check that the following two reduced expressions are
in the same equivalence class
$$ \tw_{\lb  \ell-1,m  \rb^{(\jj_{p}-1)}} * \ s_{t+1}s_{t} \cdots    s_{\jj_p+1}  s_{\jj_p} \equiv\tw_{\lb  \ell,m  \rb^{(\jj_{p})}}$$
with respect to the commutation relation $s_is_j \equiv s_js_i$ $(|i-j|>1)$. Then our assertion follows from~\eqref{eq: step 1 for b}.
\end{proof}

\Lemma\label{lem:PBW}
For $j \in J$ and $(\ell,m) \in \Z_{\ge 1} \times \Z_{\ge 1}$, we have
\bnum
\item $\tw_{\lb \ell,m \rb^{(j)}} $ is reduced,
\item
$\tw_{\lb \ell,m \rb^{(j)}} \Lambda_{j}=\La_{j+m}-\sum_{k=1}^m\ep_{j-\ell+k}$,
\item
$\M( \tw_{\lb \ell,m \rb^{(j)}} \Lambda_{j},\La_j)\simeq \W_{m,j}^{(\ell)}$. \label{item:KR}
\ee
\enlemma

\begin{proof}
Let us prove (i) and (ii) by induction on $m$.
Write $\tw\seteq\tw_{\lb \ell,m \rb^{(\jj_p)}}= s_{i_1}s_{i_2} \cdots s_{i_r} $.
Now we shall prove that
\begin{align}\label{eq: claim}
\lan h_{i_s}, \tw_{\le s-1}\La_{\jj_p}  \ran = 1 \quad \text{ for all } 1 \le s \le r.
\end{align}
By the induction hypothesis, we have
$$\tw_{\lb \ell,m-1 \rb^{(j)}}(\La_{j}) = \La_{j+m-1}-\sum_{k=1}^{m-1}\epsilon_{j-\ell+k}.$$
Hence, we have
\begin{align*}
\tw_{\lb \ell,m \rb^{(j)}}(\La_{j} ) &=   s_{j-\ell+m}\cdots s_{j+m-1}\tw_{\lb \ell,m-1 \rb^{(j)}}
(\La_{j})  \\
&=s_{j-\ell+m}\cdots s_{j+m-1}(\La_{j+m-1}-\sum_{k=1}^{m-1}\epsilon_{j-\ell+k}).
\end{align*}
Since
$$s_{j-\ell+m}\cdots s_{j+m-1}(\La_{j+m-1})
=s_{j-\ell+m}\cdots s_{j+m-2}(\La_{j+m}-\ep_{j+m-1})
=\La_{j+m}-\ep_{j-\ell+m}$$
and
\eqn
\ang{h_{j+m-k},\;s_{j+m-k+1}\cdots s_{j+m-1}\tw_{\lb \ell,m-1 \rb^{(j)}} \Lambda_{j}}=1
\qt{for $1\le k\le \ell$,}\label{eq:wred}
\eneqn
(ii) and \eqref{eq: claim}  follow.
(i) follows from \eqref{eq: claim}.

\smallskip
\noi(iii)
Let us prove (iii) by induction on $m$.
By Theorem~\ref{thm: p+,p monoidal}, we have
\eqn
\M(\tw_{\lb \ell,m \rb^{(j)}} \Lambda_{j},\La_j)
&&\simeq\M(\tw_{\lb \ell,m \rb^{(j)}} \Lambda_{j},\tw_{\lb \ell,m-1 \rb^{(j)}} \Lambda_{j})\hconv
\M(\tw_{\lb \ell,m-1 \rb^{(j)}} \Lambda_{j},\La_j)\\
&&\simeq
\M(\tw_{\lb \ell,m \rb^{(j)}} \Lambda_{j},\tw_{\lb \ell,m-1 \rb^{(j)}} \Lambda_{j})\hconv
\W_{m-1,j}^{(\ell)}.
\eneqn
Hence it is enough to show
\eqn
&&\M(\tw_{\lb \ell,m \rb^{(j)}} \Lambda_{j},\tw_{\lb \ell,m-1 \rb^{(j)}} \Lambda_{j})\simeq L[j-\ell+m,j+m-1].
\eneqn

Then Proposition~\ref{prop: weyl left right} implies
\[\M(\tw_{\lb \ell,m \rb^{(j)}} \Lambda_{j},\tw_{\lb \ell,m-1 \rb^{(j)}} \Lambda_{j})\simeq \tF_{j-\ell+m}\cdots\tF_{j+m-1}\one
\simeq L[j-\ell+m,j+m-1]. \qedhere \]
\end{proof}

\begin{theorem} \label{thm :KR in KLR}
For $p \in \Z_{\ge 1}$ with $c(p)=(\ell,m)$, we have
$$\M_\ttww(p,0) = \W_{m,\jj_p}^{(\ell)}.$$
\end{theorem}

\begin{proof}
The assertion immediately follows from Proposition~\ref{prop: KR quiver redex}
and Lemma~\ref{lem:PBW}.
\end{proof}

Using the lattice points $\Z_{\ge 1} \times \Z_{\ge 1}$, we can exhibit $\{ \M_\ttww(p,0) \simeq \W_{m,i_p}^{(\ell)} \}$ as follows$\colon$
\begin{align*}
\raisebox{4em}{\scalebox{0.64}{\xymatrix@C=2ex@R=3ex{
\vdots &\vdots & \vdots &\vdots &\vdots  & \iddots\\
\Wlmj{1}{3}{-1} & \Wlmj{2}{3}{0} & \Wlmj{3}{3}{0} & \Wlmj{4}{3}{1} & \Wlmj{5}{3}{1} & \cdots \\
\Wlmj{1}{2}{0} & \Wlmj{2}{2}{0} & \Wlmj{3}{2}{1} & \Wlmj{4}{2}{1} & \Wlmj{5}{2}{2} & \cdots \\
\Wlmj{1}{1}{0} & \Wlmj{2}{1}{1} & \Wlmj{3}{1}{1} & \Wlmj{4}{1}{2} & \Wlmj{5}{1}{2}  & \cdots }}}
 \hspace{-1ex} =
\raisebox{4em}{\scalebox{0.64   }{\xymatrix@C=2ex@R=3.9ex{
\vdots &\vdots & \vdots &\vdots &\vdots  & \iddots\\
\Mkl{4}{0} & \Mkl{9}{0} &\Mkl{13}{0} &\Mkl{18}{0} &\Mkl{27}{0} &   \cdots\\
\Mkl{3}{0} & \Mkl{5}{0} &\Mkl{8}{0} &\Mkl{14}{0} &\Mkl{17}{0} &   \cdots\\
\Mkl{1}{0} & \Mkl{2}{0} &\Mkl{6}{0} &\Mkl{7}{0} &\Mkl{15}{0} &   \cdots }}}
\end{align*}
Note that we have
\eqn
&&\{\M_\ttww(p,0)\}_{p\in \Z_{\ge 1}}
=\{\hd\bl L[a,b]\conv L[a-1,b-1]\cdots\conv L[-b,-a]\br\mid a\le b,a+b\ge0\}\\
&&\hs{15ex}\sqcup\{\hd\bl L[a,b]\conv L[a-1,b-1]\cdots\conv L[1-b,1-a]\br\mid a\le b,\;a+b\ge 1\}.
\eneqn
Here we confuse $M$ and the isomorphic class of $M$.

\begin{corollary} \label{cor: La computation}
For $p \in \Z_{\ge 1}$ with $c(p)=(\ell,m)$, we have
$$ \M_\ttww(p,0) = \M\Bigl( \La_{\jj_p+m}-\sum_{k=1}^m \ep_{\jj_p-\ell+k},\;\La_{\jj_p}\Bigr) \text{ and } \wt(\M_\ttww(p,0)) =\sum_{k=1}^m (\ep_{\jj_p+k}-\ep_{\jj_p-\ell+k}).$$
\end{corollary}

For a pair $(a,b)$ of integers with $a>b$, we denote also by
$$\Lrev[a,b]=\tF_{a}\tF_{a-1}\cdots \tF_{b} \cdot \one
\simeq \tF^*_{b}\cdots \tF^*_{a+1}\tF^*_{a} \cdot \one$$
the graded $1$-dimensional $R(\ep_b-\ep_{a+1})$-module.

\begin{proposition} \label{prop: p+,p ready} \hfill
\begin{enumerate}
\item[{\rm (a)}] For $p,p' \in\Z_{\ge 1}$ with $c(p)=(\ell,m)$ and  with $c(p')=(\ell,m+1)$, we have
$$
\begin{cases}
L[\jj_p-\ell+m+1,\jj_p+m]  \conv \M_\ttww(p,0) \twoheadrightarrow \M_\ttww(p',0) & \text{ if }  \ell+m \equiv 0 \mod 2, \\
\M_\ttww(p,0) \conv L[\jj_p-\ell,\jj_p-1] \twoheadrightarrow \M_\ttww(p',0) & \text{ if }  \ell+m \equiv 1 \mod 2.
\end{cases}
$$
\item[{\rm (b)}] For each $p\in\Z_{\ge 1}$ with $c(p)=(\ell,m)$, we have
$$
\begin{cases}
L[\jj_p-\ell+m+1,\jj_p+m] \hconv \M_\ttww(p,0) \simeq \M_\ttww(p^+,0) & \text{ if } \ell+m \equiv 0 \mod 2, \\
\Lrev[\jj_p-\ell+m-1,\jj_p-\ell]  \hconv \M_\ttww(p,0) \simeq \M_\ttww(p^+,0) & \text{ if } \ell+m \equiv 1 \mod 2.
\end{cases}
$$
\end{enumerate}
\end{proposition}

\begin{proof}
{\rm (a)} is a consequence of Theorem~\ref{thm :KR in KLR},
 since we have
$$
\jj_{p'} = \begin{cases}
\jj_p & \text{ if }  \ell+m \equiv 0 \mod 2, \\
\jj_p-1 & \text{ if }  \ell+m \equiv 1 \mod 2,
\end{cases}
\quad \text{by Lemma~\ref{lem: ell,m ip}.}
$$
For {\rm (b)}, if $\ell+m \equiv 0 \mod 2$, then $\jj_{p'}=\jj_{p^+}$ and hence it is the case of
{\rm (a)}. Now let us consider when $\ell+m \equiv 1 \mod 2$. In this case, $c(p^+)=(\ell+1,m)$ by Lemma~\ref{lem: ell,m ip}.
Then we have
\begin{itemize}
\item $\M_\ttww(p,0) \simeq \hd\bigl( L[\jj_p-\ell+m,\jj_p+m-1]\conv \cdots \conv L[\jj_p-\ell+2,\jj_p+1] \conv L[\jj_p-\ell+1,\jj_p]\bigr)$,
\item $\M_\ttww(p^+,0) \simeq \hd\bigl( L[\jj_p-\ell+m-1,\jj_p+m-1]\conv \cdots \conv L[\jj_p-\ell+1,\jj_p+1] \conv L[\jj_p-\ell,\jj_p]\bigr)$.
\end{itemize}
Note that
$$w_{\lb \ell+1,m \rb^{(\jj_p)}} \equiv   s_{\jj_p-\ell+m-1} \cdots s_{\jj_p-\ell+1} s_{\jj_p-\ell} *w_{ \lb \ell,m \rb^{(\jj_p)}}$$
with respect to the commutation relation. By Theorem~\ref{thm: p+,p monoidal}, we have
\begin{align*}
\M(\tw_{\lb \ell+1,m \rb^{(\jj_p)}} \Lambda_{\jj_p},\Lambda_{\jj_p}) &\simeq\M(\tw_{\lb \ell+1,m \rb^{(\jj_p)}} \Lambda_{\jj_p},\tw_{\lb \ell,m \rb^{(\jj_p)}} \Lambda_{\jj_p})\hconv \M(\tw_{\lb \ell,m \rb^{(\jj_p)}} \Lambda_{\jj_p},\Lambda_{\jj_p})\\
&\simeq
\M(\tw_{\lb \ell+1,m \rb^{(\jj_p)}} \Lambda_{\jj_p},\tw_{\lb \ell,m \rb^{(\jj_p)}} \Lambda_{\jj_p})\hconv \W_{m,\jj_p}^{(\ell)}.
\end{align*}
As in the proof of Lemma~\ref{lem:PBW} \eqref{item:KR}, we have
$$  \M(\tw_{\lb \ell+1,m \rb^{(\jj_p)}} \Lambda_{\jj_p},\tw_{\lb \ell,m \rb^{(\jj_p)}}\Lambda_{\jj_p} ) \simeq  \tF_{\jj_p-\ell+m-1} \cdots \tF_{\jj_p-\ell+1} \tF_{\jj_p-\ell} \cdot \one.$$
Hence our assertion follows.
\end{proof}

\begin{corollary} \label{cor: cuspidal}
For each $p\in\Z_{\ge 1}$ with $c(p)=(\ell,m)$, we have
$$
\M_\ttww(p^+,p) \simeq \begin{cases}
L[\jj_p-\ell+m+1, \jj_p+m]     & \text{ if } \ell+m \equiv 0 \mod 2, \\
\Lrev[\jj_p-\ell+m-1,\jj_p-\ell]   & \text{ if } \ell+m \equiv 1 \mod 2.
\end{cases}
$$
\end{corollary}

\begin{proof}
By Theorem~\ref{thm: p+,p monoidal} and Proposition~\ref{prop: p+,p ready},
$$\M_\ttww(p^+,p) \hconv \M_\ttww(p,0) \simeq \M_\ttww(p^+,0) $$
and
$$
\begin{cases}
L[\jj_p-\ell+m+1, \jj_p+m] \hconv \M_\ttww(p,0) \simeq \M_\ttww(p^+,0) & \text{ if } \ell+m \equiv 0 \mod 2, \\
\Lrev[\jj_p-\ell+m-1,\jj_p-\ell]  \hconv \M_\ttww(p,0) \simeq \M_\ttww(p^+,0) & \text{ if } \ell+m \equiv 1 \mod 2.
\end{cases}
$$
Then our assertion follows from Proposition~\ref{prop: head inj}.
\end{proof}

\subsection{Quantum monoidal seed $\Seed_\infty$} Since
Proposition~\ref{prop: reduced} tells that $\ttww_{\le p}$ is reduced for every $p \in \Z_{\ge 1}$,
we can consider the quiver $\tQ$ associated to $\ttww$ by taking $p \to \infty$ and using Definition~\ref{def: quiver Q assoc tw}.
The set of vertices of $\tQ$ is $K\seteq\Z_{\ge1}$.
The set of the frozen vertices of $\K$ is empty.
Note that there is a bijection $c\col K\to \Z_{\ge1}\times \Z_{\ge1}$.

\begin{proposition}
Each vertex of the quiver $\tQ$ is of finite degree. Hence the associated matrix $\tB$ satisfies
the conditions in~\eqref{eq: condition B}.
\end{proposition}
\begin{proof}
Note that for $p,p' \in\Z_{\ge 1}$ with  $c(p)=(\ell,m)$ and $c(p')=(\ell-1,m+1)$,
$$
\begin{cases}
p = p'+1 &\text{ if } \ell+m \equiv 0 \mod 2, \\
p = p'-1 &\text{ if } \ell+m \equiv 1 \mod 2.
\end{cases}
$$

By Lemma~\ref{lem: ell,m ip}, for $p \in \Z_{\ge 1}$ with $c(p)=(\ell,m)$, we have
\begin{equation}\label{eq: cp+-}
\begin{aligned}
&c(p_-) = \begin{cases}
(\ell-1,m) & \text{ if it exists and } \ell+m \equiv 0 \mod 2, \\
(\ell,m-1) & \text{ if it exists and } \ell+m \equiv 1 \mod 2,
\end{cases} \\
&c(p_+) = \begin{cases}
(\ell,m+1) & \text{ if } \ell+m \equiv 0 \mod 2, \\
(\ell+1,m) & \text{ if } \ell+m \equiv 1 \mod 2,
\end{cases}
\end{aligned}
\end{equation}
\vskip -0.8em
\begin{equation} \label{eq: cp+- pm1}
\begin{aligned}
& c(p^+(\jj_p-1))   = \begin{cases}
(\ell-\true(\ell \ne 1),m+2)   & \text{ if } \ell+m \equiv 0 \mod 2, \\
(\ell-\true(\ell \ne 1),m+1)   & \text{ if } \ell+m \equiv 1 \mod 2,
\end{cases} \\
& c(p^+(\jj_p+1)) = \begin{cases}
(\ell+1,m-\true(m \ne 1)) & \text{ if } \ell+m \equiv 0 \mod 2, \\
(\ell+2,m-\true(m \ne 1)) & \text{ if } \ell+m \equiv 1 \mod 2.
\end{cases}
\end{aligned}
\end{equation}

Combining \eqref{eq: cp+-} and \eqref{eq: cp+- pm1}, we have
\begin{align*}
& c(p^+(\jj_p-1)_+)   = \begin{cases}
(\ell,m+2 +\delta(\ell=1) ) & \text{ if } \ell+m \equiv 0 \mod 2, \\
(\ell,m+1 +\delta(\ell=1)) & \text{ if } \ell+m \equiv 1 \mod 2,
\end{cases} \\
& c(p^+(\jj_p+1)_+) = \begin{cases}
(\ell+1 +\delta(m=1),m) & \text{ if } \ell+m \equiv 0 \mod 2, \\
(\ell+2 +\delta(m=1),m) & \text{ if } \ell+m \equiv 1 \mod 2.
\end{cases}
\end{align*}
Hence we have the followings for $c(p)=(\ell,m)$:
\begin{enumerate}[{\rm (i)}]
\item If $\ell+m \equiv 0 \mod 2$ and $m\neq1$, then
\begin{equation*}
p < p^+(\jj_p+1) <  p^+(\jj_p+1)_+ < p_+ <  p^+(\jj_p+1)_{++} \text{ and } p < p_+ < p^+(\jj_p-1).
\end{equation*}
\item If $\ell+m \equiv 0 \mod 2$ and $m=1$, then
\begin{equation*}
p < p^+(\jj_p+1) <  p_+ <  p^+(\jj_p+1)_{+} \text{ and } p < p_+ < p^+(\jj_p-1).
\end{equation*}
\item If $\ell+m \equiv 1 \mod 2$ and $\ell\neq 1$, then
\begin{equation*}
p < p^+(\jj_p-1) <  p^+(\jj_p-1)_+ < p_+ <  p^+(\jj_p-1)_{++} \text{ and } p  <  p_+ < p^+(\jj_p+1).
\end{equation*}
\item If $\ell+m \equiv 1 \mod 2$ and $\ell= 1$, then
\begin{equation*}
p < p^+(\jj_p-1) <  p_+ <  p^+(\jj_p-1)_{+} \text{ and } p  <  p_+ < p^+(\jj_p+1).
\end{equation*}
\end{enumerate}

Hence, Definition~\ref{def: quiver Q assoc tw} tells that
there is only one ordinary arrow with source $p$, which is given by $p \to  p^+(\jj_p+1)_+$ or $p \to  p^+(\jj_p+1)$ if $\ell+m \equiv 0 \mod 2$, and $p \to  p^+(\jj_p-1)_+$ or
 $p \to  p^+(\jj_p-1)$ if $\ell+m \equiv 1 \mod 2$, respectively.
Hence each vertex $p \in J$ with $c(p)=(\ell,m)$ has two incoming arrows and two outgoing arrow unless $\min(\ell,m)=1 \colon$
\begin{align} \label{eq: nbd of initial quiver}
& \raisebox{3em}{\xymatrix@C=3ex@R=3ex{ &  {\scriptstyle p_+} \ar[d]  \\
{\scriptstyle p_-}  & {\scriptstyle p} \ar[l] \ar[r] & {\scriptstyle p^+(\jj_p+1)_+} \\
 & {\scriptstyle p^-(\jj_p+1)}\ar[u]
}} &&
\raisebox{3em}{\xymatrix@C=3ex@R=3ex{   & {\scriptstyle p^+(\jj_p-1)_+}   \\
{\scriptstyle p^-(\jj_p-1)} \ar[r] & {\scriptstyle p} \ar[d]\ar[u]  & {\scriptstyle p_+} \ar[l] \\
 & {\scriptstyle p_-}
}}
\\*
& \quad  \text{ if } \ell+m \equiv 0 \mod 2 && \hspace{8ex} \text{ if } \ell+m \equiv 1 \mod 2. \nonumber
\end{align}
In particular, when {\rm (i)} $\ell+m \equiv 0 \mod 2$ and $m=1$, or {\rm (ii)} $\ell+m \equiv 1 \mod 2$ and $\ell= 1$, the arrows incident with vertex $p$ can be described as follows:
\begin{align} \label{eq: nbd of initial quiver 2}
& {\rm (i)} \quad
\raisebox{3em}{\xymatrix@C=4ex@R=4ex{   & {\scriptstyle p^+}  \ar[d] \\
{\scriptstyle p_-}  & {\scriptstyle p} \ar[l]\ar[r]  & {\scriptstyle p^+(\jj_p+1)}
}} \qquad\qquad
{\rm (ii)} \quad \raisebox{3em}{\xymatrix@C=4ex@R=4ex{ {\scriptstyle p^+(\jj_p-1)}   \\
 {\scriptstyle p}  \ar[u]\ar[d] & {\scriptstyle p_+} \ar[l]\\
{\scriptstyle p_-}
}}
\end{align}

By replacing $p=c^{-1}(\ell,m)$ with $\W^{(\ell)}_{m,\jj_p}$ in the diagrams above, the quiver $\tQ$ can be exhibited as follows:
\begin{align} \label{eq: Q with W}
\raisebox{6em}{\scalebox{0.78}{\xymatrix@C=3ex@R=3ex{
\vdots&\vdots\ar[d] &\vdots & \vdots\ar[d] &\vdots &\vdots \ar[d] & \iddots\\
3&\Wlmj{1}{3}{-1}  \ar@{.>}[r]& \Wlmj{2}{3}{0}\ar@{.>}[u]\ar[d]  & \Wlmj{3}{3}{0}  \ar@{.>}[r]\ar[l]& \Wlmj{4}{3}{1}\ar@{.>}[u]\ar[d]  & \Wlmj{5}{3}{1}  \ar@{.>}[r]\ar[l]& \cdots \\
2&\Wlmj{1}{2}{0} \ar@{.>}[u]\ar[d] & \Wlmj{2}{2}{0} \ar@{.>}[r]\ar[l] & \Wlmj{3}{2}{1} \ar@{.>}[u]\ar[d] & \Wlmj{4}{2}{1} \ar@{.>}[r]\ar[l] & \Wlmj{5}{2}{2}\ar@{.>}[u]\ar[d]  & \cdots \\
1&\Wlmj{1}{1}{0} \ar@{.>}[r]& \Wlmj{2}{1}{1}\ar@{.>}[u]  & \Wlmj{3}{1}{1} \ar@{.>}[r]\ar[l]& \Wlmj{4}{1}{2}\ar@{.>}[u] & \Wlmj{5}{1}{2}  \ar@{.>}[r]\ar[l] & \cdots \\
m / \ell & 1 & 2 & 3 & 4 & 5 & \cdots}}}
\end{align}
Hence our assertion follows.
\end{proof}

Note that the arrows of $\tQ$ oriented right or above are arrows
of horizontal type,
and arrows of $\tQ$ oriented left or below are
arrows of ordinary type. Also the quiver $\tQ$ is known as the {\it square product} $Q_{A_\infty} \square Q_{A_\infty}$ of the bipartite Dynkin quiver $Q_{A_\infty}$ of type $A_\infty$,
which is related to the periodicity conjecture (see~\cite{FZ03,IIKKN13A,IIKKN13B,Ke10,Ke13,Sz09,W12}). Here $Q_{A_\infty}$ is the quiver
$\xymatrix@R=0.5ex@C=4ex{ *{\circ}<3pt> \ar@{->}[r]_<{1  \ }
&*{\circ}<3pt> \ar@{<-}[r]_<{2 \ }  &   *{\circ}<3pt> \ar@{->}[r]_<{3 \ }
&  *{\circ}<3pt> \ar@{.}[r]_<{4 \ } &}
$.

\begin{remark} \label{rem: periodcity}
Take $k \in \Z_{\ge 1}$.  When we restrict the full subquiver $Q^{(k)}$ of $\tQ$ consisting of vertices $(\ell,m)$ with $\ell+m \le k+1$, the $Q^{(k)}$ is mutation equivalent to
the well-known quiver $Q_k^{BFZ}$ of the coordinate ring $\C[N]$ of the unipotent group $N$ of type $A_k$ (see~\cite{BFZ05} and~\cite[Theorem 4.5]{Ke10}). For example $k=3$, $Q^{(3)}$ is given as follows:
$$
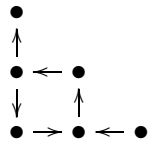
\xymatrix@R=2ex@C=2ex{ \bullet \\ \bullet \ar[d] \ar[u]  & \bullet \ar[l] \\ \bullet \ar[r] & \bullet  \ar[u] & \bullet \ar[l] }$$

Note that $Q_k^{BFZ}$ is isomorphic to the quiver
associated to some adapted reduced expression
of the longest element of the Weyl group of $A_k$
(see Definition~\ref{def: quiver Q assoc tw}).
\end{remark}

Now we take the skew-symmetric integer-valued $\Z_{\ge 1} \times \Z_{\ge 1}$-matrix $\Lambda$ as follows:
\begin{align*}
\Lambda=\left(\Lambda\bigl(\M_\ttww(i,0),\M_\ttww(j,0)\bigr) \right)_{i,j \in \Z_{\ge 1}}.
\end{align*}

Note that for each $p \in \Z_{\ge 1}$, we can take sufficiently large $t$ such that $p$ is an exchangeable index of
$\ttww_{\le t}$ (see~\eqref{eq: decomposition of indices}).
Thus the following corollary follows from Theorem~\ref{thm: monoidal} for $\ttww_{\le t}$:

\begin{corollary} \label{cor: consequences} \hfill
\begin{enumerate}[{\rm (a)}]
\item For any finite sequence $\Seq$ in $\K$, $\mu_{\Seq}(-\Lambda, \tB)$ is compatible with $d=2$, where $\tB$ is the matrix associated with $\tQ$.
\item For every pair of positive integers $p=c^{-1}(\ell,m)$ and $p'=c^{-1}(\ell',m')$,
$\W^{(\ell)}_{m,\jj_p}$ and $\W^{(\ell')}_{m',\jj_{p'}}$ strongly commute.
\item For each $p \in \Z_{\ge 1}$, there exists $\M_\ttww(p,0)'$ in $R^J\gmod$ satisfying~\eqref{eq:ses graded mutation KLR}.
\end{enumerate}
\end{corollary}

Let us denote by
\begin{align*} 
\Seed_\infty = (\{\M_\ttww(p,0) \}_{p \in K},\tB).
\end{align*}
Then $\Seed_\infty$ becomes a quantum monoidal seed.
Furthermore, we have
\begin{itemize}
\item the pair $(\{\M_\ttww(p,0) \}_{p \in K},\tB)$ is admissible,

\vs{.5ex}
\item $\Seed_\infty$ admits successive mutations in $R^J\gmod$ for all the directions.
\end{itemize}

Let $\mathscr{A}_{q^{1/2}}(\infty)$
be the quantum cluster algebra associated with the quantum seed
$$ [\Seed_\infty] \seteq (\{q^{-(d_p,d_p)/4}[\M_\ttww(p,0)]\}_{i\in\K}, -\Lambda, \widetilde B)$$
 without  frozen variables.

\begin{remark}
Note that the quantum cluster algebras associated with a quiver of infinite rank is the $\Z[q^{\pm 1/2}]$-subalgebra of skew field $\mathscr{F}$ generated by all elements
obtained from initial cluster variables by {\it finite} sequences of mutations. We refer  to \cite{GG,HL16} for the details of the definition of (quantum) cluster algebra of infinite rank.
\end{remark}

\begin{theorem} \label{Thm: monoidal cat}
The category $R^J\gmod$ is a monoidal categorification of the
quantum cluster algebra $\mathscr{A}_{q^{1/2}}(\infty)$.
\end{theorem}

\begin{proof}

Theorem~\ref{thm: monoidal categorification} implies that
any cluster monomial of $\mathscr{A}_{q^{1/2}}(\infty)$ is contained in $K(\shc_{\ttww_{\le t}}) \subset \K(R^J\gmod)$ for sufficiently large $t \in \Z_{\ge 1}$.
On the other hand, let $M$ be a simple $R^J(\beta)$-module.
We write $\beta = \sum_{k=-p}^p a_k \alpha_k$. Then
 by Remark~\ref{rem: coordinate} (a), $[M]$ is contained in $K(\shc_{\ttww_{\le a(p+1)}})$.
 Thus we conclude that $\K(R^J\gmod) = \mathscr{A}_{q^{1/2}}(\infty)$, which completes the proof.
\end{proof}

\section{The category $\T_N$ and its cluster structure} \label{Sec: T_N}
\subsection{Category $\T_N$}
We keep the notations in \S\,\ref{Sec: Module}.
In this subsection we briefly recall the quotient category and the localizations of $R\gmod$ introduced
in~\cite[\S 4.4--\S 4.5]{KKK18A}.
For details of the constructions, we refer to \cite[Appendix A and B]{KKK18A}. Then in the next section we apply one of the main results in~\cite{KKO18} to show that a generalized quantum affine Schur-Weyl duality functor
$\F$ factors thorough the category $\T_N$ defined below.

\smallskip

Set $\A_{\beta} \seteq R^J({\beta}) \gmod$ and $\A \seteq \soplus_{{\beta} \in \rl_J^+} \A_{\beta}$.
Let $\Ser_N$ be the smallest Serre subcategory of $\A$ such that
\begin{enumerate}[{\rm (i)}]
\item $\Ser_N$ contains $L[a, a+N]$ for any $a \in J$,
\item $X \conv Y, \ Y \conv X \in \Ser_N$ for all $X \in \A$ and $Y \in \Ser_N$.
\end{enumerate}
Note that $\Ser_N$ contains $L[a,b]$ if $b-a +1 >  N$.

Let us denote by $\A/ \Ser_N$ the quotient category of $\A$ by $\Ser_N$ and denote by $\FQ_N\col \A \rightarrow \A/ \Ser_N$
the canonical functor.

Note that $\A$ and $\A/\Ser_N$ are monoidal categories with the convolution as tensor products. The module $R(0)\simeq \cor$ is a unit object. Note also that $Q\seteq q R(0)$ is an invertible
central object of $\A/\Ser_N$ and $X\mapsto  Q\conv X\simeq X\conv Q$ coincides with the grading shift functor.  Moreover, the functors $\FQ_N$
is a monoidal functor.

\smallskip

\begin{definition} \label{def: quot coeff}
For each $a,j \in J$,  we define
\begin{enumerate}
\item[{\rm (a)}] $L_a \seteq L[a,a+N-1]$, \label{def: L_a}
\item[{\rm (b)}] an abelian group homomorphism $c_a:\rl_J \to \Z$  as  $c_a(\al_j) \seteq (\ep_a+\ep_{a+N},\al_j)$, \label{def: c_a}
\item[{\rm (c)}] a polynomial $f_{a,j}(z) \seteq (-1)^{\delta_{j,a+N}}z^{-\true(a\le j<a+N-1)-\delta_{j,a+N}} \in \cor[z^{\pm1}]$.  \label{def: f_aj}
\end{enumerate}
\end{definition}

\begin{definition}
Let  $\mathsf{S}_N $ be the automorphism of $\rl_J=\soplus_{a\in\Z} \Z\,\al_a$ given by $\mathsf{S}_N(\al_a)=\al_{a+N}$.
We define the bilinear form $\mathfrak{B}_N $ on $\rl_J$ by
\begin{align*} 
\mathfrak{B}_N(x,y) =-\sum_{k>0}(\mathsf{S}_N^k x,y) \quad \text{ for } x,y\in \rl_J.
\end{align*}
\end{definition}

\begin{proposition}[{\cite[Proposition 4.17]{KKK18A}}]
For any $M \in (\A/\Ser_N)_\be$, we have an isomorphism
$$
L_a \conv M \isoto q^{c_a(\be)}M \conv L_a.
$$
in $\A/\Ser_N$ which is functorial, and
\begin{align} \label{eq: ca(be)}
c_a(\be) = -\mathfrak{B}_N(\ep_a-\ep_{a+N},\be)+\mathfrak{B}_N(\be,\ep_a-\ep_{a+N}).
\end{align}
\end{proposition}

\begin{definition} \label{def: star op}
We define the new tensor product
$\starop\col\A\times \A\to\A$ by
\begin{align} \label{def:star}
X\starop Y=q^{\mathfrak{B}_N(\al,\beta) }  X\conv Y,
\end{align}
where $X\in\A_{\al}$ and $Y\in\A_{\beta}$.
\end{definition}
Then, $\A$ as well as $\A/\Ser_N$ is endowed with a new structure of
a monoidal category by $\starop$ as shown in
\cite[Appendix A.8]{KKK18A}.

\begin{theorem}[{\cite[Theorem 4.21]{KKK18A}}] \label{thm:L_a commutes with X}
The following statements  hold.
\begin{enumerate}
\item[{\rm (i)}] $L_a$ is a {\it central object} in  $\A/\Ser_N$; i.e.,
\begin{enumerate}
\item[{\rm (a)}] $f_{a,j}(z)R_{L_a,L(j)_z}$ induces an isomorphism $R_a(X)\col L_a\starop X\isoto X\starop L_a$  functorial in  $X\in\A/\Ser_N$,
\item[{\rm (a)}] the diagram
$$\xymatrix@C=12ex
{L_a\starop X\starop Y\ar[r]_{R_a(X)\starop Y}\ar@/^2pc/[rr]^-{R_a(X\starop Y)\hspace{2ex} }& X\starop L_a\starop Y\ar[r]_{X\starop R_a(Y)} & X\starop Y \starop L_a }
$$
is commutative in $\A/\Ser_N$ for any $X,Y\in\A/\Ser_N$.
\end{enumerate}
\item[{\rm (ii)}] The isomorphism $R_a(L_a)\col L_a\starop L_a\isoto L_a\starop L_a$ coincides with $\id_{L_a\starop L_a}$ in $\A/\Ser_N$.
\item[{\rm (iii)}] For $a,b\in\Z$, the isomorphisms
\begin{align*}
 R_a(L_b)\col L_a\starop L_b\isoto L_b\starop L_a \ \text{and} \ R_b(L_a)\col L_b\starop L_a\isoto L_a\starop L_b
\end{align*}
in $\A/\Ser_N$ are  inverse to each other.
\end{enumerate}
\end{theorem}

By the preceding theorem, $\{(L_a, R_a)\}_{a \in J}$ forms a \emph{commuting family of central objects} in $(\A / \Ser_N , \starop)$
(See~\cite[Appendix A. 4]{KKK18A}).
Following~\cite[Appendix A. 6]{KKK18A}, we localize $(\A / \Ser_N , \starop)$ by this commuting family.
Let us denote by  $\T'_N$ the resulting category $(\A / \Ser_N)[L_a^{\starop -1}\mid a\in J]$.
Let $\Upsilon \colon \A / \Ser_N \to \T'_N$ be the  projection functor.
We denote by $\T_N$ the monoidal category $(\A / \Ser_N)\,[L_a\simeq\one\mid a\in J]$ and by $\Xi \col \T'_N \to \T_N$ the canonical functor (see~\cite[Appendix A.7]{KKK18A} and~\cite[Remark 4.22]{KKK18A}).
Thus we have a chain of monoidal functors
$$\A \To[\ {\FQ_N}\ ]\A/\Ser_N\To[\ \Upsilon\ ] \T'_N\seteq(\A / \Ser_N)[L_a^{\starop -1}\mid a\in J] \To[\ \Xi\ ] \T_N\seteq(\A / \Ser_N)[L_a\simeq\one\mid a\in J].$$

We set
\begin{align*} 
\FO_N \seteq \Xi \circ \Upsilon \circ \FQ_N\col \A\To \T_N.
\end{align*}

\begin{theorem}[{\cite[Theorem 4.25]{KKK18A}}]
The categories $\T_N$ and $\T'_N$ are rigid monoidal categories; i.e., every object has a right dual and a left dual.
\end{theorem}

\subsection{Cluster structure on $\T_N$}
In this and next subsections, we prove that
$K(\T_N)$ has a structure of quantum cluster algebra, and
$\T_N$ is its monoidal categorification.

Let us recall the quiver $\tQ$ in~\eqref{eq: Q with W} associated to the infinite sequence $\ttww$ of $S_J$ in~\eqref{eq: def tw}.
The vertices of $\tQ$ are labeled by $\K=\Z_{\ge 1}$ which is also identified with $\Z_{\ge 1} \times \Z_{\ge 1}$ under the map $c$ sending $p$ to $(\ell,m)$.

\begin{proposition} \label{prop: first mutation}
For each $p \in \K$ with $c(p)=(\ell,m)$, the $R$-module $\M_\ttww(p,0)'$ in~\eqref{eq: prime} is given as follows$\col$
$$
\M_\ttww(p,0)' \simeq \begin{cases}
\W^{(\ell)}_{m,\jj_p+1 } & \text{ if } \ell+m \equiv 0 \mod 2, \\
\W^{(\ell)}_{m,\jj_p-1} & \text{ if } \ell+m \equiv 1 \mod 2.
\end{cases}
$$
\end{proposition}

\begin{proof}By~\eqref{eq: nbd of initial quiver}, \eqref{eq: nbd of initial quiver 2} and Proposition~\ref{prop: KR quiver redex}, the neighborhood of $\M_\ttww(p,0)$ in $\tQ$ with $c(p)=(\ell,m)$ and $\ell+m \equiv 0 \mod 2$ can be described as follows:
$$
\raisebox{3.2em}{\xymatrix@C=3ex@R=3ex{ &  {\scriptstyle \W^{(\ell)}_{m+1,\jj_p}} \ar[d] \\
{\scriptstyle \W^{(\ell-1)}_{m,\jj_p}}  & {\scriptstyle \W^{(\ell)}_{m,\jj_p}} \ar[l] \ar[r] & {\scriptstyle \W^{(\ell+1)}_{m,\jj_p+1}} \\
& {\scriptstyle \W^{(\ell)}_{m-1,\jj_p+1}} \ar[u]
} }.
$$

Then we have
\begin{itemize}
\item $\M_\ttww(p_+,p) \simeq L[\jj_p-\ell+m+1,\jj_p+m]$ by Corollary \ref{cor: cuspidal},
\item $\displaystyle\sodot_{t < p <t_+ <p_+} \M_\ttww(t,0)^{\snconv |a_{\jj_p,\jj_t}|}$ in~\eqref{eq: prime} is isomorphic to $\W^{(\ell)}_{m-1,\jj_p+1}$.
\end{itemize}
Thus ~\eqref{eq: prime} implies that
\begin{align*}
\M_\ttww(p,0)' \simeq L[\jj_p-\ell+m+1,\jj_p+m] \hconv \W^{(\ell)}_{m-1,\jj_p+1} \simeq \W^{(\ell)}_{m,\jj_p+1}.
\end{align*}
Similarly, we can prove the assertion when $\ell+m \equiv 1 \mod 2$.
\end{proof}

Let us take a total order on $\Z_{\ge 1}\times \Z_{\ge 1}$ as follows:
$$ (\ell,m) > (\ell',m') \quad  \text{ if } \ell+m  > \ell'+m', \text{ or } \ell+m = \ell'+m' \text{ and } \ell>\ell'.$$

Let $\Sev$ (resp.\ $\Sod$) be the ascending sequence on the set of all $(\ell,m) \in \Z_{\ge 1} \times \Z_{\ge 1}$ with $\ell+m \equiv 0 \mod 2$ (resp.\ $\ell+m \equiv 1 \mod 2$)$\colon$
\begin{align*}
\Sev&=\bigl( (1,1),  (1,3),(2,2),(3,1), (1,5),(2,4),(3,3),(4,2),(5,1) ,\ldots \bigr),  \\
\text{(resp. \ }\Sod&=\bigl( (1,2),(2,1), (1,4),(2,3),(3,2),(4,1), (1,6),(2,5),(3,4),\ldots \bigr) \text{)}.
\end{align*}

Let $\Spl$ be the sequence $\Sev$ followed by $\Sod$, and
$\Seq^-$ be the sequence $\Sod$ followed by $\Sev\colon$
\begin{align*}
\Spl&=\bigl( (1,1),(1,3),(2,2),(3,1), \ldots, (1,2),(2,1), (1,4),(2,3), \ldots\bigr),  \\
\Smi&=\bigl( (1,2),(2,1), (1,4),(2,3),\ldots, (1,1),(1,3),(2,2),(3,1), \ldots \bigr).
\end{align*}

For a sequence $\Seq=( (\ell_1,m_1),(\ell_2,m_2),\ldots )$ of $\Z_{\ge 1} \times \Z_{\ge 1}$, we denote by
$\mu_{\Seq}(\Seed)$ be the new quantum monoidal seed after performing the sequence of mutations indexed by $\Seq$; that is, regarding the sequence $\Seq$
as a sequence in $\Kex$ by
$$ \Seq = (c^{-1}(\ell_1,m_1),c^{-1}(\ell_2,m_2),\ldots),$$
$\mu_{\Seq}(\Seed)$ is defined as follows:
$$\mu_{\Seq}(\Seed) =  \cdots \mu_{c^{-1}(\ell_2,m_2)}\,\mu_{c^{-1}(\ell_1,m_1)}(\Seed).$$
For $r \in \Z_{\ge 1}$, we denote by $\mu^{(r)}_{\Seq}(\Seed)$ the quantum monoidal seed obtained from $\Seed$ after $r$-repetitions of
the mutation sequence $\mu_{\Seq}$.

From now on, we sometime write $\mu_{(\ell,m)}$ instead of $\mu_{c^{-1}(\ell,m)}$
for simplicity.

\begin{lemma} \label{lem: reversing} We have
$$  \mu_{\Sod}(\tQ) \simeq \mu_{\Sev}(\tQ) \simeq \tQ^{{\rm op}} \quad \text{ as quivers},$$
where $\tQ^{{\rm op}}$ is the quiver obtained by reversing all arrows in $\tQ$.
\end{lemma}

\begin{proof}
The neighborhoods of $(\ell,m)$ with $\ell+m \equiv 0 \mod 2$ in $\tQ$ and $\mu_{(\ell,m)}(\tQ)$ can be described as follows:
\begin{align}\label{eq: 1st mutation}
\raisebox{3em}{\xymatrix@C=3ex@R=3ex{ {\scriptstyle(\ell-1,m+1)} &  {\scriptstyle(\ell,m+1)} \ar[d] &  {\scriptstyle(\ell+1,m+1)} \\
{\scriptstyle(\ell-1,m)}  & {\scriptstyle(\ell,m)} \ar[l] \ar[r] & {\scriptstyle(\ell+1,m)} \\
{\scriptstyle(\ell-1,m-1)} & {\scriptstyle(\ell,m-1)} \ar[u] & {\scriptstyle(\ell+1,m-1)}
} }
\To[\mu_{(\ell,m)}]
\raisebox{3em}{\xymatrix@C=3ex@R=3ex{ {\scriptstyle(\ell-1,m+1)} &  {\scriptstyle(\ell,m+1)}\ar[dl]\ar[dr]  &  {\scriptstyle(\ell+1,m+1)}   \\
{\scriptstyle(\ell-1,m)} \ar[r] & {\scriptstyle(\ell,m)}\ar[d]\ar[u]    & {\scriptstyle(\ell+1,m)} \ar[l]\\
{\scriptstyle(\ell-1,m-1)} & {\scriptstyle(\ell,m-1)}\ar[ul]\ar[ur]   & {\scriptstyle(\ell+1,m-1).}
} }
\end{align}
When we apply $\mu_{(\ell,m)}$ in the sequence of mutations $\mu_{\Sev}$, there exist
\begin{itemize}
\item the arrow $(\ell-1,m) \to (\ell,m-1)$ made by $\mu_{(\ell-1,m-1)}$,
\item the arrow $(\ell-1,m) \to (\ell,m+1)$ made by $\mu_{(\ell-1,m+1)}$.
\end{itemize}
Thus the arrows $(\ell,m-1)\to (\ell-1,m)$ and $(\ell,m+1)\to (\ell-1,m)$ in~\eqref{eq: 1st mutation} are removed.
Also,
\begin{itemize}
\item the arrow $(\ell,m-1)\to (\ell+1,m)$ in~\eqref{eq: 1st mutation} will be removed when we apply $\mu_{(\ell+1,m-1)}$,
\item the arrow $(\ell,m+1)\to (\ell+1,m)$ in~\eqref{eq: 1st mutation} will be removed when we apply $\mu_{(\ell+1,m+1)}$,
\end{itemize}
by the same reason. Hence our assertion for $\Sev$ follows. The second assertion can be proved in a similar way.
\end{proof}

\begin{proposition} \label{pro: mu S+ S-}  \hfill
\begin{enumerate}[{\rm (a)}]
\item $\mu_{\Spl}(\tQ) \simeq \mu_{\Smi}(\tQ) \simeq \tQ$ as quivers.
\item For $p \in \K$  with $c(p)=(\ell,m)$, we have
$$\text{$\mu_{\Spl}(\M_\ttww(p,0)) = \W^{(\ell)}_{m,\jj_p+1}$ and $\mu_{\Smi}(\M_\ttww(p,0)) = \W^{(\ell)}_{m,\jj_p-1}$.}$$
\end{enumerate}
\end{proposition}

\begin{proof}
(a) The first assertion can be proved by applying the same argument as
in Lemma~\ref{lem: reversing}.

\smallskip\noi
(b) Note that $\mu_{\Spl} = \mu_{\Sod} \circ \mu_{\Sev}$.  For $(\ell,m)> (\ell',m')$ with $\ell+m \equiv \ell'+m' \equiv 0 \mod 2$, the arrows incident with $\M_\ttww(p,0)$ sitting at the coordinate $(\ell,m)$
are not affected by the mutation $\mu_{(\ell',m')}$. Thus the first assertion for $p$ with $\ell+m \equiv 0 \mod 2$ follows from Proposition~\ref{prop: first mutation}.

After performing $\mu_{\Sev}$, Lemma~\ref{lem: reversing} tells that the neighborhood of $\M_\ttww(p,0)$ with $\ell+m \equiv 1 \mod 2$ in $\mu_{\Sev}(\tQ)$ can be described as follows:
\begin{align}\label{eq: nbd after even}
\raisebox{3.3em}{\xymatrix@C=3ex@R=3ex{ &  {\scriptstyle \W^{(\ell)}_{m+1,\jj_p}} \ar[d]  \\
{\scriptstyle \W^{(\ell-1)}_{m,\jj_p}}  & {\scriptstyle \W^{(\ell)}_{m,\jj_p}} \ar[r]\ar[l]  &  {\scriptstyle \W^{(\ell+1)}_{m,\jj_p+1}}  \\
& {\scriptstyle \W^{(\ell)}_{m-1,\jj_p+1}} \ar[u]
}}.
\end{align}
Then we have (up to grading shifts)
\begin{align*}
& \bigl(L[\jj_p-\ell+m+1,\jj_p+m] \hconv \Wlmj{\ell}{m-1}{\jj_p+1} \bigr) \conv \Wlmj{\ell}{m}{\jj_p} \\
& \hspace{20ex} \rightarrowtail \Wlmj{\ell}{m-1}{\jj_p+1} \conv L[\jj_p-\ell+m+1,\jj_p+m]  \conv \Wlmj{\ell}{m}{\jj_p} \\
& \hspace{30ex} \twoheadrightarrow \Wlmj{\ell}{m-1}{\jj_p+1} \conv \Wlmj{\ell}{m+1}{\jj_p},
\end{align*}
where the composition is non-zero by~\cite[Lemma 3.1.5]{KKKO18}.
Since $\Wlmj{\ell}{m-1}{\jj_p+1}$ and $\Wlmj{\ell}{m+1}{\jj_p}$ strongly commute (Corollary~\ref{cor: consequences}), the composition is an epimorphism.
Note that
$$ L[\jj_p-\ell+m+1,\jj_p+m] \hconv \Wlmj{\ell}{m-1}{\jj_p+1} \simeq \Wlmj{\ell}{m}{\jj_p+1}.$$

Since  $\M_\ttww(p,0)'$ is unique, our first assertion follows from Proposition~\ref{prop: head inj}. The second assertion can be proved in a similar way.
\end{proof}

By applying the argument in the proof of Proposition~\ref{pro: mu S+ S-} repeatedly, we have the following corollary.
\begin{corollary} \label{cor: shift}
For each $p \in \K$ with $c(p)=(\ell,m)$, we have
$$ \mu^{(r)}_{\Spl}(\M_\ttww(p,0)) = \W^{(\ell)}_{m,\jj_p+r} \ \ \text{ and } \ \ \mu^{(r)}_{\Smi}(\M_\ttww(p,0)) = \W^{(\ell)}_{m,\jj_p-r} \ \ \text{ for any $r \in \Z_{\ge 1}$}.$$
In particular, we have the following exact sequence$\colon$ For any $k \in J$,
\begin{align} \label{eq: T-system in qHecke}
0 \to  q \Wlmj{\ell}{m-1}{k+1}  \nconv \Wlmj{\ell}{m+1}{k}  \to  q^{\tL} \Wlmj{\ell}{m}{k}  \conv \Wlmj{\ell}{m}{k+1}  \to   \Wlmj{\ell-1}{m}{k} \nconv \Wlmj{\ell+1}{m}{k+1}  \to 0,
\end{align}
where $\tL=\tL\bigl(\Wlmj{\ell}{m}{k},\Wlmj{\ell}{m}{k+1} \bigr)$.
\end{corollary}

For $N \in \Z_{\ge 1}$,  let us denote by $\K_{\le N}$, $\K_{\ge N}$ and
$K_{=N}$ the subsets of $\K$ as follows:
\begin{align*}
\K_{\le N} &\seteq \{ p \in \K \ | \  \ell \le N \text{ for } c(p)=(\ell,m)\}, \\
\K_{\ge N} &\seteq \{ p \in \K \ | \  \ell \ge N \text{ for } c(p)=(\ell,m)\},\\
 \K_{=N} &\seteq \{ p \in \K \ | \  \ell= N \text{ for } c(p)=(\ell,m)\}.
\end{align*}
We denote by $\tQ_N$ the subquiver of $\tQ$ obtained by
\begin{align} \label{eq: the procedure}
\text{ deleting the vertices in $\K_{\ge N+1}$ and the arrows joining vertices in $\K_{=N}$}.
\end{align}
Then the set of vertices of $\tQ_N$ are labeled by $\K_N\seteq \K_{\le N}$.

For each $N \in \Z_{\ge 1}$, we decompose $\K_{N}$ into the set of exchange vertices $\Kex_N$ and the set of frozen vertices $\Kfr_N$ by defining
$$\Kex_N \seteq \K_{\le N-1} \quad \text{ and } \quad  \Kfr_N \seteq \K_{=N}.$$

For any quiver $Q$ with $\K$ as its set of vertices,
we denote by $Q_N$
the quiver obtained by applying the procedure
~\eqref{eq: the procedure} to $Q$.

For any sequence $\Seq$ in $\Kex_N$, we have
$$\mu_\Seq(Q_N)=\bl\mu(Q)\br_N.$$

\begin{lemma} \label{lem: in out}
For any sequence $\Seq$ in $\Kex_N$, there exists no arrow in $\mu_\Seq(\tQ)$ between a vertex in $\K_{\ge N+1}$ and a vertex in $\K_{\le N-1}$.
\end{lemma}

\begin{proof}
For the initial quiver $\tQ$, it is obvious.
Then our assertion follows from the mutation rule of quivers described in~\eqref{eq: mutation in a direction k} {\rm (b)}.
\end{proof}

For a quiver $Q$, let $\overline{Q}$ be the full subquiver of $Q$
obtained by deleting
all frozen vertices.

\begin{example} The quivers $\tQ_3$ and $\overline{\tQ_3}$ can be described as follows:
$$
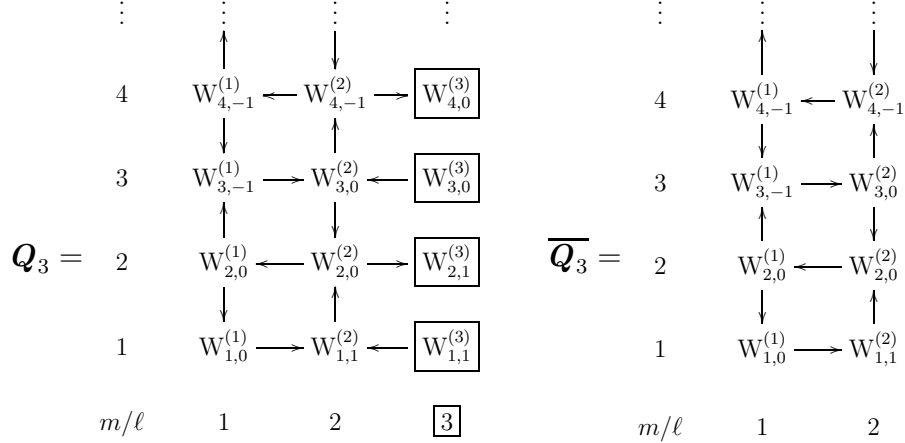
\tQ_3 =\raisebox{8em}{\scalebox{0.8}{\xymatrix@C=3ex@R=1.8ex{
\vdots&\vdots &\vdots\ar[d] & \vdots \\
4&\Wlmj{1}{4}{-1} \ar[u]\ar[d] & \Wlmj{2}{4}{-1} \ar[r]\ar[l] & \boxed{\Wlmj{3}{4}{0}} \\
3&\Wlmj{1}{3}{-1}  \ar[r]& \Wlmj{2}{3}{0}\ar[u]\ar[d]  & \boxed{\Wlmj{3}{3}{0}} \ar[l] \\
2&\Wlmj{1}{2}{0} \ar[u]\ar[d] & \Wlmj{2}{2}{0} \ar[r]\ar[l] & \boxed{\Wlmj{3}{2}{1}} \\
1&\Wlmj{1}{1}{0} \ar[r]& \Wlmj{2}{1}{1}\ar[u]  & \boxed{\Wlmj{3}{1}{1}}\ar[l] \\
m /\ell & 1 & 2 & \boxed{3} }}} \qquad \overline{\tQ_3} =
\raisebox{8em}{\scalebox{0.8}{\xymatrix@C=3ex@R=3ex{
\vdots&\vdots &\vdots\ar[d]  \\
4&\Wlmj{1}{4}{-1} \ar[u]\ar[d] & \Wlmj{2}{4}{-1}\ar[l]  \\
3&\Wlmj{1}{3}{-1}  \ar[r]& \Wlmj{2}{3}{0}\ar[u]\ar[d]   \\
2&\Wlmj{1}{2}{0} \ar[u]\ar[d] & \Wlmj{2}{2}{0} \ar[l]\\
1&\Wlmj{1}{1}{0} \ar[r]& \Wlmj{2}{1}{1}\ar[u]  \\
m /\ell & 1 & 2  }}}
$$
where $\boxed{\ast}$ denotes frozen vertices.
\end{example}

\begin{lemma} \label{lem: do not effect}
For any sequence $\Seq$ in $\Kex_N$, we have
$$  \overline{\mu_{\Seq}(\tQ_N)} \simeq \mu_{\Seq}\bigl(\overline{\tQ_N}\bigr)\simeq \overline{\bigl(\mu_{\Seq}(\tQ)\bigr)_N} \quad \text{ as quivers}.$$
\end{lemma}

Let $\Sev_N$, $\Sod_N$, $\Spl_N$ and $\Smi_N$ be the subsequences of $\Sev$, $\Sod$, $\Spl$ and $\Smi$
obtained by removing all the vertices outside $\Kex_N$,
respectively. Then the following lemma can be proved by applying the same arguments as in
the proofs of Lemma~\ref{lem: reversing} and Proposition~\ref{pro: mu S+ S-}.

\begin{lemma} \label{lem: seq mu in TN}
We have
\bnum
\item $\mu_{\Sod_N}\bigl( \overline{\tQ_N} \bigr) \simeq \mu_{\Sev_N}\bigl( \overline{\tQ_N} \bigr) \simeq \left( \overline{\tQ_N} \right)^{{\rm op}}$ as quivers,
\item $\mu_{\Spl_N}\bigl( \overline{\tQ_N} \bigr) \simeq \mu_{\Smi_N}\bigl( \overline{\tQ_N} \bigr) \simeq \overline{\tQ_N} $ as quivers.
\end{enumerate}
\end{lemma}

\subsection{Monoidal categorifications via $\T_N$ and $\Ca_J$}

\begin{proposition}[{\cite[Proposition 4.12, Proposition 4.31]{KKK18A}}] \label{prop: le N do not vanish} \hfill
\bnum
\item If an object $Y$ is simple in $\A/\Ser_N$, then there exists a simple object $M$ in $\A$ satisfying
\bna
\item $\FQ_N(M) \simeq Y$,
\item $b_k-a_k+1\le N-1$ for $1 \le k \le r$, where $([a_1,b_1],\ldots,[a_r,b_r])$ is the multisegment associated with $M$.
\end{enumerate}
\item \label{item:2}Let $([a_1,b_1],\ldots,[a_r,b_r])$ be the multisegment associated with a simple object $M$ in $\A$.
\bna
\item if $b_k- a_k + 1 \le N$ for any $1 \le k \le r$, then $\FO_N(M)$ is simple in $\T_N$,
\item if $b_k- a_k + 1 >N$ for some $k$, then $\FO_N(M)$ vanishes,
\item if $b_k- a_k + 1 = N$ for any $1 \le k \le r$, then $\FO_N(M) \simeq \one$.
\ee
\item By the correspondence in \eqref{item:2},
the set of isomorphism classes of self-dual simple objects
of $\T_N$ is isomorphic to the set of multisegments
$([a_1,b_1],\ldots,[a_r,b_r])$ with
$b_k- a_k + 1<N$ $(1 \le k \le r)$.
\end{enumerate}
\end{proposition}

\Prop\label{prop:PBWprod}
Let $M_k$ be a real simple $R$-module for $1\le k\le r$.
Let $m_k$, $n_k\in\Z_{\ge0}$ $(k=1,\ldots,r)$.
We assume that $\tL(M_i,M_j)=0$ if $1\le i<j\le r$.
Then we have
\bnum
\item
$M\seteq\hd\bl M_1^{\circ m_1}\conv\cdots\conv M_r^{\circ m_r}\br$
and $N\seteq\hd\bl M_1^{\circ n_1}\conv\cdots\conv M_r^{\circ n_r}\br$
are simple modules,
\item $L\seteq\hd\bl M_1^{\circ m_1+n_1}\conv\cdots\conv M_r^{\circ m_r+n_r}\br$
is isomorphic to a simple subquotient of $M \conv N$
up to a grading shift.
\ee
\enprop

\Proof
(i) follows from \cite[Corollary 2.10, Lemma 2.6]{KK18} and
$\tL(M_i^{\circ m_i},M_j^{\circ m_j})=0$ for $i<j$.

\smallskip\noi
(ii)\ In the course of the proof, we ignore grading shifts.
Set $L_{2k-1}=M_k^{\circ m_k}$ and $L_{2k}=M_k^{\circ n_k}$.
Set $\beta_k=-\wt(L_k)\in\rl^+$ and $\ell_k=|\beta_k|$,
$m=|-\wt(M)|=\sum_{k=1}^r\ell_{2k-1}$,
$n=|-\wt(N)|=\sum_{k=1}^r\ell_{2k}$.
Then $M\simeq \hd(L_1\conv L_3\conv\cdots \conv L_{2r-1})$,
$N\simeq \hd(L_2\conv L_4\conv\cdots \conv L_{2r})$ and
$L\simeq \hd(L_1\conv L_2\conv\cdots \conv L_{2r})$.
Let $\rmat{L_i,L_j}\col L_i\conv L_j\to L_j\conv L_i$ be the $R$-matrix.
Then if $i<j$ and $i\equiv 0$, $j\equiv 1\mod 2$, then
we have $\tL(L_i,L_j)=0$ and hence
$$\rmat{L_i,L_j}(u\etens v)-\tau_{w[\ell_j,\ell_i]} (v\etens u)
\in\sum_{ w< w[\ell_j,\ell_i]}\tau_w (L_j\etens L_i)\qt{for $u\in L_i$ and $v\in L_j$.} $$
Here, for $\beta\in\rl^+$ and $w\in\sym_{\height{\beta}}$,
$\tau_w=\tau_{i_1}\cdots\tau_{i_\ell}\in R(\beta)$ where
$s_{i_1}\cdots s_{\i_\ell}$ is a reduced expression of $w$.
The element $w[a,b]\in\sym_{a+b}$ is defined by
$w[a,b](k)=k+b$ if $1\le k\le a$ and
$w[a,b](k)=k-a$ if $a<k\le a+b$.

Then we have a homomorphism
\eqn
L_1\conv L_2\conv \cdots \conv L_{2r}
\To[{\mathbf r}] \bl L_1\conv L_3\conv\cdots \conv L_{2r-1}\br
\conv \bl L_2\conv L_4\conv\cdots\conv L_{2r}\br
\eneqn
by a product of $\rmat{L_i, L_j}$'s.
Hence we have
\eqn
&&{\mathbf r}(u_1\etens\cdots\etens  u_{2r})
-\tau_{w_1}\bl(u_2\etens\cdots\etens  u_{2r})\etens
(u_1\etens \cdots \etens u_{2r-1})\br\\*
&&\hs{25ex}\in \sum_{w<w_1}\tau_w\bl L_1\conv L_3\conv\cdots \conv L_{2r-1}\br
\conv \bl L_2\conv L_4\conv\cdots\conv L_{2r}\br
\eneqn
Here $w_1\in \sym_{m+n}$  is a product of $w[\ell_j,\ell_i] $'s.
Note that $w_1$ is given explicitly
$$w_1(c)=
\bc
c+\dsum_{1\le k<s}\ell_{2k}&\text{if $\dsum_{1\le k<s}\ell_{2k-1}<c\le
\dsum_{1\le k\le s}\ell_{2k-1}$, $1\le s\le r$,}\\[2ex]
c-\dsum_{s< k\le r}\ell_{2k-1}&\text{if $ m +  \dsum_{1\le k<s}\ell_{2k}<c\le
 m+\dsum_{1\le k\le s}\ell_{2k}$, $1\le s\le r$.}
\ec$$
Now recall the shuffle lemma:
\begin{align}\label{eq:shuffle}
M\conv N =\soplus_{w\in\sym_{ m,n }}\tau_w(M\etens N).
\end{align}
Here $\sym_{m,n}=\set{w\in\sym_{m+n}} {\text{$w(k)<w(k+1)$ if $1\le k<n+m$ and $k\not= m $}}$.
Note that $w_1 \in \sym_{m,n}$.
Since
$L_1\tens L_3\tens\cdots \tens L_{2r-1}\subset M$
and $L_2\tens L_4\tens\cdots \tens L_{2r}\subset N$,
the shuffle lemma \eqref{eq:shuffle} says that the composition
\eqn
L_1\tens L_2\tens\cdots \tens L_{2r}
\to L_1\conv L_2\conv \cdots \conv L_{2r}
&\to& \bl L_1\conv L_3\conv\cdots \conv L_{2r-1}\br
\conv \bl L_2\conv L_4\conv\cdots\conv L_{2r}\br\\
&\to& M\conv N
\eneqn
is injective.
Hence we obtain a non-zero homomorphism $$L_1\conv L_2\conv \cdots \conv L_{2r}
\to M\conv N.$$ Since $L_1\conv L_2\conv \cdots \conv L_{2r}$ has a simple head $L$,
the second assertion follows.
\QED

\Cor\label{cor:pbwprod}
Let $[a_1,b_1],\ldots, [a_r,b_r]$ be a sequence of segments such that
$[a_k,b_k]>[a_{k+1},b_{k+1}]$.
Let $m_k$, $n_k\in\Z_{\ge0}$ $(k=1,\ldots,r)$.
If $M\seteq\hd\bl(L[a_1,b_1])^{\circ m_1}\conv\cdots\conv(L[a_r,b_r])^{\circ m_r}\br$
and $N\seteq\hd\bl(L[a_1,b_1])^{\circ n_1}\conv\cdots\conv(L[a_r,b_r])^{\circ n_r}\br$
strongly commute, then
$M\conv N$ is isomorphic to
$\hd\bl(L[a_1,b_1])^{\circ m_1+n_1}\conv\cdots\conv(L[a_r,b_r])^{\circ m_r+n_r}\br$
up to a grading shift.
\encor
\Proof
This follows from Proposition~\ref{prop:PBWprod} and
$\tL(L[a_i,b_i],L[a_j,b_j])=0$ if $i<j$.
\QED

For $p \in \K$, let us denote by $\M_N(p,0)\in \T_N$ the image of $\M_{\ttww}(p,0)$ under $\FO_N$; i.e.,
$$\M_N(p,0)  \seteq \FO_N(\M_{\ttww}(p,0))\in\T_N.$$

Since $\M_\ttww(p,0)$ is associated to an ordered multisegment $\lb \ell,m \rb^{(j)}$ consisting of segments $[a,b]$ with $b-a+1=\ell$, where
$c(p)=(\ell,m) \in \Z_{\ge 1} \times \Z_{\ge 1}$,
Proposition~\ref{prop: le N do not vanish} implies
the following lemma.
\Lemma\hfill \label{lem: consequence via functor}
\bnum
\item If $p\in K_{\le N}$, then $\M_N(p,0)$ is simple.
\item If $p\in K_{=N}=\Kfr_N$, then $\M_N(p,0)\simeq\one$.
\item If $p\in K_{\ge N+1}$, then $\M_N(p,0)$ vanishes.
\item For $\{m_p\}_{p\in\Kex_N}\in \Z_{\ge0}^{\oplus \Kex_N}$
and $\{n_p\}_{p\in\Kex_N}\in \Z_{\ge0}^{\oplus \Kex_N}$, assume that
$$\text{ $\sodot_{p\in\Kex_N}\M_N(p,0)^{\snconv m_p}$ and $\sodot_{p\in\Kex_N}\M_N(p,0)^{\snconv n_p}$
are isomorphic.}$$
Then we have $m_p=n_p$ for all $p\in\Kex_N$.
\ee
\enlemma

\Proof
(i)--(iii) follow from Proposition~\ref{prop: le N do not vanish}.
(iv) follows from Corollary~\ref{cor:pbwprod} and
Proposition~\ref{prop: le N do not vanish}.
\QED

Set $$\rl_{J,N}\seteq\rl_J/\sum_{a\in \Z}\Z(\ep_a-\ep_{a+N})
\subset \wl_{J,N}\seteq \dfrac{\soplus\nolimits_{a\in\Z}\Z\ep_a}{\sum_{a\in \Z}\Z(\ep_a-\ep_{a+N})}
.$$
Let $\pi_N\col \rl_J\to\rl_{J,N}$ denote the projection.
Let $\bal_a\in \rl_{J,N}$ be the image of $\al_a$ in $\rl_{J,N}$.
Similarly, let $\bep_a$ denote the image $\ep_a$ in $\wl_{J,N}$.
Then $\bal_{a+N}=\bal_a$, $\rl_{J,N}=\soplus_{a=1}^{N-1}\Z\bal_a$
 and $\wl_{J,N}=\soplus_{a=0}^{N-1}\Z\ep_a$.

The category $\T_N$ is graded by $\rl_{J,N}$, i.e., it has a decomposition
$$\T_N=\soplus_{\al\in \rl_{J,N}}(\T_N)_\al.$$
The duality functor $*$ of $\A$ induces a duality functor $*$ of $\T_N$ which satisfies
\begin{align*} 
(X\star Y)^*\simeq q^{(\al,\beta)_N} Y^*\star X^*.
\end{align*}
Here $(\cdot,\cdot)_N$ is defined by $(\bep_a,\bep_b)_N=\true(a\equiv b\mod N)$.

Let $\tB_N=(b_{ij})_{(i,j)\in\Kex_N\times \Kex_N}$ be the exchange matrix associated
with the quiver  $\overline{\tQ_N}$.
Let $L_N=(\la_{ij})_{i,j\in \Kex_{N}}$ be the skew-symmetric matrix defined by
$$\M_N(i,0) \starop M_N(j,0)\simeq q^{-\la_{ij}}
\M_N(j,0)\starop M_N(i,0).$$

\begin{theorem} \label{thm: monoidal step1}
$\Seed_N = \left(\bigl\{ \M_N(p,0) \bigr\}_{p \in \Kex_{N}}, \tB_N \right)$
is a quantum monoidal seed of $\T_N$ and admits successive mutations in all directions.
\end{theorem}

\begin{proof}
Since the category $\A$ gives a monoidal categorification of the quantum cluster algebra $\mathscr{A}_{q^{1/2}}(\infty)$,
we can make successive mutation
on the initial quantum monoidal seed $\Seed_\infty$ of $\A$.
Let $\Sigma$ be a sequence of
at vertices in $\Kex_N$.
Set
$\Seed \seteq \mu_\Seq(\Seed_\infty) = (\{M_i\}_{i \in \K}, (\tb_{ij})_{(i,j) \in \K \times \K})$.

We shall show that
\begin{eqnarray*} &&
\parbox{75ex}{
$\Omega_N(\Seed)\seteq\bl\{\Omega_N(M_i)\}_{i\in \Kex_N},
(\tb_{ij})_{(i,j) \in \Kex_N \times \Kex_N}\br$, obtained from $\mu_\Seq(\Seed_N)$, is a quantum monoidal seed in $\T_N$.}
\end{eqnarray*}
by applying the induction of the length of $\Sigma$.

In order to see this, by arguing by induction on the length of $\Sigma$,
it is enough to show the following statement:
\eq
\parbox{\mylength}{
let  $\Seed$ be a quantum monoidal seed $(\{M_i\}_{i\in K},\tB)$ in $\A$ such that
$\Omega_N(\Seed)$ is a quantum monoidal seed in $\T_N$. Then for $p\in\Kex_N$,
$\Omega_N(\mu_p \Seed)$ is the mutation of $\Omega_N(\Seed)$ at $p$.}
\eneq

Set $\mu_p(\Seed)=(\{M'_i\}_{i \in \K}, (\tb'_{ij})_{(i,j) \in \K \times \Kex})$. Then $M'_i=M_i$ for $i\not=p$, and we have an exact sequence
\begin{align}\label{eq:ses graded mutation1 prime}
0 \to q \sodot_{\tb_{ip} >0} M_i^{\snconv \tb_{ip}} \to q^{m_p} M_p \conv M_p' \to
 \sodot_{\tb_{ip} <0} M_i^{\snconv (-\tb_{ip})} \to 0,
\end{align}
in $\A$.

Note that we have
$$\text{$M_i=M_\ttww(i,0)$ for $i\in K\setminus \Kex_N$ and
$\FO_N(M_i)\simeq\one$ for $i\in\K_{=N}$.}$$

By Lemma~\ref{lem: in out},
any index $i$ such that $\tb_{ip} \ne 0$ satisfies $i \in \K_{\le N}$.
Hence we obtain
$$\sodot_{\tb_{ip} >0} \Omega_N(M_i)^{\snconv \tb_{ip}}
\simeq\hs{-2ex} \sodot_{i\in\Kex_N,\;\tb_{ip} >0} \hs{-2ex}\Omega_N(M_i)^{\snconv \tb_{ip}}
\qtq \sodot_{\tb_{ip} <0} \Omega_N(M_i)^{\snconv (-\tb_{ip})}\simeq\hs{-2ex}\sodot_{i\in\Kex_N,\,\tb_{ip} <0}\hs{-2ex} \Omega_N(M_i)^{\snconv (-\tb_{ip})}
$$
hold in $\T_N$,

Applying the exact functor $\FO_N$ to \eqref{eq:ses graded mutation1 prime}, we obtain an exact sequence in $\T_N$.
Hence $\bl\{\FO_N(M'_i)\}_{i \in \Kex_N}, (\tb'_{ij})_{i,j \in \Kex_N}\br$
is the mutation of $\Omega_N(\Seed)$ at $p$.
The conditions (iv), (v), (vi) can be checked by using the similar arguments in \cite[Proposition 7.1.2]{KKKO18}.
Thus our assertion follows.
\end{proof}

\begin{proposition} \label{pro: mu S+ S- in TN}
For each $p \in \Kex_N$ with $c(p)=(\ell,m)$ and $r \in \Z_{\ge 1}$, we have
$$ \mu^{(r)}_{\Spl_N}\bigl(\M_N(p,0)\bigr) \simeq \FO_N\bigl(\W^{(\ell)}_{m,\jj_p+r}\bigr) \ \   \text{ and } \ \  \mu^{(r)}_{\Smi_N}\bigl(\M_N(p,0)\bigr) \simeq \FO_N\bigl(\W^{(\ell)}_{m,\jj_p-r}\bigr) \text{ in $\T_N$},$$
up to grading shifts.
\end{proposition}

\begin{proof}
We will proceed by induction on $r$.  When $r=0$, it is the definition.
Assume that $r >0$.
Note that  every  index $(\ell,m) \in c(\K)$ appears at most once in the sequence of mutations $\mu_{\Spl}$ and $\mu_{\Spl_N}$.
Let $\mu_{\Spl}(\ell,m)$ and $\mu_{\Spl_N}(\ell,m)$ be the subsequences of  $\mu_{\Spl}$ and $\mu_{\Spl_N}$, respectively, consisting of the mutations  preceding  the one at the vertex $(\ell,m)$.
 By induction, we assume that
 $\mu_{\Spl_N}(\ell,m) \circ \mu^{(r-1)}_{\Spl_N}\bigl(\M_N(p',0)\bigr) \simeq \FO_N\bigl(\W^{(\ell')}_{m',\jj_{p'}+r-1}\bigr)$ for any $p' \in \Kex_N$ such that either $p'=p$ or $c(p')=(\ell',m')$ is subsequent to $c(p)$ in $\mu_{\Spl_N}$, and we assume that
  $\mu_{\Spl_N}(\ell,m) \circ \mu^{(r-1)}_{\Spl_N}\bigl(\M_N(p'',0)\bigr) \simeq \FO_N\bigl(\W^{(\ell'')}_{m'',\jj_{p''}+r}\bigr)$ for any $p''=(\ell'',m'') \in \Kex_N$ such that $c(p'')$ is  preceding  $c(p)$ in $\mu_{\Spl_N}$.

Note that we have
 $\overline {(\mu_{\Spl}(\ell,m)\bigl(\tQ \bigr))_N}= \mu_{\Spl_N}(\ell,m)\bigl(\overline { \tQ _N}\bigr)$
 for all $(\ell,m)\in c(\Kex_N)$.
 Indeed, the mutation at a vertex in $\K_{\ge N+1}$ does not affect the arrows between the vertices in $\Kex_N$ by Lemma \ref{lem: in out},
 and for any $\tilde m$,  the quiver $\mu_{\Spl}(N,\tilde m)(\tQ)$   has  only one arrow connecting the vertex $(N,\tilde m)$ with the ones in $\Kex_N$  so that the mutation at $(N,\tilde m)$ does not affect the  arrows between the vertices in $\Kex_N$, either.

 It follows that
\eqn
0 \to  \FO_N(\Wlmj{\ell}{m-1}{{\jj_p+r}})  \tens  \FO_N( \Wlmj{\ell}{m+1}{{\jj_p+r-1}} )
\to   \FO_N\bigl(\W^{(\ell)}_{m,\jj_{p}+r-1}\bigr)  \tens   \mu^{(r)}_{\Spl_N}\bigl(\FO_N(M(p,0))\bigr) \\  \to   \FO_N(\Wlmj{\ell-1}{m}{{\jj_p+r-1}} ) \tens \FO_N( \Wlmj{\ell+1}{m}{{\jj_p+r}} )  \to 0
\eneqn
up to a power of $q$.
Note that when $\ell=N-1$,  the fourth term should be $\FO_N(\Wlmj{\ell-1}{m}{{\jj_p+r-1}} )$, but it is isomorphic to $ \FO_N(\Wlmj{\ell-1}{m}{{\jj_p+r-1}} ) \tens \FO_N( \Wlmj{\ell+1}{m}{{\jj_p+r}} )$,
since we have $ \FO_N( \Wlmj{\ell+1}{m}{{\jj_p+r}} )\simeq \one$.
By applying $\FO_N$ to
the exact sequence~\eqref{eq: T-system in qHecke} and comparing it with the above exact sequence, we get
\eqn
\mu_{\Spl_N}^{(r)}\bigl(\FO_N(M(p,0))\bigr)=\mu_{(\ell,m)}\bigl( \FO_N\bigl(\W^{(\ell)}_{m,\jj_{p}+r-1}\bigr)\bigr)
\simeq \FO_N\bigl( \Wlmj{\ell}{m}{\jj_p+r}  \bigr),
\eneqn
as desired.
\end{proof}

Let $\mathscr{A}_{q^{1/2}}(N)$ be the quantum cluster algebra whose initial quantum seed is given as follows:
$$ [\Seed_N] = \left( \bigl\{  q^{-\frac{1}{4}(d_p,d_p)_N}[\M_N(p,0)] \bigr\}_{p \in \Kex_N}, -L_N,\tB_N \right).$$
Note that  $\bigl\{  q^{-\frac{1}{4}(d_p,d_p)_N}[\M_N(p,0)] \bigr\}_{p \in \Kex_N}$
is algebraically independent by Lemma~\ref{lem: consequence via functor}.

By Theorem~\ref{thm: monoidal step1}, we can conclude the following:
\begin{corollary}
$\mathscr{A}_{q^{1/2}}(N)$ is a subalgebra in $\Ah \tens_{\Z[q^{\pm 1}]} K(\T_N)$.
\end{corollary}

Now, we shall show that indeed
$$\text{$\mathscr{A}_{q^{1/2}}(N)$ coincides with $\Ah \tens_{\Z[q^{\pm 1}]} K(\T_N)$.}$$

\begin{remark}[{see Lemma~\ref{lem: seq mu in TN} and also \cite[Remark 3.3]{HL16}}]\label{eq: finite sequence}
For each $p\in\Z_{\ge 1}$ with $c(p)=(\ell,m)$ and $r \in\Z_{\ge 1}$, there exists {\it finite} sequences $\Seq_1$ and $\Seq_2$ of mutations satisfying
$$ \mu_{\Seq_1}(\M_N(p,0)) = \FO_N(\W^{(\ell)}_{m,\jj_p+r})  \ \ \text{ and } \ \ \mu_{\Seq_2}(\M_N(p,0)) = \FO_N(\W^{(\ell)}_{m,\jj_p-r}).$$
\end{remark}

By~\cite[Proposition 4.31]{KKK18A}, the Grothendieck ring $K(\T_N)$ of $\T_N$
is generated by $[\FO_N(L[a,b])]$ ($b-a+1<N$) as a $\Z[q^{\pm 1}]$-algebra.

On the other hand,
since $L[a,b]= \W^{(b-a+1)}_{1,b}$, Proposition~\ref{pro: mu S+ S- in TN} and
Remark~\ref{eq: finite sequence} imply  that every $[\FO_N(L[a,b])]$ with $b-a+1<N$
appears as a cluster variable and hence it
is contained in $\mathscr{A}_{q^{1/2}}(N)$.

Thus we have
\begin{theorem} \label{thm: monoidal step2}
 As $\Ah$-algebras, we have an isomorphism
$$\mathscr{A}_{q^{1/2}}(N) \simeq \Ah \tens_{\Z[q^{\pm 1}]}K(\T_N).$$
\end{theorem}

By Theorem~\ref{thm: monoidal step1} and Theorem~\ref{thm: monoidal step2},
we obtain the following result.

\begin{theorem}\label{th:monoidalTN}
The category $\T_N$ gives a monoidal categorification of the quantum cluster algebra $\mathscr{A}_{q^{1/2}}(N)$. Hence, we have
\begin{enumerate}
\item[{\rm (i)}] each cluster monomial in $\mathscr{A}_{q^{1/2}}(N)$ corresponds to the isomorphism class of a real simple object of $\T_N$ up to a power of $q^{1/2}$,
\item[{\rm (ii)}] each cluster monomial in $\mathscr{A}_{q^{1/2}}(N)$ is a Laurent polynomial of the initial cluster variables with coefficient in $\Z_{\ge 0}[q^{\pm1/2}]$.
\end{enumerate}
\end{theorem}

\section{Main result} \label{Sec: Main}
In this section, we first assume that we have a family of
 quasi-good modules with {\it certain conditions} which induces a generalized Schur-Weyl duality functor $\F$ of type $\Ainf$,
and investigate properties of $\F$ under the assumption. Then we will prove that some subcategories in $R^J\gmod$ and $\Ca_\g$
give monoidal categorifications of quantum cluster algebras. In the last part, we give families of distinct quasi-good modules for quantum affine algebras
of type $A$, $B$, $C$ and $D$, satisfying the conditions. Note that the families for the types $A$ and $B$  were taken from \cite{KKK18A, KKKO15C}, and \cite{KKO18}, but the ones for types $C$ and $D$ are  new.

\subsection{Schur-Weyl duality functor of type $\Ainf$}
\label{subsec A_inf}

Let $U'_q(\g)$ be a quantum affine algebra
as in \S\;\ref{subsec:quntum affine}.
Let $J=\Z$ be the index set as in the previous section.
We assume that
there exists a family of quasi-good modules $\{V_a\}_{a \in J}$ in $\Ca_\g$ satisfying the following properties:
\begin{eqnarray} &&\hs{2ex}
\parbox{81ex}{
\begin{enumerate}[{\rm(a)}]
\item The quiver $\Gamma^J$ in~\eqref{eq: GammaJ} is of type $\Ainf$.
\label{it:a}
\item There exists an integer $N\ge 2$ such that the followings are satisfied
 for  any $a \in J$, \label{it:b}
\begin{enumerate}[(1)]
\item the head of $V_a \tens V_{a+1} \tens \cdots \tens V_{a+N-1}$ is isomorphic to $\cor$,\label{it:1}
\item $\hd(V_a \tens V_{a+1} \tens \cdots \tens V_{a+k-1}) \tens V_{a+k}$ is not simple  for $1 \le k \le N-1$, \label{it:2}
\item ${}^{**}V_a \simeq V_{a+N}$. \label{it:3}
\end{enumerate}
\end{enumerate}
}
\label{eq: general assumption}
\end{eqnarray}

With the assumption \eqref{it:a} in~\eqref{eq: general assumption}, we have the exact functor in Section~\ref{Subsec: functors}
$$\F \col \soplus_{\beta \in \rl^+_J} R^J( \beta) \gmod \to \Ca_\g.$$
Here $R^{J}(\be)$ is the quiver Hecke algebra of type $\Ainf$ defined in Section~\ref{subsec: Ainfty}.

\begin{lemma} \label{lem: N> to 0}
The family $\{V_a\}_{a\in \Z}$ has the following properties.
\bnum
\item $V_a\not\simeq V_b$ if $a\not=b$,
\item $\hd(V_{a} \tens V_{a+1} \tens \cdots \tens V_{a+\ell})$ is simple
if $0\le \ell\le N-1$,
\item
for any $a \in J$ and $\ell \in \Z_{\ge 1}$,
$$\F(L[a,a+\ell-1]) \simeq  \bc
\hd(V_{a} \tens V_{a+1} \tens \cdots \tens V_{a+\ell-1})
&\text{if $\ell\le N$,}\\
0&\text{otherwise.}
\ec
$$
\item $\F(L[a,a+N-1]) \simeq \cor$.
\ee
\end{lemma}

\begin{proof}
(i) Assume $a<b$. Then $V_{a-1}\tens V_a$ is not simple, but $V_{a-1}\tens V_{b}$ is simple by \eqref{eq: general assumption}~\eqref{it:a}.

\smallskip\noi
(ii) follows from  \eqref{eq: general assumption}~\eqref{it:1}.

\smallskip\noi
(iii) We first prove the statement for $\ell\le N$ by induction on $\ell$.
When $\ell=1$, $\F(L(a)) \simeq V_a$ by Proposition~\ref{prop:image of tau} {\rm (a)}. By induction on $\ell$, we may assume that
$\F(L[a,a+\ell-2])\simeq \hd(V_{a} \tens V_{a+1} \tens \cdots \tens V_{a+\ell-2})$
if $\ell\le N$.
Since $\de(L[a,a+\ell-2],L(a+\ell-1)) = 1$ by~\cite[Theorem 4.3]{KKK18A}
and  $L[a,a+\ell-1] \simeq L[a,a+\ell-2] \hconv L(a+\ell-1)$
the condition \eqref{eq: general assumption}~\eqref{it:2} and Lemma \ref{lem: non-simple tensor} imply
\begin{align*}
\F(L[a,a+\ell-1]) & \simeq \hd(V_{a} \tens V_{a+1} \tens \cdots \tens V_{a+\ell-2}) \hconv V_{a+\ell-1} \\
&  \simeq \hd(V_{a} \tens V_{a+1} \tens \cdots \tens V_{a+\ell-2} \tens V_{a+\ell-1}).
\end{align*}
Hence we have proved (iii) when $\ell\le N$.
In particular, $\F(L[a,a+N-1]) \simeq \cor$ by \eqref{it:1}.
Hence we have
\begin{align*}
& V_a\simeq \F(L(a))\simeq \F(L(a)) \tens \F(L[a+1,a+N])  \twoheadrightarrow \F(L[a,a+N])\\
& V_{a+N}\simeq\F(L(a+N))\simeq \F(L[a,a+N-1]) \tens \F(L(a+N)) \twoheadrightarrow \F(L[a,a+N])
\end{align*}
for all $a \in \Z$.  If $\F(L[a,a+N])$ did not vanish, then we
would have
$$V_a \simeq \F(L[a,a+N])\simeq V_{a+N},$$
which contradicts (i).

Hence we conclude that
$$\F(L[a,a+N])\simeq0 \qquad \text{ for all } a \in \Z.$$
Now assume that $\ell>N+1$. Then we have a surjective homomorphism
$$0 \simeq \F(L[a,a+N]) \tens \F(L[a+N+1,a+\ell-1]) \twoheadrightarrow \F(L[a,a+\ell-1]),$$
which yields $\F(L[a,a+\ell-1]) \simeq 0$.
\end{proof}

Let us denote by $\Ca_J$ is the smallest abelian full subcategory of  $\Ca_\g$ such that
\begin{enumerate}
\item[{\rm (a)}] $\Ca_J$ contains $\{V_a \}_{a \in J}$,
\item[{\rm (b)}] it is stable under taking submodules, quotients, extensions and tensor products.
\end{enumerate}
Then we have an exact functor
\begin{align}\label{eq: the F}
\F \col \soplus_{\beta \in \rl^+_J} R^J( \beta) \gmod \to \Ca_J \subset \Ca_\g.
\end{align}

Now we shall show that the functor $\F$ under the assumption~\eqref{eq: general assumption} factors through the
category $\T_N$ with a {\it suitable} choice of  duality coefficient $\{P_{i,j}(u,v)\}_{i,j\in J}$. To do that, we employ the framework of \cite[\S 2.6]{KKO18}.

\begin{lemma}
For $a,b \in J$, set $\tz=z_{V_b}$,
\begin{align*}
&\Rnorm_{a,b}= (\Rnorm_{V_a,(V_b)_{\tz} }  \tens 1 \tens \cdots \tens 1) \circ
(1 \tens \Rnorm_{V_{a+1},(V_b)_{\tz} } \tens 1 \tens \cdots \tens 1) \circ
\cdots \\
&\hskip 1em  \circ (1\tens \cdots \tens 1 \tens  \Rnorm_{V_{a+k}, (V_b)_{\tz}  }  \tens 1 \tens \cdots \tens 1)
\circ \cdots \circ  (1\tens \cdots \tens 1 \tens  \Rnorm_{V_{a+N-1},(V_b)_{\tz}  })
\end{align*}
and
\begin{align*}
g_{a,b}(z) \seteq \prod_{k=0}^{N-1} a_{V_a,V_b}(1+z)^{-1}
\end{align*}
where $z=\tz-1$ and $a_{M,N}(z_N)$ denotes the ratio
of $\Runiv_{M,N}$ and $\Rnorm_{M,N}$ in~\eqref{eq: aMN}.
Then the following diagram is commutative
\begin{align*}
\scalebox{0.93}{\xymatrix{
\bigl( V_a \tens \cdots \tens V_{a+N-1} \bigr) \tens (V_b)_{\tz} \ar[r]^{\Rnorm_{a,b}}
\ar@{>>}[d]^{\varphi \tens (V_b)_{\tz} }
& (V_b)_{\tz} \tens  \bigl( V_a \tens \cdots \tens V_{a+N-1} \bigr)
\ar@{>>}[d]^{(V_b)_{\tz} \tens \varphi}  \\
\cor \tens (V_b)_{\tz}  \ar[d]_*[@]{\sim}& (V_b)_{\tz} \tens \cor \ar[d]_*[@]{\sim} \\
(V_b)_{\tz} \ar[r]^{ g_{a,b}(z) } & (V_b)_{\tz}
}}
\end{align*}
for any surjective  homomorphism $\varphi\col  V_a \tens \cdots \tens V_{a+N-1}  \epito \cor$.
\end{lemma}

\begin{proof}
Let
\begin{align*}
&\Runiv_{a,b}= (\Runiv_{V_a,(V_b)_{\tz} }  \tens 1 \tens \cdots \tens 1) \circ
(1 \tens \Runiv_{V_{a+1},(V_b)_{\tz} } \tens 1 \tens \cdots \tens 1) \circ
\cdots \\
 & \hspace{15ex} \circ \cdots \circ  (1\tens \cdots \tens 1 \tens  \Runiv_{V_{a+N-1},(V_b)_{\tz}  }).
\end{align*}
By replacing $\Rnorm_{a,b}$ and $g_{a,b}(z)$ with $\Runiv_{a,b}$ and $\id$ respectively, the diagram is commutative.

Then our assertion follows from $\Runiv_{a,b}=g_{a,b}(z)\Rnorm_{a,b}$.
\end{proof}

Now we temporarily fix a duality coefficient $\{P^\dT_{i,j}(u,v)\}_{i,j\in J}$ satisfying the conditions in~\eqref{eq:Pij}. We denote the corresponding functor from $\A/\Ser_N$ to $\Ca_J$ by $\F^\dT$.

Recall $L_a=L[a,a+N-1]$ and $f_{a,b}(z)$ in Definition~\ref{def: quot coeff}.

\begin{proposition} For $a,b \in J$, let
$$ h_{a,b}(z) \seteq f_{a,b}(z)g_{a,b}(z)\prod_{k=a}^{a+N-1} P^\dT_{k,b}(0,z).$$
Then the following diagrams are commutative$\col$
$$
\scalebox{0.94}{\xymatrix@C=5ex@R=6ex{
\Ft(L_a \starop L(b)_z)  \ar[r]^{\Ft(R_a(L(b)_z))} \ar[d]_*[@]{\sim}
& \Ft(L(b)_z \starop L_a) \ar[d]_*[@]{\sim} \\
\Ft(L_a) \tens \Ft(L(b)_z)    \ar[d]_*[@]{\sim}  ^{\varphi_a \tens \Ft(L(b)_z)}
& \Ft(L(b)_z) \tens  \Ft(L_a) \ar[d]_*[@]{\sim}  ^{ \Ft(L(b)_z) \tens \varphi_a}  \\
\cor \tens \Ft(L(b)_z) \ar[d]_*[@]{\sim}
& \Ft(L(b)_z) \tens \cor \ar[d]_*[@]{\sim} \\
\Ft(L(b)_z) \ar[r]^{h_{a,b}(z)}
& \Ft(L(b)_z)
}}
\hspace{-1ex}
\scalebox{0.94}{\xymatrix@C=5ex@R=6.4ex{
\Ft(L_a \starop L_b)  \ar[r]^{\Ft(R_a(L_b))} \ar[d]_*[@]{\sim} & \Ft(L_b \starop L_a) \ar[d] _*[@]{\sim}\\
\Ft(L_a) \tens \Ft(L_b)   \ar[d]_*[@]{\sim}^{\varphi_a \tens \varphi_b} & \Ft(L_b)  \tens \Ft(L_a) \ar[d]_*[@]{\sim}^{\varphi_b \tens \varphi_a}  \\
\cor \tens \cor  \ar[d]_*[@]{\sim} & \cor \tens \cor  \ar[d]_*[@]{\sim} \\
\cor  \ar[r]^{ \prod_{k=b}^{b+N-1} h_{a,k}(0)}  & \cor}}
$$
for any $\varphi_a\col\Ft(L_a) \isoto \cor$ and $\varphi_b\col\Ft(L_b) \isoto \cor$. Furthermore $h_{a,b}(z)$ has no poles and no zeros at $z=0$.
\end{proposition}

\begin{proof}
The proof is the same as one of \cite[Proposition 2.6.2, Corollary 2.6.3, Proposition 2.6.4]{KKO18}.
\end{proof}

\begin{corollary} \label{cor: la a b}
Set $\lambda_{a,b} \seteq \prod_{k=b}^{b+N-1} h_{a,k}(0)$. Then we have
$$  \lambda_{a,a}=1 \ \ (a \in J) \quad \text{ and } \quad  \lambda_{a,b}\lambda_{b,a}=1 \ \ (a,b \in J).$$
\end{corollary}

\begin{proof}
By Theorem~\ref{thm:L_a commutes with X} {\rm (ii)}, we have $R_a(L_a)=\id_{L_a \starop L_a}$. Thus $\Ft(R_a(L_a))$ is the identity map on $\cor$, which implies
$\lambda_{a,a}=1$ by the previous proposition. Similarly, Theorem~\ref{thm:L_a commutes with X} {\rm (iii)} implies the second assertion.
\end{proof}

\begin{lemma}[{\cite[Lemma 2.6.6, Lemma 2.6.7]{KKO18}}] \label{cor: la a b h a b}
Assume that
\begin{align*}
\lambda_{a,b}\seteq\prod_{k=b}^{b+N-1} h_{a,k}(0) \quad (a,b \in J)
\end{align*}
satisfies that
\begin{align*}
\lambda_{a,a}=1 \quad (a \in \Z) \quad \text{and} \quad \lambda_{a,b} \lambda_{b,a} =1 \quad (a,b \in J).
\end{align*}
Then there exists a  family $\{c_{a,b}(u,v)\ | \ a,b \in J \}$ of elements in $\cor[[u,v]]$
satisfying
\begin{equation} \label{eq: c ab}
\begin{aligned}
& c_{a,a}(u,v)=1 \quad  (a \in J), \quad c_{a,b}(u,v) c_{b,a}(u,v)=1 \quad (a,b \in J), \\
& \hspace{10ex} h_{a,b}(z)=\prod_{k=a}^{a+N-1} c_{k,b}(0,z) \quad (a,b \in J).
\end{aligned}
\end{equation}
\end{lemma}

By Lemma~\ref{cor: la a b h a b}, we have a family $\{c_{a,b}(u,v)\ | \ a,b \in J \}$ of elements in $\cor[[u,v]]$ satisfying~\eqref{eq: c ab}. Using $\{c_{a,b}(u,v)\ | \ a,b \in J \}$, we define
new duality coefficient $\{P^\new_{a,b}(u,v) \ | \ a,b \in J  \}$ as follows:
$$   P^\new_{a,b}(u,v) \seteq c_{a,b}(u,v)^{-1}P^\dT_{a,b}(u,v)=c_{b,a}(v,u)P^\dT_{a,b}(u,v).$$
We denote the corresponding functor from $\A/\Ser_N$ to $\Ca_J$ by $\F'$. Since $c_{a,b}(u,v)c_{b,a}(v,u)=1$, we have
$$  Q_{a,b}(u,v)=\true(a \ne b)P^\new_{a,b}(u,v)P^\new_{b,a}(v,u).$$

\begin{theorem}
Under the assumption of~\eqref{eq: general assumption}, there exists a duality coefficient $\{P^\new_{i,j}(u,v)\}_{i,j\in J}$ that makes the following diagram commutative$\colon$
\begin{align*}
\xymatrix@C=8.5ex{
\F'(L_a\starop M)\ar[d]^-{\F'(R_a(M))}\ar[r]^-\sim
&\F'(L_a)\tens\F'(M)\ar[r]^-{ g_a \tens \F'(M)}
&\cor\tens \F'(M)\ar[dr]\\
\F'(M\starop L_a)\ar[r]^-\sim&\F'(M)\tens\F'(L_a)\ar[r]^-{\F'(M)\tens
g_a} &\F'(M)\tens  \cor  \ar[r]&\F'(M) }
\end{align*}
for any $a \in J$, $M \in \A/\Ser_N$ and an isomorphism $g_a \col\F'(L_a)\isoto\cor$.
\end{theorem}

\begin{proof}
It is enough to show that our assertion holds when $M=L(b)_z$ for some $b \in J$. Equivalently, it is enough to check that
$f_{a,b}(z) g_{a,b}(z) \prod_{k=a}^{a+N-1} P^{\rm{new}}_{k,b}(0,z)=1$. By the choice of $\{c_{a,b}(u,v)\ | \ a,b \in J \}$, we have
$$
\displaystyle f_{a,b}(z) g_{a,b}(z) \prod_{k=a}^{a+N-1} P^{\rm{new}}_{k,b}(0,z)
=h_{a,b}(z) \prod_{k=a}^{a+N-1} c_{k,b}(0,z)^{-1} =1,$$
by \eqref{eq: c ab}. Thus our assertion follows.
\end{proof}

Now, \cite[Proposition A.11, Proposition A.12]{KKK18A} imply the following theorem:
\begin{theorem} \label{thm: Factoring}
Under the assumption of~\eqref{eq: general assumption},
there exists an exact monoidal functor $\tFF \col \T_N\to \Ca_J$ such that the following diagram quasi-commutes$\colon$
\begin{equation*}
 \xymatrix@C+7ex{
\A\ar[r]^-{\FQ_N}\ar[drr]_-{\F}&\A/\Ser_N\ar[r]^{\Upsilon }\ar[rd]^(.55){\F'} &\T'_N\ar[d]\ar[r]^{ \Xi  }&\T_N\ar[dl]^-{\tFF}\\
&&\Ca_J\,.}
\end{equation*}
In particular, $\tFF$ induces a surjective ring homomorphism $\phi_{\tFF} \col K(\T_N)\vert_{q=1} \twoheadrightarrow K(\Ca_J)$, where $K(\T_N)\vert_{q=1} \seteq K(\T_N)/(q-1)K(\T_N)$.
\end{theorem}

\begin{corollary}
The functor $\tFF$ sends a non-zero object in $\T_N$ to a non-zero module in $\Ca_J$. In particular, it sends a simple to a simple and $\Ca_J$ is rigid.
\end{corollary}
\begin{proof}
Let $M$ be a non-zero object in $\T_N$. Since $\T_N$ is rigid, there exists $M^* \in \T_N$ such that there is an epimorphism $M^* \starop M \twoheadrightarrow \cor$.
Since $\tFF$ is exact, we have an epimorphism $\tFF(M^*) \tens \tFF(M) \twoheadrightarrow \cor$. Hence $\tFF(M)$ is non-zero.
\end{proof}

Using the same argument in \cite[Lemma 2.6.11, Lemma 2.6.12, Theorem 2.6.13]{KKO18}, we can conclude the following:

\begin{theorem} \label{thm: Gro iso}
The functor $\tFF$ sends simples to simples bijectively and induces a ring isomorphism
$$ \phi_{\tFF} \col K(\T_N)\vert_{q=1} \isoto K(\Ca_J).$$
\end{theorem}

Recall that we have an exact functor $\F\col \A\to \Ca_J$.
Theorem~\ref{thm: Gro iso} along with Theorem~\ref{th:monoidalTN}
implies the following theorem:

\begin{theorem} \label{Thm: main}
Let $\mathscr{A}(N)$ be the cluster algebra whose initial seed is given as follows:
$$\left( \left\{   \left[\F(\M_\ttww(p,0)\right] \right\}_{p \in \Kex_N},\tB_N \right).$$
Then the category $\Ca_J$ gives a monoidal categorification of the cluster algebra $\mathscr{A}(N)$. Hence, we have
\begin{enumerate}
\item[{\rm (i)}] any cluster monomial in $\mathscr{A}(N)$ corresponds to the isomorphism class of a real simple object in $\Ca_J$,
\item[{\rm (ii)}] any cluster monomial in $\mathscr{A}(N)$ is a Laurent polynomial of the initial cluster variables with coefficient in $\Z_{\ge 0}$.
\end{enumerate}
\end{theorem}

\subsection{The category $\Ca_J$} \label{Sec: Ca_J}
 In this subsection,
we shall give explicitly families of quasi-good modules $\{V_a\}_{a \in J}$ which satisfy~\eqref{eq: general assumption}
for quantum affine algebras of type $A$, $B$, $C$ and $D$.
Thus
one can apply the results in the previous subsections. In particular, for quantum affine algebras of type $A$ and $B$, the category $\Ca_J$ coincides with the Hernandez-Leclerc category $\Ca_\g^0$
in Section~\ref{sec:HL}.

\subsubsection{Type A} \label{Sec: CJ for affine A}
Throughout this subsection, $\g^{(t)}$ denotes the affine Kac-Moody algebra of type $A^{(t)}_{N-1}$ $(t=1,2)$. Here $N\ge2$ if $t=1$ and
$N\ge 3$ if $t=2$.
Let us denote by $V^{(t)}(\varpi_i)$ the fundamental module over $\g^{(t)}$.
For each $a \in J=\Z$, define $V^{(t)}_a$ as follows(see \cite[\S 4.1]{KKK18A} and \cite[\S 3.1]{KKKO15C}):
$$  V^{(t)}_a \seteq V^{(t)}(\varpi_1)_{q^{2a}}.$$

Then, by Theorem~\ref{thm: denominators},~\cite[Lemma B.1]{AK} and~\cite[Theorem 3.5, Theorem 3.9]{Oh15}, we have the followings (see~\cite{KKK18A,KKKO15C} for details):
\begin{enumerate}[{\rm (i)}]
\item The family of quasi-good modules $\{ V^{(t)}_a \}_{a \in J}$ satisfies the conditions in~\eqref{eq: general assumption}.
\item The subcategory $\Ca_J$ coincides with the category $\Ca^{0}_{\g^{(t)}}$ in
\S\;\ref{sec:HL}.
\end{enumerate}

Then, by Theorem~\ref{thm: Factoring}, we have an exact functor
$$ \F^{(t)} \col \A=\soplus_{\be \in \rl^+_J} R^J(\be)\gmod \to \Ca^{0}_{\g^{(t)}}$$
which factors thorough $\T_N$ via $\tFF^{(t)}$
\begin{align*}
\raisebox{2em}{\xymatrix@C=8ex@R=3ex{\A\ar[rr]^-{\FO_N}\ar[dr]_-{\F^{(t)}}
&&\T_N.\ar[dl]^-{\tFF^{(t)}}\\
& \Ca^0_{\g^{(t)}}   }}
\end{align*}

For $i \in I_0$, $k \in \Z_{\ge 1}$ and $r \in \cor^\times$, let us denote by $W^{(i)}_{k,r}$  the {\it Kirillov-Reshetikhin module}
$$ W^{(i)}_{k,r} \seteq \hd(V^{(t)}(\varpi_i)_{rq^{2(k-1)}} \tens \cdots \tens  V^{(t)}(\varpi_i)_{rq^{2}} \tens V^{(t)}(\varpi_i)_{r}),   $$
which is known as a quasi-good module.

\begin{proposition}[{\cite[Proposition 3.9]{KKK18A},~\cite[Proposition 3.3]{KKKO15C}}]
For a segment $[a,b]$ with $\ell \seteq b-a+1 <N$, we have
$$ \F^{(t)}(L[a,b]) \simeq  \begin{cases}
V^{(t)}(\varpi_\ell)_{(-q)^{a+b}} & \text{ if } t=1, \text{ or } t=2 \text{ and } 1 \le \ell \le \lfloor N/2 \rfloor, \\
V^{(2)}(\varpi_{N-\ell})_{(-1)^N(-q)^{a+b}} & \text{ if } t=2 \text{ and } \lfloor N/2 \rfloor < \ell < N.
\end{cases}  $$
\end{proposition}

Thus we can conclude that the image of an $R^J$-module $\Wlmj{\ell}{m}{j}$ in~\eqref{eq: def Wlmj} by $\F^{(t)}$
is a Kirillov-Reshetikhin module:

\begin{corollary} For an $R^J$-module $\Wlmj{\ell}{m}{j}$ $(\ell <N)$, we have
$$ \F^{(t)}(\Wlmj{\ell}{m}{j}) \simeq  \begin{cases}
W^{(\ell)}_{m,\;(-q)^{2j-\ell+1}} & \text{ if } t=1, \text{ or } t=2 \text{ and } 1 \le \ell \le \lfloor N/2 \rfloor,\\[1ex]
W^{(N-\ell)}_{m,\;(-1)^N(-q)^{2j-\ell+1}} & \text{ if }  t=2 \text{ and } \lfloor N/2 \rfloor < \ell < N.\\
\end{cases}
$$
\end{corollary}

Then the exact sequence~\eqref{eq: T-system in qHecke} in $R^J\gmod$ can be translated into the exact sequence
\begin{equation}\label{eq: T-system A1}
\begin{aligned}
0 \to  W^{(\ell)}_{m-1,\;(-q)^{k+2}} \tens
W^{(\ell)}_{m+1,\;(-q)^{k}}  & \to  W^{(\ell)}_{m,\;(-q)^{k}} \tens
W^{(\ell)}_{m,\;(-q)^{k+2}} \\ & \hspace{-2.3ex} \to
W^{(\ell-1)}_{m,\;(-q)^{k+1}} \tens W^{(\ell+1)}_{m,\;(-q)^{k+1}}
\to 0
\end{aligned}
\end{equation}
in $\Ca_{\g^{(1)}}^{0}$ if $t=1$, and
the exact sequence 
\begin{equation}\label{eq: T-system A2}
\begin{aligned}
 0 \to  \pi^{(2)}\bigl(W^{(\ell)}_{m-1,(-q)^{k+2}}\bigr) \tens \pi^{(2)}\bigl(W^{(\ell)}_{m+1,(-q)^{k}}\bigr)
 &\to  \pi^{(2)}\bigl(W^{(\ell)}_{m,(-q)^{k}}\bigr) \tens \pi^{(2)}\bigl(W^{(\ell)}_{m,(-q)^{k+2}}\bigr) \\
& \hspace{-4.3ex} \to
\pi^{(2)}\bigl(W^{(\ell-1)}_{m,(-q)^{k+1}}\bigr) \tens
\pi^{(2)}\bigl(W^{(\ell+1)}_{m,(-q)^{k-1}}\bigr) \to 0
\end{aligned}
\end{equation}
in $\Ca_{\g^{(2)}}^{0}$ if $t=2$, where
$$\pi^{(2)}\bigl(W^{(\ell)}_{m,r}\bigr) \seteq \begin{cases}
W^{(\ell)}_{m,r} & \text{ if } 0 \le \ell \le \lfloor N/2 \rfloor, \\
W^{(N-\ell)}_{m,(-1)^{N-1}r} & \text{ if } \lfloor N/2 \rfloor< \ell \le N.
\end{cases} $$
Here we understand $W^{(0)}_{k,r}$ and $W^{(N)}_{k,r}$ as trivial modules.
Note that the exact sequences~\eqref{eq: T-system A1} and~\eqref{eq: T-system A2} are well-known as (twisted) T-system of $U_q'(A^{(t)}_{N-1})$~\cite{Her06,Her10,KNS,Nak}.

Since $\Ah \tens_{\Z[q^{\pm 1}]} K(\T_N)$ has a quantum cluster algebra structure, the isomorphism $\phi_{\tFF^{(t)}}$ in Theorem~\ref{thm: Gro iso} endows the {\it cluster algebra} structure on
$K\bigl(\Ca^{0}_{\g^{(t)}}\bigr)$ whose initial seed can be identified with
$$ [\Seed_N]_{q=1} \seteq \left( \bigl\{ [ \F^{(t)}(\Wlmj{\ell}{m}{\jj_p})  ] \bigr\}_{p \in \K_{\le N-1}}, \tB_N \right).$$

Now we can conclude the following theorem:

\begin{theorem} \label{thm apllication A}
The category $\Ca^0_{\g^{(t)}}$ $(t=1,2)$ gives a monoidal categorification of the cluster algebra $\mathscr{A}(N)$.
Hence we have the followings$\col$
\begin{enumerate}
\item[{\rm (i)}] Each cluster monomial in $\mathscr{A}(N)$ corresponds to an isomorphism class of a real simple object in $\Ca^{0}_{\g^{(t)}}$.
\item[{\rm (ii)}] Each cluster monomial in $\mathscr{A}(N)$ is a Laurent polynomial of the initial cluster variables with coefficient in $\Z_{\ge0}$.
\end{enumerate}
\end{theorem}

As we have mentioned in Section~\ref{subsec: HL cluster -}, Hernandez and Leclerc give a cluster algebra structure  $\mathscr{A}$ on the Grothendieck ring of a ``half'' of $\Ca_{A_{N-1}^{(1)}}^0$, denoted by $\Ca_{A^{(1)}_{N-1}}^-$ in~\cite{HL16},
which is associated to the initial quiver $G^-_{A^{(1)}_{N-1}}$ with infinite rank. For example, the quiver $G^-_{A^{(1)}_{3}}$ is given as follows:

$$\raisebox{8em}{\scalebox{0.8}
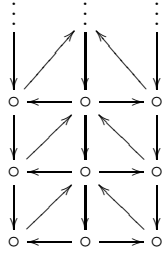
{\xymatrix@C=4ex@R=4ex{
 \vdots\ar[d] &\vdots\ar[d]  &\vdots\ar[d]   & \\
 \circ\ar[ur]\ar[d] & \circ \ar[d]\ar[r]\ar[l]  & \circ\ar[d]\ar[ul] \\
 \circ\ar[ur]\ar[d] & \circ \ar[d]\ar[r]\ar[l]  & \circ\ar[d]\ar[ul] \\
 \circ\ar[ur] & \circ \ar[r]\ar[l] & \circ\ar[ul] }}}
$$
Note that the quiver $G^-_{A^{(1)}_{N-1}}$ satisfies~\eqref{eq: condition B} and does not have frozen vertex either.

The quiver $G^-_{A^{(1)}_{N-1}}$ is mutation equivalent to $\overline{\tQ}_N$ via a sequence of mutation of {\it infinite} length.
To show that, we need to introduce sequences of mutations:
For each $s \in \Z_{\ge 1}$, let ${}_s\Kex_{0}$ (resp.\ ${}_s\Kex_{1}$)
be the vertices $(\ell,m)$ in $\Kex_N$ with $m=s$ and $\ell \equiv 0 \mod 2$ (resp.\ $\ell \equiv 1 \mod 2$).
Note that ${}_s\Kex_{0}$ and ${}_s\Kex_{1}$ are finite subsets of $\Kex_N$.
Let $\Seq^{(s,0)}_{N}$ (resp.\ $\Seq^{(s,1)}_{N}$) be the ascending sequence on $_{s}\Kex_{0}$ (resp.\  $_{s}\Kex_{1}$).
Finally, we set the sequence $\Seq_{N-1}^{{\rm HL}}$ as follows:
\begin{align*}  \Seq^{{\rm HL}}_{N-1} =
\left( \bigl(\Seq_{N}^{(2,1)},\Seq_{N}^{(3,0)}, \cdots \bigr), \bigl(\Seq_{N}^{(3,1)},\Seq_{N}^{(4,0)},\cdots  \bigr), \cdots   \right).
\end{align*}
Then we have
$$  \mu_{\Seq_{N-1}^{{\rm HL}}}\circ \mu_{\Sev_N} (\overline{\tQ}_N) \simeq \mu_{\Seq_{N-1}^{{\rm HL}}}\circ \mu_{\Sod_N} (\overline{\tQ}_N)  \simeq   \mu_{\Seq_{N-1}^{{\rm HL}}}\bigl( (\overline{\tQ}_N)^{{\rm op}} \bigr)  \simeq   G^-_{A^{(1)}_{N-1}},$$
as quivers (see also~\cite[Exercise 2.6.5, Exercise 2.6.6]{FWZ}).

We remark here that
\begin{itemize}
\item each mutation in $\mu_{\Seq_{N-1}^{{\rm HL}}}\circ \mu_{\Sev_N}$ corresponds to the T-system described in~\eqref{eq: T-system A1},
\item for $p \in \Z_{\ge 1}$ with $c(p)=(\ell,m)$ and $\ell<N$, $\mu_{\Seq_{N-1}^{{\rm HL}}}\circ \mu_{\Sev_N}( W^{(\ell)}_{m,(-q)^{2\jj_p-\ell+1}} )$ is a Kirillov-Reshetikhin module. More precisely,
$$\mu_{\Seq_{N-1}^{{\rm HL}}}\circ \mu_{\Sev_N}( W^{(\ell)}_{m,(-q)^{2\jj_p-\ell+1}} ) \simeq \begin{cases}
W^{(\ell)}_{m,(-q)^1} & \text{ if } \ell \equiv 0 \mod 2,\\
W^{(\ell)}_{m,(-q)^{2}} & \text{ if } \ell \equiv 1 \mod 2.
\end{cases}$$
\end{itemize}

\begin{remark}\label{eq: finite sequence2}
As Remark~\ref{eq: finite sequence}, for each $k \in \Z_{\ge 1}$, there exists a finite subsequence of mutations $\mu_\Seq$  of $\mu_{\Seq_{N-1}^{{\rm HL}}}\circ \mu_{\Sev_N}$ such that
\begin{align*}
\mu_{\Seq}( W^{(\ell)}_{m,(-q)^{2\jj_p-\ell+1}} ) & \simeq
W^{(\ell)}_{m,(-q)^{2-\true_\ell}}
\end{align*}
for  any $(\ell,m)\in  I_0 \times \Z_{\ge 1}$ with $\ell+m \le k$,
where $\true_\ell \seteq \true(\ell \equiv 0 \mod 2)$.
\end{remark}

As a corollary of Theorem~\ref{thm apllication A},
we can give a proof of Conjecture~\ref{conj: HL}
for $\Ca^-_{A^{(1)}_{N-1}}$:

\begin{theorem} \label{thm: HL}
The cluster monomials of $\mathscr{A}$ associated with $\g$ of type $A_{N-1}^{(1)}$ in Theorem~\ref{thm: HL16} can be identified with
the real simple modules in $\Ca^-_{A_{N-1}^{(1)}}$.
\end{theorem}

\begin{proof}

By \eqref{eq: bar structure} and \cite{AK,Kas02}, one can check that
\begin{eqnarray} &&
\parbox{85ex}{
\begin{itemize}
\item[{\rm (i)}] for an exact sequence $0 \to C \to A' \otimes A \to B \to 0$ in $\Ca_{A^{(1)}_{N-1}}$, we have an exact sequence
$0 \to \overline{C} \to \overline{A} \otimes \overline{A'} \to \overline{B}\to 0$   in $\Ca_{A^{(1)}_{N-1}}$ and hence
$$0 \to \overline{B} \to \overline{A'} \otimes \overline{A} \to \overline{C}\to 0,$$
\item[{\rm (ii)}] $\overline{W^{(\ell)}_{m,(-q)^{2-\true_\ell }}} \simeq W^{(\ell)}_{m,(-q)^{ \true_\ell-2m }}$.
\end{itemize}
}\label{eq: bar behaviour}
\end{eqnarray}

Note that the cluster variables
$\{ z_{\ell,\true_\ell-2m+1} \}_{(\ell,m) \in I_0 \times \Z_{\ge 1}}$ of the initial seed
$\Seed_{HL}$ of $\mathscr{A}$ are identified with $\{
W^{(\ell)}_{m,(-q)^{\true_\ell-2m+1}} \}_{(\ell,m) \in I_0
\times \Z_{\ge 1}}$ via {\it truncated $q$-characters} (see
\cite[(3), \S 3.2.3]{HL16}). Thus cluster monomials of $\Seed_{HL}$
can be identified with the tensor product of modules in $\left\{
W^{(\ell)}_{m,(-q)^{\true_\ell-2m+1}} \right\}_{(\ell,m)
\in I_0 \times \Z_{\ge 1}}$, which can be obtained from the set of modules
$$\left\{ \overline{W^{(\ell)}_{m,(-q)^{
2-\true_\ell }}} = W^{(\ell)}_{m,(-q)^{ \true_\ell-2m }}
\right\}_{(\ell,m) \in I_0 \times \Z_{\ge 1}}$$ by taking parameter
shift by $-q$.

Thus, by taking the finite sequence of mutations $\mu_\Seq$ in
Remark~\ref{eq: finite sequence2} for $k \gg 0$, \eqref{eq: bar
behaviour} implies that, for $M \in \{ W^{(\ell)}_{m,(-q)^{
2-\true_\ell} } \}_{(\ell,m) \in I_0 \times \Z_{\ge 1}}$ and
a finite sequence $((\ell_i,m_i))_{1 \le i \le r} \in \bl I_0 \times
\Z_{\ge 1}\br^r$, we have
$$ \overline{\mu_{(\ell_r,m_r)} \circ \cdots  \circ \mu_{(\ell_1,m_1)}( M )} \simeq \mu_{(\ell_r,m_r)} \circ \cdots  \circ \mu_{(\ell_1,m_1)}(\overline{M}).$$
Hence Theorem~\ref{thm apllication A} tells that any cluster
monomial of $\mathscr{A}$ can be identified with a real simple in
$\Ca^-_{A_{N-1}^{(1)}}$.
\end{proof}

\subsubsection{Type B} Now $\g$ denotes the affine Kac-Moody algebra of type $B^{(1)}_{n}$ ($n\ge2$) and set $N=2n$.
For each $a \in J=\Z$, define $i_a \in I_0$ as follows (see \cite[\S 2.2]{KKO18}):
$$ i_a \seteq  \begin{cases}
 \ \ n & \text{ if } a \equiv -1,0 \mod N, \\
 \ \ 1 & \text{ otherwise}.
\end{cases}$$

We define a map $X\col J \to \cor^\times$ as follows:
\begin{align*}
& X(0) = q^0, \ \  X(j) =q_\kappa \,q^{2(j-1)} \ (1 \le j \le N-2), \ \  X(N-1) = q^{3N-5}
\end{align*}
where $q_\kappa=(-1)^{n+1}q^{n+1/2}$,
and extend it by
\begin{align} \label{eq: double dual}
X(j+kN)\seteq X(j)q^{k(2N-2)}=X(j)(p^{*})^{2k} \quad \text{ for }0\le j\le N-1 \ \text{ and } \ k\in\Z.
\end{align}

For each $a \in J=\Z$, define $V_a$ as follows:
$$  V_a \seteq V(\varpi_{i_a})_{X(a)}.$$

By~\eqref{eq:denom B^(1) le n} and~\eqref{eq:denom B^(1) n}, the quiver $\Gamma^J$ in~\eqref{eq: GammaJ} is of type $\Ainf$.
Furthermore,~\eqref{eq: double dual} tells that the family of quasi-good modules $\{ V_a \}$ satisfies ${\rm (iii)}$ in~\eqref{eq: general assumption}.

\begin{proposition}[{\cite[Proposition 2.4.2]{KKO18},\cite[\S 4]{Oh15}}] \hfill
\bnum
\item The family of quasi-good modules $\{ V_a \}_{a \in J}$ satisfies the conditions in~\eqref{eq: general assumption}.
\item The subcategory $\Ca_J$ coincides with the category $\Ca^{0}_{\g}$ in~\eqref{eq:S0B^(1)_n}.
\end{enumerate}
\end{proposition}

Then we have an exact functor
$$ \F \seteq \soplus_{\be \in \rl^+_J} R^J(\be)\gmod \to \Ca^{0}_{\g}$$
which factors thorough $\T_N$ via $\tFF\col \T_N\to \Ca^{0}_{\g}$.

\begin{proposition}[{\cite[Proposition 3.2.1]{KKO18}}] \label{prop: image B}
Let $0\le a \le N-1$ and $1\le b-a+1\le N-1$.
 Then we have
\begin{enumerate}[{\rm (i)}]
\item $\F(L[0,b]) \simeq V(\varpi_n)_{q^{2b}}$  for   $0 \le b \le N-2$.
\item $\F(L[a,N-1]) \simeq V(\varpi_n)_{q^{2a+N-3}}$ for $1 \le a \le N-1$.
\item $ \F(L[a,b]) \simeq  \begin{cases}
V(\varpi_{b-a+1})_{(-1)^{b-a}q_{\kappa} q^{a+b-2}} & \text{for} \ 1 \le a \le b \le N-2, \ b-a+1 < n, \\
V(\varpi_n)_{q^{2b}} \hconv V(\varpi_n)_{q^{2a+N-3}} & \text{for} \ 1 \le a \le b \le N-2, \ b-a+1 \ge  n, \\
V(\varpi_n)_{q^{2a+N-3}} \hconv V(\varpi_n)_{q^{2b-2}} & \text{for} \  1\le a\le N-1<b, \ b-a+1 \le  n, \\
V(\varpi_{N-b+a-1})_{(-1)^{b-a}q_{\kappa} q^{a+b-3}} & \text{for} \ 1\le a\le N-1<b,  \  b-a+1 > n.
\end{cases}  $
\end{enumerate}
\end{proposition}

In this case, the exact sequence~\eqref{eq: T-system in qHecke} in $R^J\gmod$ is not translated into the known $T$-system of $U_q'(B_n^{(1)})$ in~\cite{Her06,KNS,Nak}.

\smallskip

Now we have a $U_q'(B_n^{(1)})$-version of Theorem~\ref{thm apllication A} as follows:
\begin{theorem} \label{thm apllication B}
The category $\Ca^0_{\g}$ gives a monoidal categorification of the cluster algebra $\mathscr{A}(N)$.
In particular, we have an isomorphism $(t=1,2)$ of cluster algebras
$$\phi_t \colon K\bigl(\Ca^{0}_{A_{2n-1}^{(t)}}\bigr) \isoto K\bigl(\Ca^{0}_{B_n^{(1)}}\bigr).$$
\end{theorem}

\subsubsection{Type C}  Let $\g$ be the affine Kac-Moody algebra of type $C^{(1)}_{n}$ $(n \ge 3)$ and set $N=n+1$.
Recall the denominator formulas $d_{k,l}(z)$ for $U_q'(C^{(1)}_{n})$.

\begin{theorem}[{\cite{AK}}] \label{thm: denominators C}
For $\g=C^{(1)}_{n}$ $(n \ge 3)$, we have
\begin{align} \label{eq:denom C^(1)_n}
d_{k,l}(z)=  \displaystyle  \prod_{i=1}^{ \min(k,l,n-k,n-l)}
\big(z-(-q_s)^{|k-l|+2i}\big) \prod_{i=1}^{ \min(k,l)} \big(z-(-q_s)^{2n+2-k-l+2i}\big),
\end{align}
where $q_s \seteq q^{1/2}$.
\end{theorem}

For each $a \in J=\Z$, define $i_a \in I_0$ as follows:
$$ i_a \seteq  \begin{cases}
 \ \ n & \text{ if } a \equiv 0 \mod N, \\
 \ \ 1 & \text{ otherwise}.
\end{cases}$$

Note that $p^*=q_s^{2n+2}$.  Set
$$  q_{\kappa} \seteq  (-q_s)^{n+3}.$$

We define a map $X\col J \to \cor^\times$ as follows:
\begin{align*}
& X(0) = q_s^0=1, \ \  X(j) = q_{\kappa}q_s^{2(j-1)}  \ (1 \le j \le N-1),
\end{align*}
and extend it by
\begin{align} \label{eq: double dual C}
X(j+kN)\seteq X(j){p^{*}}^{2k} \quad \text{ for }0\le j\le N-1 \ \text{ and } \ k\in\Z.
\end{align}

For each $a \in J=\Z$, define $V_a$ as follows:
$$  V_a \seteq V(\varpi_{i_a})_{X(a)}.$$

\begin{proposition}[{\cite[Proposition C.2]{AK}}] \label{prop: Cn1 morphism} For each $i,j \in I_0$ with $i+j \le n$, there exists a $U'_q(C_n^{(1)})$-homomorphism given as follows:
\begin{align}\label{eq: Cn1 homo}
V(\varpi_i)_{(-q_s)^{-j}} \tens V(\varpi_j)_{(-q_s)^{i}} \twoheadrightarrow V(\varpi_{i+j}).
\end{align}
\end{proposition}
In particular, from \eqref{eq: Cn1 homo}, we have a $U'_q(C_n^{(1)})$-homomorphism as follows:
\begin{align*}
V(\varpi_{n-u+1})_{(-q_s)^{-1}} \tens V(\varpi_1)_{(-q_s)^{n+u+1}} \twoheadrightarrow V(\varpi_{n-u})
\end{align*}
for $1 \le u \le n-1$.

By~\eqref{eq:denom C^(1)_n}, the quiver $\Gamma^J$ in~\eqref{eq: GammaJ} is of type $\Ainf$. Then we have an exact functor
$$ \F \col \soplus_{\be \in \rl^+_J} R^J(\be)\gmod \to \Ca_J \subset \Ca_\g^0.$$

\begin{proposition} \label{prop:segmentimage C step1}
Let $0\le a \le N-1$ and $b-a+1\le N-1$. Then we have
\begin{enumerate}[{\rm (i)}]
\item $\F(L[0,b]) \simeq V(\varpi_{n-b})_{(-q_s)^{b}}$  for $0 \le b \le N-1$.

\item $\F(L[a,b]) \simeq V(\varpi_{b-a+1})_{q_{\kappa}(-q_s)^{a+b-2}}$ for $1 \le a \le b \le N-1$.
\end{enumerate}
Here we understand $V(\varpi_{0})$ and $V(\varpi_{N})$ as the trivial module $\cor$.
\end{proposition}

\begin{proof}
{\rm (i)} When $b=0$, it is obvious by Proposition~\ref{prop:image of tau} {\rm (a)}. By induction on $b \le N-1$, we may assume that
$\F(L[0,b-1]) \simeq V(\varpi_{n-b+1})_{(-q_s)^{b-1}}$.

On the other hand, Theorem~\ref{thm: denominators C} and Proposition~\ref{prop: Cn1 morphism} imply that
\begin{align*}
&\F(L[0,b-1]) \hconv \F(L(b))  \\
& \hspace{5ex}
\simeq\begin{cases}
V(\varpi_{n-b+1})_{(-q_s)^{b-1}} \hconv V(\varpi_1)_{(-q_s)^{n+3+2(b-1)}} \simeq V(\varpi_{n-b})_{(-q_s)^{b}}  & \text{ if } b \le N-2, \\
V(\varpi_{1})_{(-q_s)^{n-1}} \hconv V(\varpi_1)_{(-q_s)^{3n+1}} \simeq \cor & \text{ if } b = N-1.
\end{cases}
\end{align*}
By Lemma~\ref{lem: non-simple tensor}, we have
$\F(L[0,b])\simeq\F(L[0,b-1]) \hconv \F(L(b))$. Hence, we obtain the desired result.

\smallskip\noi
{\rm (ii)} When $a=N-1$, it is obvious by Proposition~\ref{prop:image of tau} {\rm (a)}.  By descending induction on $a$, we may assume that
$\F(L[a+1,N-1]) \simeq V(\varpi_{N-a-1})_{q_{\kappa}(-q_s)^{N-2+a}}$.

On the other hand, Theorem~\ref{thm: denominators C} and Proposition~\ref{prop: Cn1 morphism} imply that
\begin{align*}
&\F(L(a)) \hconv \F(L[a+1,N-1])   \simeq  \\
& \hspace{5ex} \begin{cases}
V(\varpi_1)_{q_{\kappa}q_s^{2(a-1)}} \hconv  V(\varpi_{N-a-1})_{q_{\kappa}(-q_s)^{N-2+a}} \simeq  V(\varpi_{N-a})_{q_{\kappa}(-q_s)^{N-3+a}} & \text{ if } 0 < a \le N-2, \\
V(\varpi_n) \hconv V(\varpi_{n})_{q_{\kappa}(-q_s)^{n-1}}  \simeq \cor & \text{ if } a = 0.
\end{cases}
\end{align*}
By Lemma~\ref{lem: non-simple tensor} again, we have
$\F(L[a,N-1])\simeq\F(L(a)) \hconv \F(L[a+1,N-1])$
and we obtain the desired result.

\smallskip\noi
{\rm (iii)} By fixing $a$ or $b$, one can apply the same argument as in {\rm (i)} or {\rm (ii)}.
\end{proof}

\begin{proposition} \label{prop: general assumption C}
The family of quasi-good modules $\{ V_a \}_{a \in J}$ satisfies
condition~\eqref{eq: general assumption}.
\end{proposition}

\begin{proof}
By Proposition~\ref{prop:segmentimage C step1}, condition~\eqref{eq: general assumption}~\eqref{it:b}
for $a=0$ is guaranteed. Note that for a segment $[a,b]$ with $b-a+1=N$,
$[a,b]$ is one of the following forms:
$$
\begin{cases}
[kN,(k+1)N-1], \\
[kN+t,(k+1)N-1] \cup [(k+1)N,(k+1)N+t-1].
\end{cases}
$$
for some $k \in \Z$ and $1 \le t \le N-1$.
By~\eqref{eq: double dual C}, one can conclude that
$$
\begin{cases}
\F(L[kN,(k+1)N-1]) \simeq \cor, \\[1ex]
\F(L[kN+t,(k+1)N-1] \hconv F(L[(k+1)N,(k+1)N+t-1]) \\
\hs{24ex} \simeq  V(\varpi_{N-t})_{q_{\kappa}(-q_s)^{N-3+t}{p^{*}}^{2k}} \hconv  V(\varpi_{N-t})_{(-q_s)^{t-1}{p^{*}}^{2k+2}}  \simeq \cor,
\end{cases}
$$
which implies~\eqref{eq: general assumption} \eqref{it:1}. Using similar argument to a segment $[a,b]$ with $b-a+1 \le N$, one can prove
~\eqref{eq: general assumption} \eqref{it:2}. Hence our assertion follows.
\end{proof}

By Theorem~\ref{thm: Factoring} and Proposition~\ref{prop: general assumption C}, the functor $\F$ factors thorough $\T_N$.

\begin{remark}
The subcategory $\Ca_J$ is a proper subcategory of $\Ca_\g^0$.  For example, the quasi-good module
$V(\varpi_1)_{(-q_s)^{-2}}$ is not contained in $\Ca_J$, since
$V(\varpi_1)_{(-q_s)^{-2}}$  does not strongly
commute with infinitely many $\F(\Wlmj{1}{m}{\jj_p})$'s.
\end{remark}

\begin{theorem} \label{thm apllication C}
The category $\Ca_{J}$ gives a monoidal categorification of the cluster algebra $\mathscr{A}(N)$.
\end{theorem}

\subsubsection{Type D}  Let $\g^{(t)}$ be the affine the Kac-Moody algebra of type $D^{(t)}_{n}$ $(n \ge 4$ if $t=1,2$, and $n=4$ if $t=3$) and set $N=n$.
Let us denote by $V^{(t)}(\varpi_i)$ the fundamental module over $U_q'(D_{n}^{(t)})$ $(t=1,2,3)$. The denominator formulas for these types were calculated in~\cite{KKK15B,Oh15,OT18}.

Note that
$$ V^{(t)}(\varpi_i)_x \simeq V^{(t)}(\varpi_i)_{y} \Longleftrightarrow x^t=y^t  \quad
\begin{cases}
\text{if $t=1$ and $1\le i\le n$},\\
\text{if $t=2$ and $1\le i\le n-2$},\\
\text{if $t=3$ and $i=2$.}
\end{cases}
$$

For each $a \in J=\Z$, define $i_a \in I_0$ as follows:
$$ i_a \seteq  \begin{cases}
 \ \ n+1-t & \text{ if } t=1,2  \text{ and }  a \equiv 0,1 \mod N, \\
 \ \ 1 & \text{ otherwise}.
\end{cases}$$

Set
$$ q_{\kappa} \seteq \begin{cases}
(-q)^n  & \text{ if } t=1, \\
(\sqrt{-1})^{n-1} q^n & \text{ if } t=2,  \\
\omega q^4& \text{ if } t=3,
\end{cases}
$$
where $\omega$ is  a primitive third root of unity.

We define a map $X\col J \to \cor^\times$ as follows:
\begin{align*}
& X(0) =   q^0, \ \ X(1) = q^2  \text{ and }  X(j) = q_{\kappa}q^{2(j-1)}  \ \ (2 \le j \le N-1),
\end{align*}
and extend it by
\begin{align*} 
X(j+kN)\seteq X(j){p^{*}}^{2k} \quad \text{ for }0\le j\le N-1 \ \text{ and } \ k\in\Z.
\end{align*}

For each $a \in J=\Z$, define $V^{(t)}_a$ as follows:
$$  V^{(t)}_a \seteq V^{(t)}(\varpi_{i_a})_{X(a)}.$$

Now we shall present the denominator formulas only related to the choice of the family of distinct quasi-good modules $\{ V^{(t)}_a \} \colon$

\begin{theorem}[{\cite{KKK15B,KMOY07,Oh15,OT18}}] \label{thm: denominators D}
\begin{enumerate}[{\rm (a)}]
\item For $t=1$, we have
$$d_{1,1}(z)=(z-q^2)(z-(-q)^{2n-2}), \  d_{1,n}(z)=(z-(-q)^{n}), \  d_{n,n}(z)=\displaystyle \prod_{s=1}^{\lfloor \frac{n}{2} \rfloor} \big(z-(-q)^{4s-2}\big).$$
\item For $t=2$, we have
$$d_{1,1}(z)\hspace{-.3ex}=\hspace{-.3ex}(z^2-q^4)(z^2-(-q^2)^{2n-2}), \  d_{1,n-1}(z)\hspace{-.3ex}=\hspace{-.3ex}(z^2+(-q^2)^{n}), \  d_{n-1,n-1}(z)\hspace{-.3ex}=\hspace{-.3ex}\displaystyle \prod_{s=1}^{n-1} \big(z+(-q^2)^{s}\big).$$
\item For $t=3$, we have
$$d_{1,1}(z)  = (z-q^{2})(z-q^{6})(z-\omega q^{4})(z-\omega^2 q^{4}).$$
\end{enumerate}
\end{theorem}

Then one can check the following holds:
\begin{corollary}
The graph $\Gamma^J$ associated to $\{ V^{(t)}_a \}$ is of type $\Ainf$.
\end{corollary}

By the morphisms in~\cite{KKK15B,Oh15,OT18},  we have the following propositions
by applying the same arguments as in Proposition~\ref{prop:segmentimage C step1} and Proposition~\ref{prop: general assumption C}:

\begin{proposition} \label{prop:segmentimage D}
Let $0\le a \le N-1$ and $b-a+1\le N-1$. Then we have
\begin{enumerate}[{\rm (i)}]
\item $\F^{(t)}(L[0,b]) \simeq V^{(t)}(\varpi_{N+1-t-b-\delta})_{q^{b}}$  for $0 \le b \le N-2$, \\
where
$\delta =
\begin{cases}
1 & \text{if $b \ne 0$ and $t=1$},\\
0 & \text{ otherwise.}
\end{cases}
$

In particular, $\F^{(t)}(L[0,N-1]) \simeq \cor$.
\item $\F^{(t)}(L[1,b]) \simeq
\begin{cases}
V^{(1)}(\varpi_{n-\ep})_{q^{2(b-1)}} & \text{ if $t=1$},\\
V^{(2)}(\varpi_{n-1})_{(-1)^\ep q^{2(b-1)}} & \text{ if  $t=2$},\\
V^{(3)}(\varpi_{1})_{\omega^\ep q^{2(b-1)}} & \text{ if  $t=3$},
\end{cases}
\quad \text{ for } b \le N-1,
$ \\
where $\ep \in \{ 0,1\}$ such that $b-1 \equiv \ep \mod 2$.
\item $\F^{(t)}(L[a,b]) \simeq V(\varpi_{b-a+1})_{q_{\kappa}(-q)^{a+b-2}}$ for $2 \le a \le b \le N-1$.
\end{enumerate}
\end{proposition}

\begin{proposition} The family of distinct quasi-good modules $\{ V_a \}_{a \in J}$ satisfies the conditions in~\eqref{eq: general assumption}.
\end{proposition}

Now we have functors $\F^{(t)}$ and $\tFF^{(t)}$ onto $\Ca_J^{(t)}$ as in the previous subsections.

\begin{theorem} \label{thm apllication D}
The category $\Ca^{(t)}_{J}$ gives a monoidal categorification of the cluster algebra $\mathscr{A}(N)$.
\end{theorem}

\begin{remark}
The subcategory $\Ca^{(t)}_{J}$ is a proper subcategory of  $\Ca_{\g^{(t)}}^0$.
For example, the quasi-good module
$V(\varpi_1)_{q_\kappa}$ is not contained in $\Ca_J$ of type $D_n^{(1)}$
since $V(\varpi_1)_{q_\kappa}$  does not strongly
commute with infinitely many $\F(\Wlmj{1}{m}{\jj_p})$'s.
\end{remark}

\end{document}